\documentclass[10pt]{article}

\usepackage[english]{babel}
\usepackage{ae}
\usepackage{graphicx}
\usepackage{tikz}
\usetikzlibrary{arrows}
\usetikzlibrary{matrix}
\usepackage{amsmath, amsfonts, amssymb, mathtools, amsthm}
\usepackage{url}
\usepackage{authblk}
\usepackage{nicefrac}

\newtheorem{lemma}{Lemma}
\newtheorem{example}{Example}
\newtheorem{remark}{Remark}
\newtheorem{proposition}{Proposition}

\providecommand{\abs}[1]{\lvert#1\rvert}

\providecommand{\norm}[1]{\lVert#1\rVert}

\newcommand{\num}[1]{\mathbb{#1}}
\newcommand{\Reals}{\num{R}}
\newcommand{\cali}[1]{\mathcal{#1}}

\def\vec#1{\mathbf{#1}}

\begin{document}

\title{Optical Flow on Evolving Surfaces with Space and Time Regularisation}

\author[1]{Clemens Kirisits}
\author[1]{Lukas F. Lang}
\author[1,2]{Otmar Scherzer}
\affil[1]{\footnotesize Computational Science Center, University of Vienna, Oskar-Morgenstern-Platz\ 1, 1090 Vienna, Austria}
\affil[2]{Radon Institute of Computational and Applied Mathematics, Austrian Academy of Sciences, Altenberger Str.\ 69, 4040 Linz, Austria}

\maketitle

\begin{abstract}
\noindent
We extend the concept of optical flow with spatiotemporal regularisation to a dynamic non-Euclidean setting. Optical flow is traditionally computed from a sequence of flat images. The purpose of this paper is to introduce variational motion estimation for images that are defined on an evolving surface. Volumetric microscopy images depicting a live zebrafish embryo serve as both biological motivation and test data.
\end{abstract}


\section{Introduction}
\subsection{Motivation}\label{sec:motivation}
Advances in laser-scanning microscopy and fluorescent protein technology have increased resolution of microscopy imaging up to a single cell level~\cite{MegFra03}. 
They allow for four-dimensional (volumetric time-lapse) imaging of living organisms and shed light on cellular processes during early embryonic development. 
Understanding cellular processes often requires estimation and analysis of cell motion. However, the amount of data that is recorded is tremendous and therefore in many cases automated image analysis is necessary.

The specific biological motivation for this work is to understand the motion and division behaviour of fluorescently labelled endodermal cells of a zebrafish embryo. Although of considerable importance for developmental biology, knowledge about the migration patterns of these cells is scarce \cite{SchmShaScheWebThi13}. The dataset under consideration consists of volumetric time-lapse images taken by a laser-scanning microscope. The recorded sequence depicts a cuboid section $S \subset \mathbb{R}^3$ of said zebrafish embryo, whose endodermal cells express a fluorescent protein. We model this sequence by a scalar function
\begin{equation*}
	\bar F : [0,T] \times S \to \mathbb{R}
\end{equation*}
that assigns to every pair $(t,x) \in [0,T] \times S$ a nonnegative value $\bar F (t,x)$ proportional to the fluorescence response of point $x$ at time $t$.

Optical flow methods are used regularly to estimate cellular motion, see Sec.~\ref{sec:related}. Applying them directly to our data $\bar F$ to obtain a dense 3D velocity field
\begin{equation*}
	\vec{m} : [0,T] \times S \to \mathbb{R}^3
\end{equation*}
is possible but problematic from a computational point of view \cite{AmaMyeKel13}, even more so if temporal regularisation is to be included. We propose a solution to this by adapting our model according to biological facts about the nature of the marked cells.

Endodermal cells develop on the surface of the embryo's yolk, where they form a non-contiguous monolayer \cite{WarNus99}. Loosely speaking, they only sit next to each other but not on top of each other. Moreover, the yolk's shape is roughly spherical and deforms over time. This means that the yolk's surface can be modelled by an embedded two-dimensional manifold $\mathcal{M}_t\subset\mathbb{R}^3$, the subscript indicating dependence on time. In practice, $\mathcal{M}_t$ can be approximated by fitting piecewise polynomials, for instance, to the cell centres.\footnote{Sometimes it is possible to already capture the yolk's surface with the microscope in a second sequence of images. We do not, however, use such additional data in this article.} Consequently it is possible to reduce the data dimension by only considering the restriction $F$ of $\bar F$ to this moving surface; see Fig.~\ref{fig:embryo}. More details on the acquisition and preprocessing of the microscopy data are given in Sec.~\ref{sec:parametrisation}. This dimension reduction, in turn, necessitates the development of an optical flow model for data defined on an evolving surface, which is the main contribution of this article.

Let $t_0$ be a fixed instant of time and $x_0 \in \mathcal{M}_{t_0}$. Assume a cell located at $x_0$, indicated by a relatively high value of $F(t_0,x_0)$, moves with velocity $\vec{m}(t_0,x_0)$. On the other hand, suppose the yolk's surface has velocity $\vec{V}(t_0,x_0)$. The purely tangential vector
\begin{equation}\label{eq:decomp}
	\vec{u}(t_0,x_0) = \vec{m}(t_0,x_0) - \vec{V}(t_0,x_0)
\end{equation}
describes the cell's velocity relative to $\vec{V}$. Put differently, the total observed velocity $\vec{m}$ of a cell is the sum of the surface velocity $\vec{V}$ and the cell's tangential velocity $\vec{u}$. Compare Fig.~\ref{fig:sketch}. While the former is a quantity extrinsic to the surface the latter is intrinsic. A motion estimation method dealing with the full 4D dataset $\bar F$ would directly try to calculate $\vec{m}$ for all $(t,x) \in [0,T] \times S$. The method proposed in this article, however, only computes the tangential field $\vec{u}$ for a given surface velocity $\vec{V}$. The total velocity can then be recovered by adding the two vector fields.

\begin{figure}
	\begin{center}
	\begin{tikzpicture}
		\draw [thick, gray] (-3,0) to [out=30,in=150] (3,0);
		\draw [thick, gray] (-3,1) to [out=40,in=150] (4,1.8);
		\node [right] at (3,0) {$\mathcal{M}_{t_0}$};
		\node [below] at (4,1.5) {$\mathcal{M}_{t_0+\Delta t}$};
		\draw [fill=black] (0,.9) ellipse [x radius=0.3, y radius=0.15];
		\draw [fill=black] (2.5,2.4) ellipse [x radius=0.3, y radius=0.15];
		\draw [-stealth', thick] (0,.9) -- (.93*1.1,.93*2.6);
		\draw [-stealth', thick] (0,.9) -- (1.5,.9);
		\draw [-stealth', thick] (0,.9) -- (.94*2.5,.94*2.4);				
		\node [left] at (.55,1.75) {$\vec{V}$};
		\node [right] at (1.5,.9) {$\vec{u}$};
		\node [right] at (1.7,1.7) {$\vec{m}$};
		\draw [-stealth', dotted, thick] (0,0) to [out=90,in=225] (0,.9) to [out=45,in=270] (1.3,1.7) to [out=90,in=225] (2.5,2.4) to [out=45,in=225] (1.25*2.5,1.25*2.4);
		\node [right] at (1.3*2.5,1.3*2.4) {$\gamma$};
	\end{tikzpicture}
	\end{center}
	\caption{Sketch of a cell (indicated by a black ellipse) moving along a trajectory $\gamma$ on a moving surface. The cell's velocity is given by $\partial_t \gamma = \vec{m}$, which can be decomposed into surface velocity $\vec{V}$ and relative tangential motion $\vec{u}$.}
	\label{fig:sketch}
\end{figure}
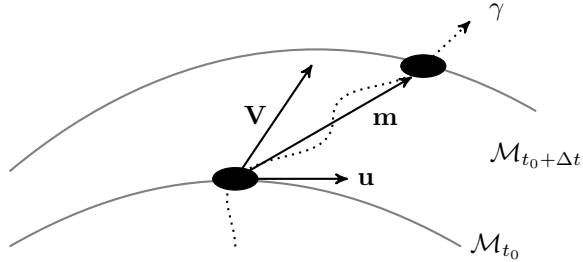

In practice the true velocity of a moving surface might not be known and might even be impossible to determine from available data. This is also the case for the microscopy data considered in this paper. Our solution consists in picking one surface velocity $\vec{V}$ that is consistent with $\mathcal{M}_t$, of which there are infinitely many in general, and to estimate the tangent field $\vec{u}$ relative to this chosen surface velocity. While the resulting $\vec{u}$ must be interpreted with care, it is reasonable to assume that the sum $\vec{u} + \vec{V}$ is close to the true total velocity $\vec{m}$. The selected surface velocity ideally strikes a balance between being easy to implement while being not too unnatural. While modelling the optical flow on an evolving surface is the main novelty of this article, from the viewpoint of our particular application, it can be regarded as a subproblem making the computation of 3D velocities feasible, namely by reducing the data dimension while keeping as much accuracy as possible.

\begin{figure*}
	\centering
	\includegraphics[width=0.32\textwidth]{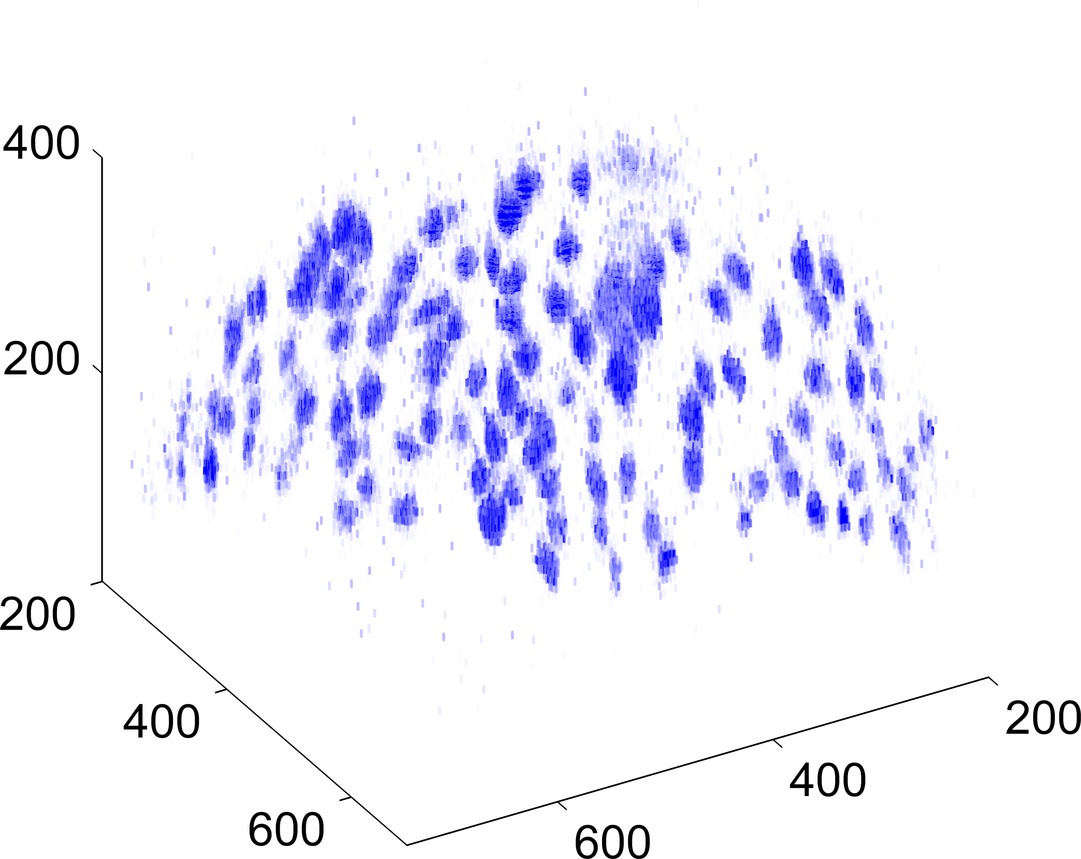}
	\hfill
	\includegraphics[width=0.32\textwidth]{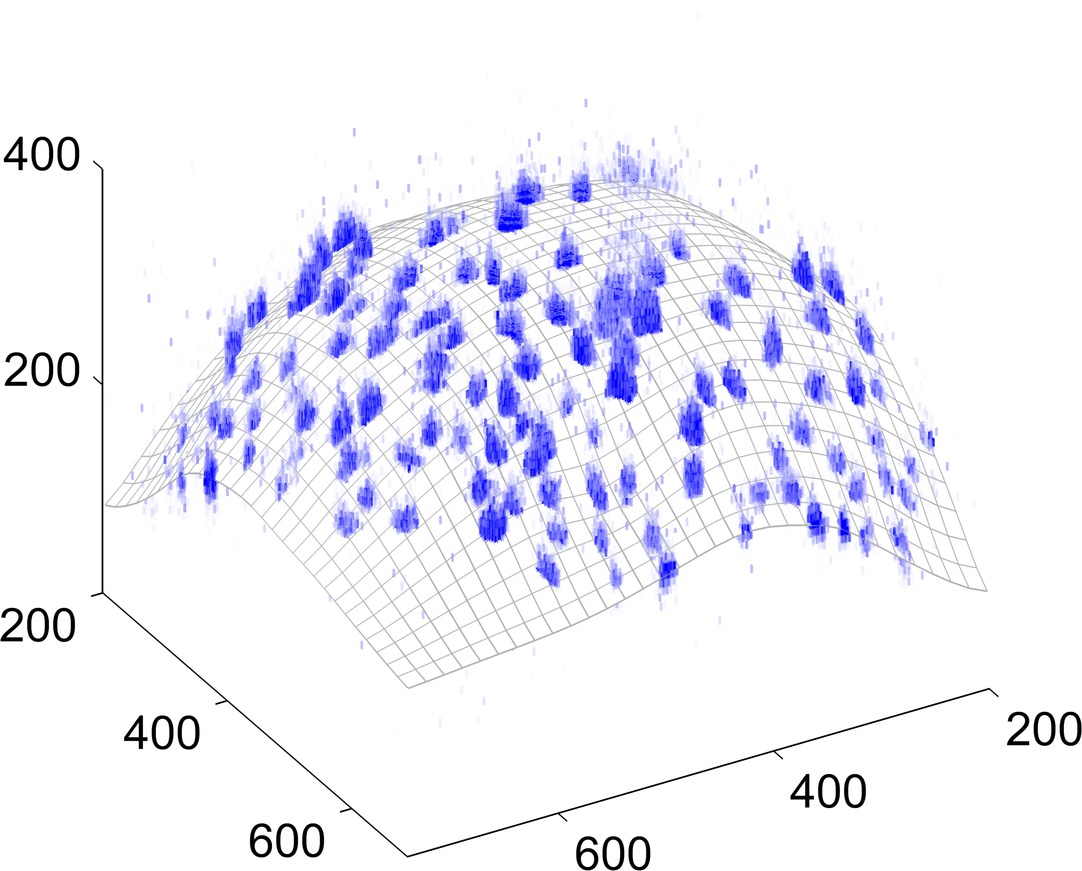}
	\hfill
	\includegraphics[width=0.32\textwidth]{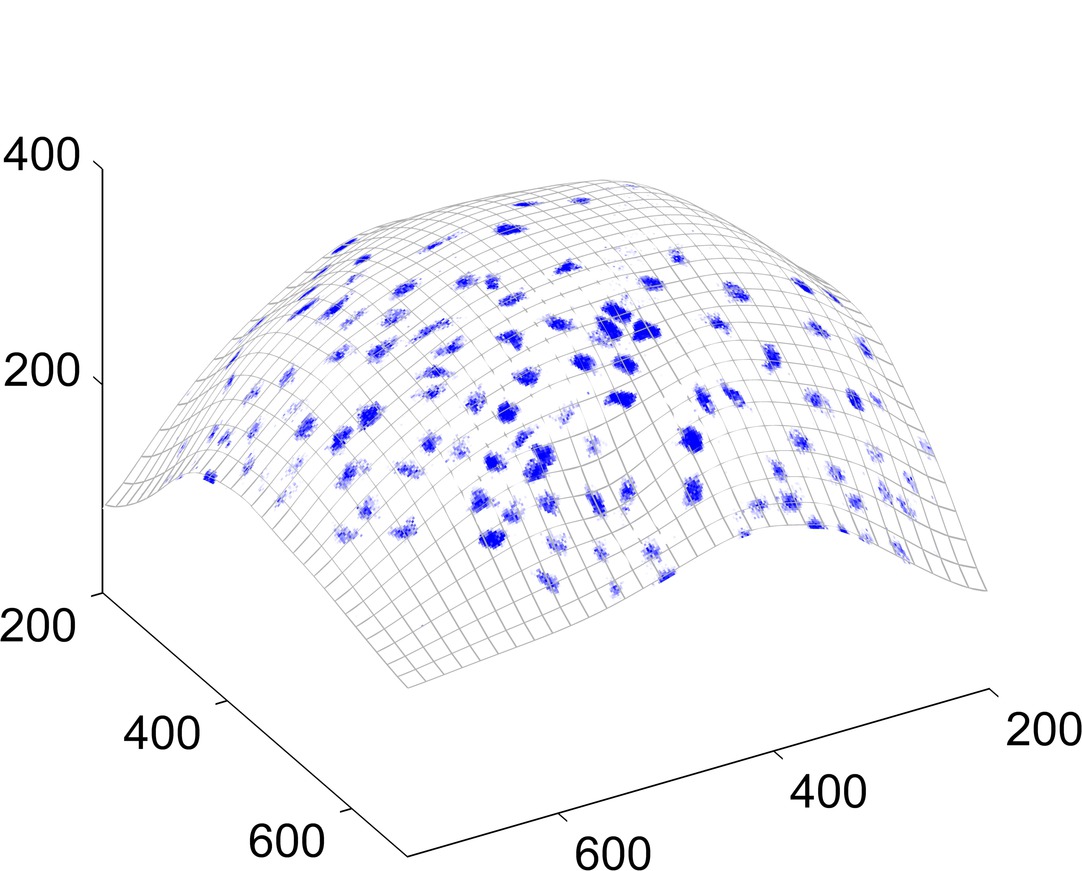}
	\\
	\includegraphics[width=0.32\textwidth]{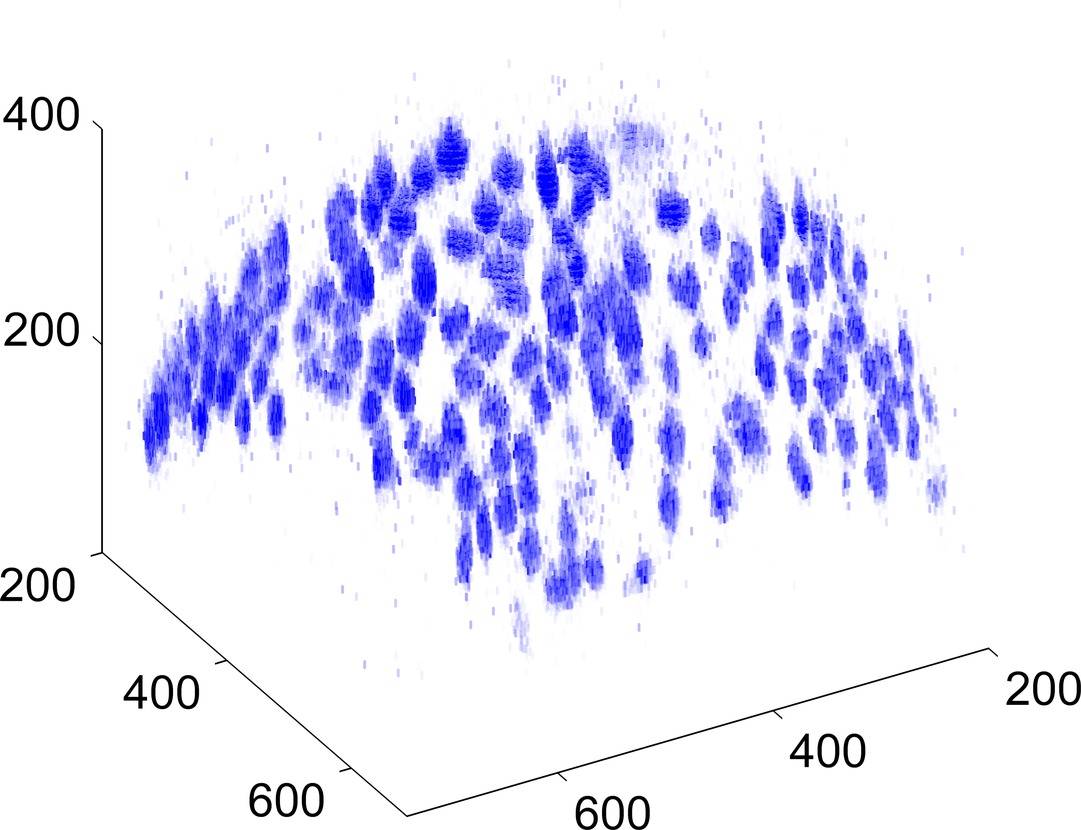}
	\hfill
	\includegraphics[width=0.32\textwidth]{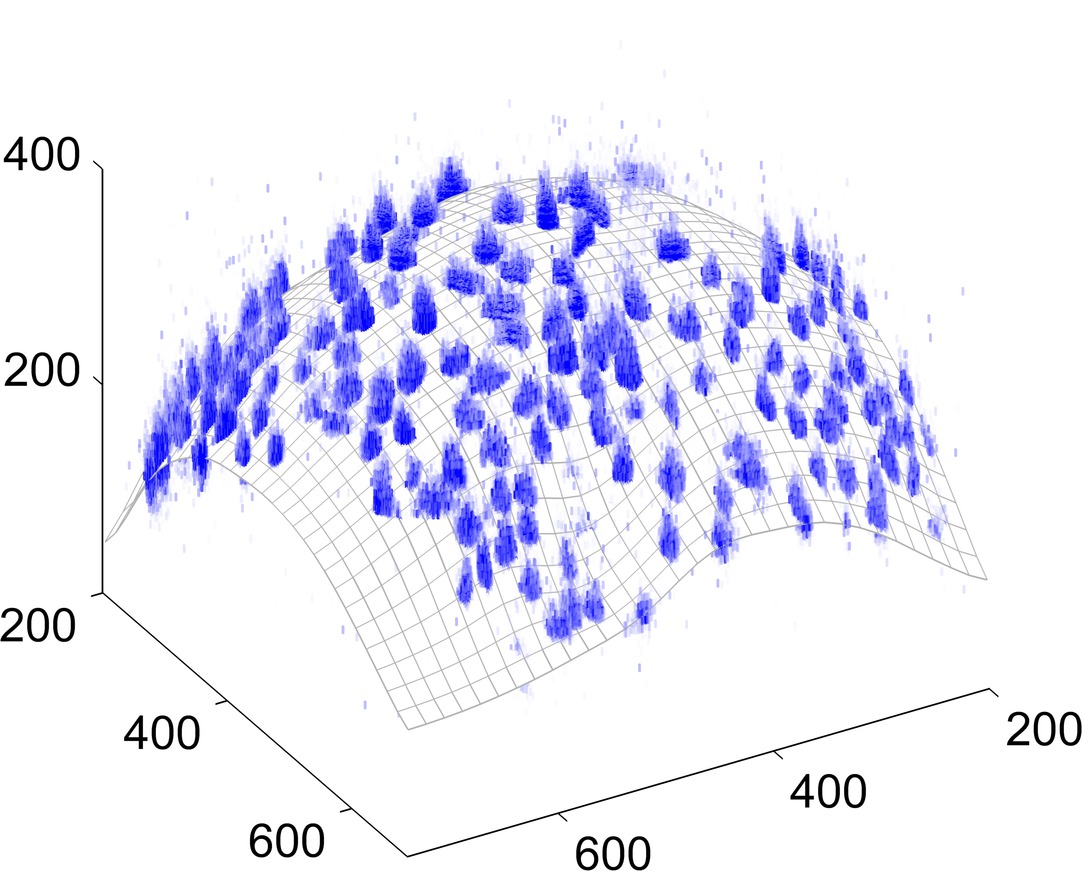}
	\hfill
	\includegraphics[width=0.32\textwidth]{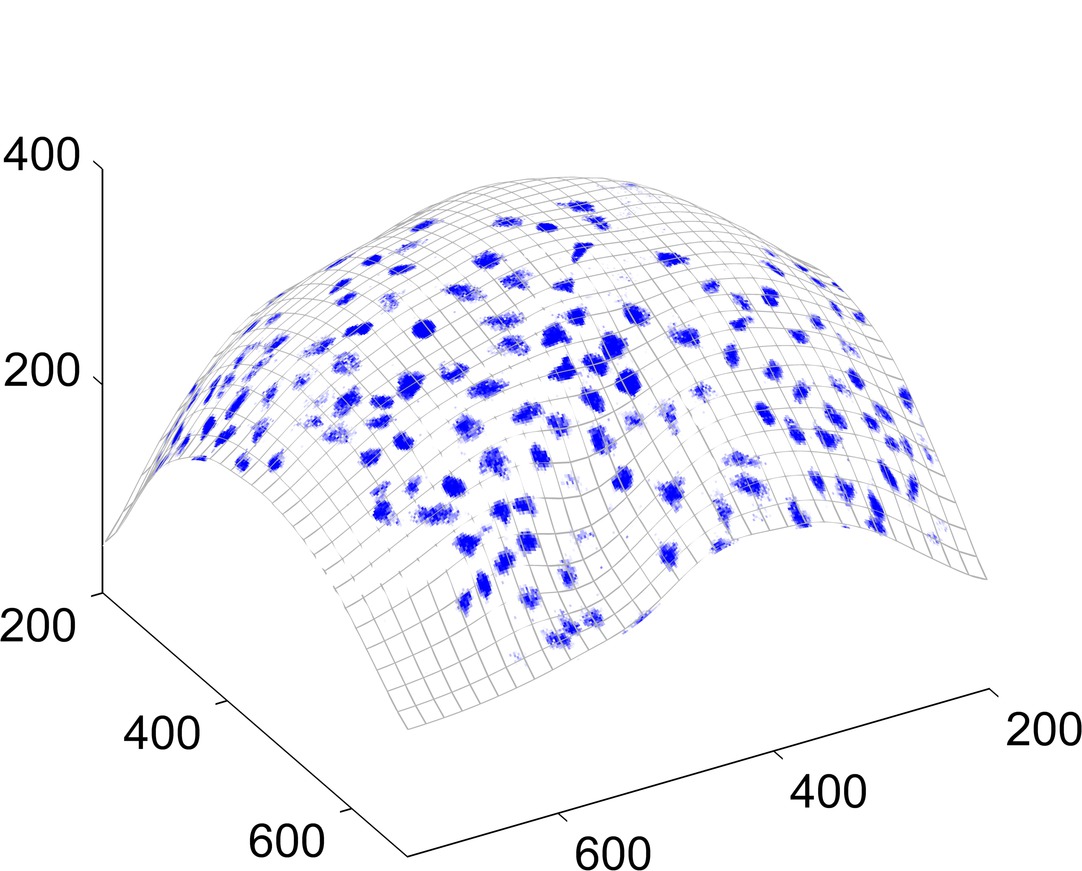}
	\caption{Frame no.\ 50 (top row) and 61 (bottom row) of the embryonic zebrafish image sequence. The left images illustrate the raw volumetric data $\bar F$. Intensity corresponds to fluorescence response. In the middle images, the curved mesh represents surfaces fitted to the cell's centres. The right images depict only the surface and the extracted two-dimensional image $F$. All dimensions are in micrometer ($\mu$m). For more details on the microscopy data and the preprocessing steps see Sec.~\ref{sec:exp}.}
	\label{fig:embryo}
\end{figure*}

\subsection{Contribution}

The contributions of this article are as follows. First, we formulate the optical flow problem on an evolving two-dimensional manifold and derive a generalised optical flow constraint. Second, we translate the classical functional by Horn and Schunck~\cite{HorSchu81} and its spatiotemporal extension by Weickert and Schn\"{o}rr~\cite{WeiSchn01b} to the setting of moving manifolds. The associated Euler-Lagrange equations are solved with a finite difference scheme requiring a global parametrisation of the moving manifold. Finally, we apply this technique to obtain qualitative results from the aforementioned zebrafish data. Our experiments show that the optical flow is an appropriate tool for analysing these data. It is capable of visualising global trends as well as individual cell movements. In particular, the computed flow field can indicate cell divisions, while its integral curves approximate cell trajectories.

Finally, we address a point raised in the recent publication by Schmid et al.~\cite{SchmShaScheWebThi13}, who also analysed endodermal cell dynamics in a zebrafish embryo. They approximated the surface by a sphere, used different map projections to reduce the amount of data by one dimension, and subsequently computed cell motion in the plane. They acknowledge, however, the need for more exact, and supposedly slower, imaging techniques that do not discard any 3D information. While our approach still requires the volume data to be projected onto a surface and thus is faster than comparable 3D approaches, it does not require the surface to be very simple --- e.g.~spherical or planar --- or static.

This article is structured as follows. In the next subsection we review related literature. Section \ref{sec:preliminaries} is devoted to providing the necessary mathematical background, notations, and definitions. Sections \ref{sec:model} and \ref{sec:eulerlagrange} introduce our variational model of optical flow on evolving surfaces and contain the continuous and discretised optimality conditions, respectively. In Sec.~\ref{sec:exp} we explain our microscopy data and the necessary preprocessing steps, summarise our approach, and finally present numerical results.

\subsection{Related work}\label{sec:related}
Optical flow is the apparent motion in a sequence of images. Its estimation is a key problem in Computer Vision. Horn and Schunck~\cite{HorSchu81} were the first to propose a variational approach assuming constant brightness of moving points and spatial smoothness of the velocity field. Since then, a vast number of modifications have been developed. See \cite{BakSchaLewRotBla11,WeiBruBroPap06} for recent surveys.

Using optical flow to extract motion information from cell biological data has gained popularity over the last decade. See, for example \cite{AbrVie02,AmaMyeKel13,BuiYuNizSil10,DelJarScheRamRui12,HubUlmMat07,MelCamLomRizVer07,Miu05,QueMenCam10,SchmShaScheWebThi13}. In  these works displacement fields are computed either from 3D images or from 2D projections of the 3D data. While projections can suffer from inaccuracies \cite{QueMenCam10,SchmShaScheWebThi13}, the extraction of dense velocities from volumetric time-lapse data poses computational challenges \cite{AmaMyeKel13}. In the present article we avoid both of these problems.

Many natural scenarios are more accurately described by a velocity field on a non-flat surface rather than on a flat domain. With applications to robot vision, Imiya~et~al.~\cite{ImiSugTorMoc05,TorImiSugMoc05} considered optical flow for spherical images. Lef\`{e}vre and Baillet~\cite{LefBai08} extended the Horn-Schunck method to general 2-Riemannian manifolds, showed well-posedness, and applied it to brain imaging data. They solved the numerical problem with finite elements on a surface triangulation. In all of the above works the underlying imaging surface is fixed over time, while in this paper it is not.

A preliminary version of this paper appeared in \cite{KirLanSch13}. The main differences to the present article are as follows. First, our current implementation allows us to regularise spatiotemporally as well as only spatially. In \cite{KirLanSch13} we only treated spatial regularisation. Second, the spatial regularisation functional has been improved in the sense that it is now parametrisation invariant. We have also conducted new experiments with the cell microscopy data and, in contrast to \cite{KirLanSch13}, computed approximate cell trajectories. Finally, we added some recent references.
\section{Notation and Background} \label{sec:preliminaries}
Whenever convenient we make use of the Einstein summation convention. Every index that appears exactly twice in an expression, once as a sub- and once as a superscript, is summed over.

\subsection{Evolving Surfaces} \label{sec:evolsurf} Let $\cali{M} = \left( \cali{M}_t \right)_{t\in I}$ be a family of compact smooth 2-manifolds $\cali{M}_t \subset \Reals^3$ indexed by a time interval $I=[0,T]$. Each $\mathcal{M}_t$ is assumed to be oriented by the unit normal field $\vec{N}(t,\cdot)$. For every $t\in I$ and $x\in\mathcal{M}_t$ the orthogonal projector onto the tangent plane $T_x\mathcal{M}_t$ is given by
\begin{equation}\label{eq:proj}
	\mathrm{P}(t,x) \coloneqq \mathrm{Id} - \vec{N}(t,x) \vec{N}(t,x)^\top.
\end{equation}
We call $\cali{M}$ an \emph{evolving surface}, if there is a smooth function
\begin{equation*}
	\phi : I \times \cali{M}_0 \to \Reals^3
\end{equation*}
such that $\phi(t,\cdot)$ is a diffeomorphism between $\cali{M}_0$ and $\cali{M}_t$ for every $t$, and $\phi(0,\cdot)$ is the identity on $\mathcal{M}_0$. Note that $\phi$ cannot be unique in general. With every $\phi$ there is associated a surface velocity. Denote the inverse of $\phi(t,\cdot)$ by $\phi_t^{-1}(\cdot)$. The surface velocity at a point $x \in \mathcal{M}_t$ is then defined by
\begin{equation}\label{eq:velocity}
	\vec{V}(t,x) \coloneqq \partial_t \phi \left(t,\phi_t^{-1}(x)\right).
\end{equation}
In contrast to $\phi$ the domain of $\vec{V}$ is not $I \times \cali{M}_0$, but rather the 3-manifold
\begin{equation*}
	\bar{\mathcal{M}} \coloneqq \bigcup_{t\in I} \left( \{t\} \times \mathcal{M}_t \right) \subset \Reals^4.
\end{equation*}
In other words, $\vec{V}$ is a Eulerian specification of $\mathcal{M}$, while $\phi$ is a Lagrangian one. Even though different functions $\phi,\phi'$ give rise to different velocities $\vec{V},\vec{V}'$, the normal velocity of $\mathcal{M}$ is independent of the choice of $\phi$. That is, $\vec{V} \cdot \vec{N} = \vec{V}'\cdot \vec{N}$. We provide a short proof of this statement in Proposition \ref{thm:normvelo} in the Appendix. Given a Eulerian specification $\vec{V}$ of $\mathcal{M}$, we can obtain, at least locally, a Lagrangian one by solving the ordinary differential equation \eqref{eq:velocity} for $\phi$ with initial condition $\phi(0,x_0) = x_0$. From now on we consider $\phi$ and $\vec{V}$ fixed. See Sec.~\ref{sec:parametrisation} for the specific $\phi$ and $\vec{V}$ we use in the numerical computations.

Let $\vec{x}_0 : \Omega \subset \Reals^2 \to \Reals^3$ be a parametrisation of $\mathcal{M}_0$ mapping local coordinates $\xi = \left( \xi^1, \xi^2 \right)$ to points $x = \left( x^1, x^2, x^3 \right)$ of Euclidean space. By composing $\phi$ and $\vec{x}_0$ we obtain a parametrisation of the evolving surface $\mathcal{M}$
\begin{equation}\label{eq:parametrisation}
	\vec{x}: I \times \Omega \to \Reals^3,\quad \vec{x}(t,\xi) = \phi \left(t, \vec{x}_0(\xi) \right).
\end{equation}
With this convention we always have $\partial_t \vec{x} = \vec{V}$. Differentiation with respect to $\xi^i$ will be denoted by $\partial_i$. The set $\{ \partial_1 \vec{x} (t,\xi), \partial_2 \vec{x}(t,\xi) \}$ forms a basis of $T_{\vec{x} (t,\xi)}\mathcal{M}_t$. Note that this basis is not orthonormal in general. Using dot notation for the standard inner product of $\mathbb{R}^3$, the components of the first fundamental form $g = \left(g_{ij}\right)$ are given by
\begin{equation}\label{eq:metric}
	g_{ij} = \partial_i \vec{x} \cdot \partial_j \vec{x}.
\end{equation}
The elements of its inverse are denoted by upper indices $g^{-1} = \left( g^{ij} \right)$.

Let $F: \bar{\mathcal{M}} \to \Reals$ be a scalar function and $f:I\times\Omega \to \mathbb{R}$ its coordinate representation,\footnote{Distinguishing between a surface quantity and its coordinate representation is often avoided. We decided, however, to make this distinction for the data $F$, and only for $F$, as we found it helpful especially in Sec.~\ref{sec:model}.} that is
\begin{equation*}
	F(t,\vec{x}(t,\xi)) = f(t,\xi).
\end{equation*}
The integral of $F$ over the evolving surface is then given by
\begin{equation*}
	\int_I \int_{\mathcal{M}_t} F \, \mathrm{d}A \, \mathrm{d}t \coloneqq \int_I \int_{\Omega} f \sqrt{\det g} \, \mathrm{d}\xi \, \mathrm{d}t,
\end{equation*}
where $\mathrm{d}A$ denotes the surface measure.

We refer to \cite{CerFriGur05}, \cite{DziEll13} and the references therein for more information on evolving surfaces. Eulerian and Lagrangian coordinates can be read up in Sec.~2.1 of \cite{Bat99}, for example.

\subsection{Derivatives on Evolving Surfaces}

\paragraph{Spatial Derivatives.} The spatial differential operators introduced below are not different from those on static manifolds. Therefore $t \in I$ can be considered fixed in this paragraph.

The surface gradient $\nabla_{\mathcal{M}} F$ of $F$ is the tangent vector field which points in the direction of greatest increase of $F$. In local coordinates it is given by
\begin{equation} \label{eq:grad}
	\nabla_{\mathcal{M}} F = g^{ij} \partial_i f \partial_j \vec{x},
\end{equation}
where we omitted the arguments $(t, \vec{x}(t,\xi))$ on the left and $(t,\xi)$ on the right hand side, respectively. The surface gradient is just the tangential part of the $\mathbb{R}^3$ gradient. More precisely, if $\hat F$ is a smooth extension of $F$ to an open neighbourhood of $\mathcal{M}_t$ in $\mathbb{R}^3$, then
\begin{equation*}
	\nabla_{\mathcal{M}} F = \mathrm{P} \nabla_{\mathbb{R}^3} \hat F.
\end{equation*}
Note that the last expression does not depend on the choice of $\hat F$.

Similarly, for two tangent vector fields $\vec{u}$, $\vec{v}$ on $\mathcal{M}_t$ the covariant derivative $\nabla_{\vec{v}} \vec{u}$ of $\vec{u}$ along $\vec{v}$ is the tangential part of the conventional directional derivative of $\vec{u}$ along $\vec{v}$. That is
\begin{equation*}
	\nabla_{\vec{v}} \vec{u} = \mathrm{P} \nabla_{\mathbb{R}^3} \hat{\vec{u}} (\vec{v}),
\end{equation*}
where $\hat{\vec{u}}$ is an extension of $\vec{u}$ as above and $\nabla_{\mathbb{R}^3} \hat{\vec{u}} (\vec{v})$ is the Jacobian of $\hat{\vec{u}}$ applied to $\vec{v}$. Let $\vec{u} \coloneqq u^i \partial_i \vec{x}$ and $\vec{v} \coloneqq v^i \partial_i \vec{x}$ be their representations in the coordinate basis. The covariant derivative then reads
\begin{equation}\label{eq:covariant}
	\nabla_{\vec{v}} \vec{u} = \left( v^i \partial_i u^j + v^i u^k \Gamma^j_{ik} \right) \partial_j \vec{x}.
\end{equation}
The Christoffel symbols $\Gamma^j_{ik}$ are defined by the action of $\nabla$ on the coordinate basis
\begin{equation}\label{eq:christoffel}
	\nabla_{\partial_i \vec{x}} \partial_k \vec{x} = \Gamma^j_{ik} \partial_j \vec{x}.
\end{equation}
An explicit expression for the Christoffel symbols in terms of the first fundamental form is given by
\begin{equation*}
	\Gamma_{ik}^{j} =
	\frac{1}{2} {g^{mj} \left( \partial_i g_{km} + \partial_k g_{mi} - \partial_m g_{ik} \right)}.
\end{equation*}

Recall that the coordinate basis is in general not orthonormal. In Sec.~\ref{sec:eulerlagrange}, however, we want to rewrite the regularisation functional in terms of an orthonormal basis in order to simplify subsequent calculations. Therefore we now make the little extra effort of expressing the covariant derivative $\nabla_{\vec{v}} \vec{u}$ in terms of an arbitrary, possibly non-coordinate, frame $\{\vec{e}_1, \vec{e}_2\}$. Writing $\vec{u} = w^i \vec{e}_i$ and $\vec{v} = z^i \vec{e}_i$ in this basis, the corresponding formula reads
\begin{equation} \label{eq:covariant2}
	\nabla_{\vec{v}} \vec{u} = \left( \nabla_{\vec{v}} w^j + z^i w^k \tilde{\Gamma}^j_{ik} \right) \vec{e}_j.
\end{equation}
For scalar functions like $w^j$ the covariant derivative $\nabla_{\vec{v}} w^j$ is just the directional derivative along $\vec{v}$. It can be computed by using linearity of the covariant derivative with respect to its lower argument
\begin{equation*}
	\nabla_{\vec{v}} w^j = \nabla_{v^i \partial_i \vec{x}} w^j = v^i \nabla_{ \partial_i \vec{x}} w^j = v^i \partial_i w^j.
\end{equation*}
The $\tilde{\Gamma}^j_{ik}$ are the symbols associated to the new frame $\{\vec{e}_1, \vec{e}_2\}$. In analogy to \eqref{eq:christoffel}, they are defined by
\begin{equation}\label{eq:symbols}
	\nabla_{\vec{e}_i} \vec{e}_k = \tilde\Gamma^j_{ik} \vec{e}_j.
\end{equation}
For an orthonormal frame $\{\vec{e}_1, \vec{e}_2\}$ the following transformation law describes the relation between the two types of symbols
\begin{equation}\label{eq:transymb}
	\tilde{\Gamma}^j_{ik} = \delta^{jp} \alpha^h_p g_{hm}\left( \alpha^\ell_i \partial_\ell \alpha^m_k + \alpha^\ell_i \alpha^n_k \Gamma^m_{\ell n} \right),
\end{equation}
where $\alpha^j_i$ is the $\partial_j \vec{x}$-coordinate of $\vec{e}_i$, that is, $\vec{e}_i = \alpha^j_i \partial_j \vec{x}$ and $\delta^{jp}$ is the Kronecker delta. We give a short derivation of the equation above in Lemma \ref{thm:transymb} in the Appendix.

The covariant derivative of $\vec{u}$ at a point $(t,\xi)$ is a linear operator on $T_{\vec{x}(t,\xi)}\mathcal{M}_t$, mapping tangent vectors $\vec{v}$ to tangent vectors $\nabla_{\vec{v}} \vec{u}$. Its 2-norm (Frobenius norm) can be computed via
\begin{equation}\label{eq:frob}
	\norm{\nabla \vec{u} (t,\xi) }_2^2 = \abs{\nabla_{\vec{e}_1} \vec{u}(t,\xi)}^2 + \abs{\nabla_{\vec{e}_2} \vec{u}(t,\xi)}^2,
\end{equation}
where $\{\vec{e}_1, \vec{e}_2\}$ now is an arbitrary orthonormal basis of the tangent space $T_{\vec{x}(t,\xi)}\mathcal{M}_t$, that is, $\vec{e}_i \cdot \vec{e}_j = \delta_{ij}$. Note that, if $\vec{x}$ is a global parametrisation, then we can obtain a frame $\{\vec{e}_1, \vec{e}_2\}$ which is orthonormal everywhere by Gram-Schmidt orthonormalisation of the coordinate basis $\{ \partial_1 \vec{x}, \partial_2 \vec{x} \}$.

For a thorough treatment of the concepts introduced in this section we refer to \cite{Car92,Lee97}. More basic differential geometry texts are \cite{Car76,Kue06}, for example.

\paragraph{Temporal Derivatives.} Let $x \in \mathcal{M}_{t_0}$. Denote by $\psi: t \mapsto \psi(t) \in \mathcal{M}_t$ a trajectory through $\mathcal{M}$ with $\psi(t_0) = x$. We define the time derivative of $F$ following $\psi$ at $x$ as\footnote{Note that this composition of $F$ with $\psi$ is necessary, because the conventional partial derivative $\partial_t F(t_0,x)$ is meaningless in general.}
\begin{equation}\label{eq:timeder}
	\mathrm{d}^{\psi}_t F (t_0,x) \coloneqq \left. \frac{\mathrm{d}}{\mathrm{d}t} F(t,\psi(t)) \right|_{t=t_0}.
\end{equation}
There are a few special cases of this derivative that are worth mentioning. Let $\psi_{\vec{N}}$ be a trajectory for which the vector $\partial_t \psi (t_0)$ is orthogonal to $T_x \mathcal{M}_{t_0}$. The corresponding derivative is called normal time derivative and denoted by
\begin{equation} \label{eq:normaltimeder}
	\mathrm{d}^{\vec{N}}_t F (t_0,x) \coloneqq \left. \frac{\mathrm{d}}{\mathrm{d}t} F(t,\psi_{\vec{N}}(t)) \right|_{t=t_0}.
\end{equation}
Every Lagrangian coordinate system $\phi$ of $\mathcal{M}$ engenders a time derivative like \eqref{eq:timeder} in a natural way. For $x = \phi(t,y) \in \mathcal{M}_t$ the time derivative of $F$ following $\phi$ is defined by
\begin{equation}\label{eq:dVFt}
	\mathrm{d}^{\vec{V}}_t F (t_0,x) \coloneqq \left. \frac{\mathrm{d}}{\mathrm{d} t} F(t,\phi(t,y)) \right|_{t=t_0}.
\end{equation}
We choose the notation $\mathrm{d}^{\vec{N}}_t F$ and $\mathrm{d}^{\vec{V}}_t F$, because the derivative \eqref{eq:timeder} in fact only depends on the velocity of $\psi$ at $x$, see Lemma \ref{thm:reul}. Finally, if $\mathcal{M}$ is parametrised according to \eqref{eq:parametrisation}, which we assume from now on, then $\mathrm{d}^{\vec{V}}_t F = \partial_t f$. For illustration see Fig.~\ref{fig:sketch2}.

We stress that if $\vec{V}$ is the physical surface velocity, then $\mathrm{d}^{\vec{V}}_t$ is the natural time derivative for functions defined on $\bar{\mathcal{M}}$, since it measures the temporal change along trajectories $\phi(\cdot,y)$ of surface points. These trajectories are \emph{not} cell trajectories in general. They coincide only if the cells do not move by themselves and all the motion is surface motion.
\begin{figure}
	\begin{center}
	\begin{tikzpicture}
		\draw [thick, gray] (-3,0) to [out=30,in=150] (3,0);
		\draw [thick, gray] (-3,1) to [out=40,in=150] (4,1.8);
		\node [right] at (3,0) {$\mathcal{M}_{t_0}$};
		\node [below] at (4,1.8) {$\mathcal{M}_{t_0+\Delta t}$};
		\draw [-stealth', gray, thick] (0,.9) -- (.91*1.1,.91*2.6);
		\node [left, gray] at (.91*1.1,.91*2.6) {$\vec{V}$};
		\draw [-stealth', gray, thick] (0,.9) -- (0,2.4);
		\draw [-stealth', gray, thick] (0,.9) -- (.94*2.5,.94*2.4);
		\node [below, gray] at (.98*2.5,.94*2.4) {$\vec{m}$};
		\draw [-stealth', thick, dotted] (0,0) to [out=90,in=225] (0,.9) to [out=45,in=270] (1.3,1.7) to [out=90,in=225] (2.5,2.4) to [out=45,in=225] (1.25*2.5,1.25*2.4);
		\node [right] at (1.3*2.5,1.3*2.4) {$\gamma$};
		\draw [-stealth', thick] (0.5,0) to [out=135,in=270] (0,.9) to [out=90,in=300] (-.5,1.25*2.4);
		\node [above] at (-.5,1.25*2.4) {$\psi_{\vec{N}}$};
		\draw [-stealth', thick, dashed] (-0.5,0) to [out=45,in=225] (0,.9) to [out=45,in=270] (1.5,3.1);
		\node [above] at (1.5,3.1) {$\phi(\cdot,x_0)$};
	\end{tikzpicture}
	\end{center}
	\caption{Sketch of different trajectories through the evolving surface giving rise to different temporal derivatives. Corresponding velocities are depicted in grey.}
	\label{fig:sketch2}
\end{figure}
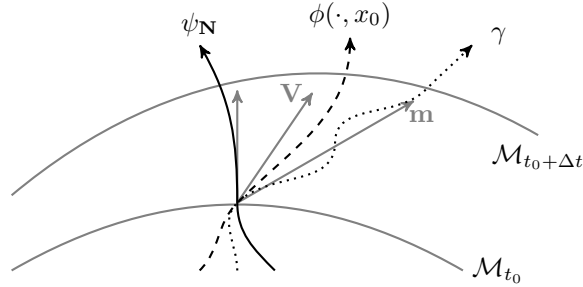
\begin{lemma} \label{thm:reul} With the definitions from above, we have
\begin{equation*}
	\mathrm{d}^{\vec{V}}_t F = \mathrm{d}^{\vec{N}}_t F + \nabla_\mathcal{M} F \cdot \vec{V}.
\end{equation*}
\end{lemma}
\begin{proof}
The main idea in this derivation from \cite{CerFriGur05} is to consider the normally constant extension $\hat{F}$ of $F$: Let $\bar{\mathcal{N}} \subset \Reals^4$ be an open 
neighbourhood of $\bar{\mathcal{M}}$. If $\bar{\mathcal{N}}$ is chosen sufficiently small, it is possible to define a function $\hat{F}:\bar{\mathcal{N}} \to \Reals$ that is smooth, 
constant on normal lines through $\mathcal{M}_t$ for every $t$, and agrees with $F$ on $\bar{\mathcal{M}}$. Therefore
\begin{align*}
	\frac{\mathrm{d}}{\mathrm{d} t} F(t,\phi(t,y))
		&= \frac{\mathrm{d}}{\mathrm{d} t} \hat{F}(t,\phi(t,y))	\\
		&= \partial_t \hat{F} + \nabla_{\mathbb{R}^{3}} \hat{F} \cdot \partial_t \phi \\
		&= \mathrm{d}^{\vec{N}}_t F + \nabla_\mathcal{M} F \cdot \vec{V}
\end{align*}
The last equality holds because, by construction, $\nabla_{\mathbb{R}^3} \hat{F}$ equals $\nabla_\mathcal{M} F$ and
\begin{equation*}
	\mathrm{d}^{\vec{N}}_t F
		= \frac{\mathrm{d}}{\mathrm{d} t} \hat{F}(t,\psi_{\vec{N}}(t)) 
		= \partial_t \hat{F} + \nabla_{\mathbb{R}^3} \hat{F} \cdot \partial_t \psi_{\vec{N}}
		= \partial_t \hat{F}.
\end{equation*}
Finally, note that by definition \eqref{eq:velocity} we have $\partial_t \phi = \vec{V}$.
\end{proof}
Note that, since $\nabla_\mathcal{M} F$ is tangential, $\mathrm{d}^{\vec{V}}_t F$ actually only depends on the tangential part $\mathrm{P}\vec{V}$ of $\vec{V}$. Here $\mathrm P$ is the orthogonal projector defined in \eqref{eq:proj}.

Let $\vec{u}$ be a tangent vector field on the evolving surface $\mathcal{M}$, that is, a function $\vec{u} : \bar{\mathcal{M}} \to \mathbb{R}^3$ such that
\begin{equation*}
	\vec{u}(t,\cdot) : \mathcal{M}_t \to T \mathcal{M}_t
\end{equation*}
for all $t$. In analogy to the covariant derivative \eqref{eq:covariant} and to \eqref{eq:dVFt}, we define the following time derivative
\begin{equation}\label{eq:covariantt}
	\nabla_t \vec{u} = \mathrm{P} \mathrm{d}_t^{\vec{V}} \vec{u},
\end{equation}
where application of $\mathrm{d}_t^{\vec{V}}$ to $\vec{u}$ is understood componentwise. Again we have $\mathrm{d}_t^{\vec{V}} \vec{u} = \partial_t \vec{u}$. A normal time derivative for $\vec{u}$ could be defined as well but will not be needed in the sequel. As in the scalar case, $\nabla_t \vec{u}$ can be considered the natural time derivative for a tangent vector field $\vec{u}$, if $\vec{V}$ is the physical surface velocity. By setting
\begin{equation}\label{eq:Christoffel02}
	\nabla_t \partial_i \vec{x} = \Gamma^j_{0i} \partial_j \vec{x}
\end{equation}
we arrive at the following expression for $\nabla_t \vec{u}$ in local coordinates
\begin{equation*}
	\nabla_t \vec{u} = \left( \partial_t u^j + u^i\Gamma^j_{0i} \right) \partial_j \vec{x}.
\end{equation*}
The new symbols have the explicit representation
\begin{equation}\label{eq:Christoffel0}
	\Gamma^j_{0i} = g^{jk} \partial_{ti} \vec{x} \cdot \partial_k \vec{x},
\end{equation}
which can be verified by taking inner products of both sides of \eqref{eq:Christoffel02} with the coordinate basis vectors.

Again, in order to simplify calculations later on, we want to express this derivative in terms of an orthonormal frame $\{\vec{e}_1, \vec{e}_2\}$. We have
\begin{equation} \label{eq:covariantt2}
	\nabla_t \vec{u}	=	\left( \partial_t w^j + w^i\tilde{\Gamma}^j_{0i} \right) \vec{e}_j,
\end{equation}
where the symbols $\tilde{\Gamma}^j_{0i}$ are defined as before and satisfy an analogous transformation law
\begin{equation} \label{eq:transymb0}
	\tilde{\Gamma}^j_{0i}	=	\delta^{jp} \alpha^h_p g_{hm} \left( \partial_t \alpha^m_i + \alpha^k_i \Gamma^m_{0k} \right).
\end{equation}
The derivation is analogous to \eqref{eq:transymb} and can be found in Lemma \ref{thm:transymb} in the Appendix.

\section{Model Statement}\label{sec:model}
\subsection{Generalised Optical Flow Equation} We assume to be given an evolving surface $\mathcal{M}$ together with a known Lagrangian specification $\phi$ or, equivalently, a Eulerian one $\vec{V}$. In addition we are given scalar data $F$ on $\mathcal{M}$ which we want to track over time.

Our optical flow model is based on the so-called brightness constancy assumption. For every $x\in\mathcal{M}_0$ we seek a trajectory $\gamma(\cdot,x): t \mapsto \gamma(t,x) \in \mathcal{M}_t$ along which $F$ is constant. 
In other words, we assume existence of a Lagrangian specification $\gamma$ of $\mathcal{M}$ such that
\begin{equation}\label{eq:bca}
	F(t,\gamma(t,x)) = F(0,x).
\end{equation}
This implies that the time derivative of $F$ following $\gamma$ has to vanish identically. We deduce from Lemma \ref{thm:reul} that the following generalised optical flow equation has to hold
\begin{equation}\label{eq:ofc}
	\mathrm{d}_t^{\vec{N}} F + \nabla_\mathcal{M} F \cdot \partial_t \gamma = 0,
\end{equation}
where $\mathrm{d}_t^{\vec{N}} F$ is the normal time derivative as defined in \eqref{eq:normaltimeder} and $\nabla_\mathcal{M} F$ is the surface gradient of $F$, cf.~\eqref{eq:grad}.

Let us continue the discussion of Sec.~\ref{sec:motivation}. According to our definition of $\gamma$, a cell located at $x_0 \in \mathcal{M}_{t_0}$ moves with velocity
\begin{equation} \label{eq:totalmotion}
	\partial_t \gamma (t_0, \gamma^{-1}_{t_0}(x_0) ) = \vec{m}(t_0,x_0) = \vec{u}(t_0,x_0) + \vec{V}(t_0,x_0),
\end{equation}
where $\gamma^{-1}_{t_0}$ is the inverse of $\gamma(t_0,\cdot)$, $\vec{m}$ is the total observed velocity of a cell as introduced in Sec.~\ref{sec:motivation} and $\vec{u}$ is its velocity relative to $\vec{V}$. The second equality above is due to decomposition \eqref{eq:decomp}. According to our assumptions at the beginning of this section, we consider $\vec{V}$ as given so that the actual unknown is $\vec{u}$.

The remaining part of this subsection is devoted to rewriting \eqref{eq:ofc} in terms of local coordinates. First, we give an interpretation of the coordinates $u^i$ of $\vec{u}$ with respect to the basis $\{\partial_1 \vec{x}, \partial_2 \vec{x}\}$. Let $\beta = \left( \beta^1, \beta^2 \right) : I \times \Omega \to \Omega$ be the coordinate counterpart of $\gamma$, defined by the equation
\begin{equation*}
	\gamma(t,\vec{x}_0(\xi)) = \vec{x}(t,\beta(t,\xi)).
\end{equation*}
See also Fig.~\ref{fig:cd}. Taking time derivatives on both sides and dropping arguments yields
\begin{equation*}
	\vec{m} = \vec{V} + \partial_t \beta^i \partial_i \vec{x},
\end{equation*}
since $\partial_t \vec{x} = \vec{V}$. We can conclude that $u^i = \partial_t \beta^i$, which means that $(u^1, u^2)$ is just the 2D velocity of the parametrised trajectory $\beta$. It remains to rewrite \eqref{eq:ofc} in terms of $u^1$ and $u^2$.
\begin{lemma}
The optical flow equation \eqref{eq:ofc} is equivalent to
\begin{equation*}
	 \mathrm{d}_t^{\vec{V}} F + \nabla_\mathcal{M} F \cdot \vec{u} = 0.
\end{equation*}
In local coordinates it reads
\begin{equation*}
	\partial_t f + u^i \partial_i f = 0.
\end{equation*}
\end{lemma}
\begin{proof}
We prove the assertion in two steps. First we show that
\begin{equation*}
	 \mathrm{d}_t^{\vec{N}} F + \nabla_\mathcal{M} F \cdot \partial_t \gamma = \mathrm{d}_t^{\vec{V}} F + \nabla_\mathcal{M} F \cdot \vec{u},
\end{equation*}
and afterwards rewrite the right hand side in local coordinates.

By Lemma \ref{thm:reul} the normal time derivative can be written as
\begin{equation*}
	\mathrm{d}_t^{\vec{N}} F = \mathrm{d}_t^{\vec{V}} F - \nabla_\mathcal{M} F \cdot \vec{V}
\end{equation*}
The other summand of \eqref{eq:ofc} rewrites as
\begin{equation*}
	\nabla_\mathcal{M} F \cdot  \partial_t \gamma = \nabla_\mathcal{M} F \cdot \left( \vec{V} + \vec{u} \right).
\end{equation*}
Note that $\vec{V}$ is not assumed to be normal to $\mathcal{M}_t$, so that the term $\nabla_\mathcal{M} F \cdot \vec{V}$ does not vanish in general. However, it does appear twice with opposite signs. Finally recall that $\mathrm{d}_t^{\vec{V}} F = \partial_t f$ and by the definition of the first fundamental form
\begin{align*}
	\nabla_\mathcal{M} F \cdot \vec{u}
		&= g^{ij} \partial_i f \partial_j \vec{x} \cdot u^k \partial_k \vec{x}	\\
		&= g^{ij} g_{jk} \partial_i f u^k	\\
		&= \partial_i f u^i.
\end{align*}
\end{proof}
It is worth noting that the parametrised optical flow equation has precisely the same form as the classical 2D equation.

\begin{figure}
	\centering
	\begin{tikzpicture}[every node/.style={scale=1.2}]
		\matrix (m) [matrix of math nodes,row sep=3em,column sep=4em,minimum width=2em]{
	    	\Omega & \Omega \\
	    	\mathcal{M}_0 & \mathcal{M}_t \\};
		\path[-stealth]
	    	(m-1-1) edge node [left]	{$\vec{x}_0(\cdot)$}	(m-2-1)
	    			edge node [above]	{$\beta(t,\cdot)$}		(m-1-2)
	    	(m-2-1) edge node [above]	{$\gamma(t,\cdot)$}		(m-2-2)
	    	(m-1-2) edge node [right]	{$\vec{x}(t,\cdot)$}	(m-2-2);
	\end{tikzpicture}
	\caption{Commutative diagram describing the relation between $\beta$ and $\gamma$.}
	\label{fig:cd}
\end{figure}

\subsection{Regularisation}
Directly solving the optical flow equation in the new setting is just as ill-posed as it is in the classical setting. We use variational regularisation to overcome this. In particular, we propose to minimise the following quadratic spatiotemporal functional to recover a vector field $\vec{u}$ describing the tangential motion of data $F$.
\begin{equation} \label{eq:functional1}
\begin{aligned}
	\int_I \int_{\mathcal{M}_t} & \Big( \left( \mathrm{d}_t^{\vec{V}} F + \nabla_\mathcal{M} F \cdot \vec{u} \right)^2 + \lambda_0 \abs{\nabla_t \vec{u}}^2 + \lambda_1 \norm{\nabla \vec{u}}_2^2 \Big) \, \mathrm{d}A \, \mathrm{d}t
\end{aligned}
\end{equation}
Here $\lambda_0 \ge 0$ and $\lambda_1 > 0$ are regularisation parameters. Recall from Sec.~\ref{sec:preliminaries} that $\vec{u}$ is temporally regularised according to the assumed surface motion $\vec{V}$. Functional \eqref{eq:functional1} is a generalisation of the one presented in \cite{WeiSchn01b} for the Euclidean setting.

Moreover, if $\lambda_0 = 0 $, minimisation of \eqref{eq:functional1} is equivalent to minimising
\begin{equation} \label{eq:functional2}
	\int_{\mathcal{M}_t} \left( \left( \mathrm{d}_t^{\vec{V}} F + \nabla_\mathcal{M} F \cdot \vec{u} \right)^2 + \lambda_1 \norm{\nabla \vec{u}}_2^2 \right) \, \mathrm{d}A
\end{equation}
for every instant $t \in I$ separately. If $\mathcal{M}_t = \mathcal{M}_0$ for all $t$, the functional reduces to that of \cite{LefBai08}. The spatial regularisation term as defined in \eqref{eq:frob} is independent of the chosen parametrisation. This is an improvement over the functional chosen in \cite{KirLanSch13}.
\begin{example}
We end this section with a brief explanation, from an applied point of view, of why we regularise with covariant derivatives. Consider as a toy manifold the non-moving unit circle $\mathcal{M}_t = \mathcal{S}^1 \subset \Reals^2$ with parametrisation $\vec{x}(\theta) = \left( \cos \theta, \sin \theta \right)^\top$, $\theta \in [0,2\pi)$ and tangent basis $ \{ \partial_\theta \vec{x} \} $. Consider the tangent vector vector field $\vec{u} = c\partial_\theta \vec{x}$, where $c \neq 0$ is a fixed number. This field would describe a uniform translation of data $F$ along the circle, and thus should not be penalised by a regularisation term that enforces spatial smoothness.
But while conventional differentiation does not yield a vanishing vector field
\begin{equation*}
	\partial_\theta \vec{u} = c \partial_{\theta \theta} \vec{x} = -c \vec{x},
\end{equation*}
covariant differentiation does
\begin{equation*}
	\nabla_\theta \vec{u} = \mathrm{P} \partial_\theta \vec{u} = -c \mathrm{P} \vec{x} = -c (\vec{x} - \vec{x} \vec{x}^\top \vec{x} ) = 0.
\end{equation*}
Here we used the fact that $\vec{N} = \vec{x}$ and $\vec{x}^\top \vec{x} = 1$.

An analogous argument explains our penalisation of $\nabla_t \vec{u}$ of $\gamma$ instead of the unprojected derivative $\partial_t \vec{u}$.
\end{example}
\section{Euler-Lagrange Equations} \label{sec:eulerlagrange}
To simplify matters from now on we will assume having a global parametrisation $\vec{x}_0$ of $\mathcal{M}_0$ and thus a global parametrisation $\vec{x}$ of the whole evolving surface, cf.~\eqref{eq:parametrisation}. 
In addition, we express the functional \eqref{eq:functional1} in an orthonormal non-coordinate basis $\{\vec{e}_1, \vec{e}_2\}$ with
\begin{equation}\label{eq:onb}
	\vec{e}_i = \alpha^j_i \partial_j \vec{x}.
\end{equation}
This leads to wearisome calculations at first, which however pay off eventually when we compute the optimality conditions for the coordinates of $\vec{u}$ with respect to this frame. Note that an orthonormal coordinate basis does not exist in general \cite{Lee97}.

In this section we use the following notational convention. First, we identify $t$ with $\xi^0$. In addition, Latin indices are always understood to run over the set $\{1,2\}$, while Greek indices are reserved for $\{0,1,2\}$.

\subsection{Rewriting Functional \eqref{eq:functional1}}
Let
\begin{equation} \label{eq:unknown}
	\vec{u} = w^i \vec{e}_i
\end{equation}
be the representation of the unknown $\vec{u}$ in the orthonormal frame \eqref{eq:onb}. It follows that $u^j = w^i \alpha^j_i$. Recall from \eqref{eq:covariant2}, \eqref{eq:covariantt2} that the derivatives of $\vec{u}$ read
\begin{align*}
	\nabla_{\vec{e}_i} \vec{u} 	&= \left( \alpha^k_i \partial_k w^j + w^k \tilde{\Gamma}^j_{ik} \right) \vec{e}_j, \\
	\nabla_t \vec{u}			&=	\left( \partial_t w^j + w^i\tilde{\Gamma}^j_{0i} \right) \vec{e}_j.
\end{align*}
If we set $\alpha_\mu^0 = \delta^0_\mu$ and $\alpha_0^\mu = \delta_0^\mu$, the coefficients of $\vec{e}_j$ above can be rewritten using the unified notation
\begin{equation*}
	D_\mu w^j = \alpha_\mu^\nu \partial_\nu w^j + w^i \tilde{\Gamma}^j_{\mu i},
\end{equation*}
where $\mu=0,1,2$ and $j=1,2$. Consequently, defining the operator $D = (D_0, D_1, D_2)^\top$, the integrand of the regularisation term becomes a weighted 2-norm of the matrix $Dw = (D_\mu w^j)_{\mu j}$. The parametrised version of energy functional \eqref{eq:functional1} now takes the following compact form
\begin{equation}\label{eq:functional3}
\begin{aligned}
	\int_0^T \int_\Omega	&  \Big( \left( \partial_t f + w^j \alpha_j^i \partial_i f \right)^2 + \sum_{\mu, j} \lambda_\mu \left( D_\mu w^j \right)^2 \Big) \sqrt{\det g} \,  \mathrm{d}\xi \, \mathrm{d}t,
\end{aligned}
\end{equation}
where $\lambda_1 = \lambda_2$ and $g$ is the first fundamental form as introduced in \eqref{eq:metric}. Observe that the simple form of the regulariser originates from representing $\nabla_{\vec{e}_i} \vec{u}$ and $\nabla_{t} \vec{u}$ in an orthonormal basis. This also simplifies the computation of the optimality conditions. 

\subsection{Optimality System}
Denote the interior of  $I \times \Omega \subset \Reals^3$ by $D$. Functional \eqref{eq:functional3} takes the general form
\begin{equation*}
	\mathcal{E}(w)	= \int_D L(w, \nabla w) \, \mathrm{d}\xi,
\end{equation*}
where the Lagrangian $L$ is a smooth function of all $w^i$ and $\partial_\mu w^i$. Denote partial derivatives of $L$ by subscripts. A minimiser $(w^1,w^2)$ of $\mathcal{E}$ has to satisfy the following second-order elliptic system
\begin{equation} \label{eq:optsys}
	\begin{aligned}
		L_{w^m}	&= \sum_\mu \partial_\mu L_{\partial_\mu w^m}, \quad \text{in } D,\\
		0 		&= \sum_\mu n_\mu L_{\partial_\mu w^m}, \quad \text{on } \partial D,
	\end{aligned}
\end{equation}
for $m=1,2$ and where $n=(n_0,n_1,n_2)^\top$ is the outward normal to $D$. The derivatives of the Lagrangian read
\begin{align*}
	L_{w^m}					&=	\sqrt{\det g} \Big( \alpha_m^i \partial_i f \left( w^j \alpha^k_j \partial_k f + \partial_t f \right) + \sum_{\mu,j} \lambda_\mu \tilde{\Gamma}^j_{\mu m} D_\mu w^j \Big) ,  \\
	L_{\partial_\nu w^m}	&=	\sqrt{\det g} \sum_{\mu}\lambda_\mu \alpha^\nu_\mu D_\mu w^m.
\end{align*}
System \eqref{eq:optsys} in terms of derivatives of $w$ together with explicit formulas for all coefficients can be found in the Appendix. For more details on variational calculus we refer to \cite{CouHil53,GelFom63}.
\begin{remark}
	If $\mathcal{M}$ is a fixed plane, then $\alpha_\mu^\nu = \delta_\mu^\nu$ and all connection symbols vanish. Consequently, the boundary conditions become standard Neumann ones and system \eqref{eq:optsys} reduces to the one derived in \cite{HorSchu81} or \cite{WeiSchn01b}, respectively.
\end{remark}

\subsection{Discretisation and Numerical Aspects} \label{sec:discretisation}
We solve the Euler-Lagrange equations~\eqref{eq:optsys} with a standard finite difference scheme. The spatiotemporal domain $D$ is assumed to be the unit cube $(0, 1)^{3}$ and is approximated by a Cartesian grid with spacing of $h_{\sigma}$ in the direction of $\xi^\sigma$, where $h_1=h_2$. Grid points are denoted by $p$. Thus, $w_p^{m} := w^{m}(p)$ refers to the numerical approximation of $w^{m}$ at $p \in D$. Partial derivatives of the unknowns are approximated using central differences. They read
\begin{align*}
	\partial_{\sigma} w^{m}(p) \approx \frac{1}{2 h_{\sigma}} \left( w_{\cali{N}_{\sigma}^{+}(p)}^{m} - w_{\cali{N}_{\sigma}^{-}(p)}^{m} \right),
\end{align*}
\begin{align*}
	\partial_{\sigma \sigma} w^{m}(p) \approx \frac{1}{h_{\sigma}^{2}} \left( w_{\cali{N}_{\sigma}^{+}(p)}^{m} - 2w_{p}^{m} + w_{\cali{N}_{\sigma}^{-}(p)}^{m} \right),
\end{align*}
and
\begin{align*}
	\partial_{\nu \sigma} w^{m}(p) \approx \frac{1}{4h_{\nu} h_{\sigma}}\left( w_{\cali{N}_{\nu\sigma}^{++}(p)}^{m} - w_{\cali{N}_{\nu\sigma}^{+-}(p)}^{m} - w_{\cali{N}_{\nu\sigma}^{-+}(p)}^{m} + w_{\cali{N}_{\nu\sigma}^{--}(p)}^{m} \right),
\end{align*}
where the symbols $\cali{N}_{\sigma}^{\pm}(p)$ and $\cali{N}_{\nu\sigma}^{\pm\pm}(p)$ denote the neighbours of $w_{p}^{m}$ in the grid along coordinates $\sigma$ and $\nu, \sigma$, respectively. From the choice of the discrete derivatives an eleven-point stencil is obtained; see Fig.~\ref{fig:stencil}. Derivatives of the data $f$ and the surface parametrisation $\vec{x}$ are handled likewise, using central differences in the interior and inward differences at the boundaries.

\begin{figure}[t]
	\centering
		\begin{tikzpicture}[thick]
			\node at (0,0) [circle, fill=black, inner sep=2pt, label=below right:$w_p^{m}$] {};
			\node at (0,1.5) [circle, fill=black, inner sep=2pt, label=above:$w_{\cali{N}_{1}^{-}(p)}^{m}$] {};
			\node at (1.5,1.5) [circle, fill=black, inner sep=2pt, label=above right:$w_{\cali{N}_{12}^{-+}(p)}^{m}$] {};
			\node at (1.5,0) [circle, fill=black, inner sep=2pt, label=right:$w_{\cali{N}_{2}^{+}(p)}^{m}$] {};
			\node at (0,-1.5) [circle, fill=black, inner sep=2pt, label=below:$w_{\cali{N}_{1}^{+}(p)}^{m}$] {};
			\node at (-1.5,-1.5) [circle, fill=black, inner sep=2pt, label=below left:$w_{\cali{N}_{12}^{+-}(p)}^{m}$] {};
			\node at (-1.5,1.5) [circle, fill=black, inner sep=2pt, label=above left:$w_{\cali{N}_{12}^{--}(p)}^{m}$] {};
\node at (-1.5,0) [circle, fill=black, inner sep=2pt, label=left:$w_{\cali{N}_{2}^{-}(p)}^{m}$] {};
			\node at (1.5,-1.5) [circle, fill=black, inner sep=2pt, label=below right:$w_{\cali{N}_{12}^{++}(p)}^{m}$] {};
			\node at (xyz cs:z=2) [circle, fill=black, inner sep=2pt, label=below:$w_{\cali{N}_{0}^{-}(p)}^{m}$] {};
			\node at (xyz cs:z=-2) [circle, fill=black, inner sep=2pt, label=above:$w_{\cali{N}_{0}^{+}(p)}^{m}$] {};
			\draw[-] (-1.5,1.5) -- (1.5,1.5);
			\draw[-] (-1.5,-1.5) -- (1.5,-1.5);
			\draw[-] (1.5,1.5) -- (1.5,-1.5);
			\draw[-] (-1.5,1.5) -- (-1.5,-1.5);
			\draw[-] (-1.5,0) -- (1.5,0);
			\draw[-] (0,-1.5) -- (0,1.5);
			\draw[-] (xyz cs:z=-2) -- (xyz cs:z=2);
		\end{tikzpicture}
\caption{Eleven point stencil arising from the discretisation.}
\label{fig:stencil}
\end{figure}
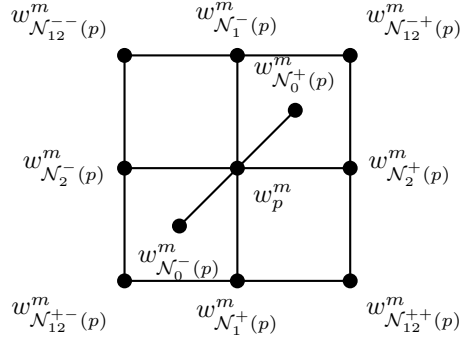

However, the resulting (sparse) linear system is underdetermined from equations~\eqref{eq:optsys} alone, because the approximations used for the mixed derivatives of $w^{m}$ refer to points not occurring in any boundary condition. Thus, at every grid point $p\in C\subset\partial D$ with
\begin{equation}
	C := \left( \{\xi^{1} = 0\} \cup \{\xi^{1} = 1\} \right) \cap  \left( \{\xi^{2} = 0\} \cup \{\xi^{2} = 1\} \right)
\label{eq:corners}
\end{equation}
additional boundary conditions are needed. At these points we set $n = (0, \pm 1, \pm 1)^{\top}$ in the boundary condition \eqref{eq:optsys}, which is a vector pointing in the direction of the undetermined grid neighbour. This leads to expressions of the form $\pm \partial_{1} w^{m} \pm \partial_{2} w^{m}$, which, interpreted as a directional derivative, can be approximated by
\begin{equation*}
	\frac{1}{2\sqrt{2}h_{\sigma}} \left( w_{\cali{N}_{ij}^{\pm\pm}(p)}^{m} - w_{\cali{N}_{ij}^{\mp\mp}(p)}^{m} \right).
\end{equation*}
\section{Experiments} \label{sec:exp}

\subsection{Zebrafish Microscopy Data}\label{sec:data}
As mentioned before, the biological motivation for this work are cellular image sequences of a zebrafish embryo. Endoderm cells expressing green fluorescent protein were recorded via confocal laser-scanning microscopy resulting in time-lapse volumetric (4D) images. See e.g.~\cite{MegFra03} for the imaging techniques.

The microscopy images were obtained during the gastrula period, which is an early stage in the animal's developmental process and takes place approximately five to ten hours post fertilisation. In short, the fish forms on the surface of a spherical-shaped yolk, which itself deforms over time. Detailed explanations and numerous illustrations can be found in~\cite{KimBalKimUllSchi95}. For the biological methods such as the fluorescence marker and the embryos used in this work we refer to~\cite{MizVerHeaKurKik08}.

The captured area is approximately $540 \times 490 \times 340\,\mu\text{m}^{3}$ and shows the pole region of the yolk. Figure~\ref{fig:embryo}, left column, depict two frames of the raw data. The sequence contains $77$ frames recorded in intervals of $240\,\text{s}$ with clearly visible cellular movements and cell divisions. The spatial resolution of the data is $512 \times 512 \times 44$\,voxels. Intensities are in the range $[0,1]$. In the following we denote by
\begin{equation*}
	\bar{F}^{\delta} \in [0,1]^{77 \times 512 \times 512 \times 44}
\end{equation*}
the unprocessed microscopy data approximating $\bar F$ from Sec.~\ref{sec:motivation}.

The important aspect about endodermal cells is that they are known to form a monolayer during gastrulation~\cite{WarNus99}. In other words, the radial extent is only a single cell. This crucial fact allows for the straightforward extraction of a surface together with an image of the stained cells. Figure~\ref{fig:embryo} illustrates the idea for two particular frames.

\begin{figure*}
	\centering
\includegraphics[width=0.19\textwidth]{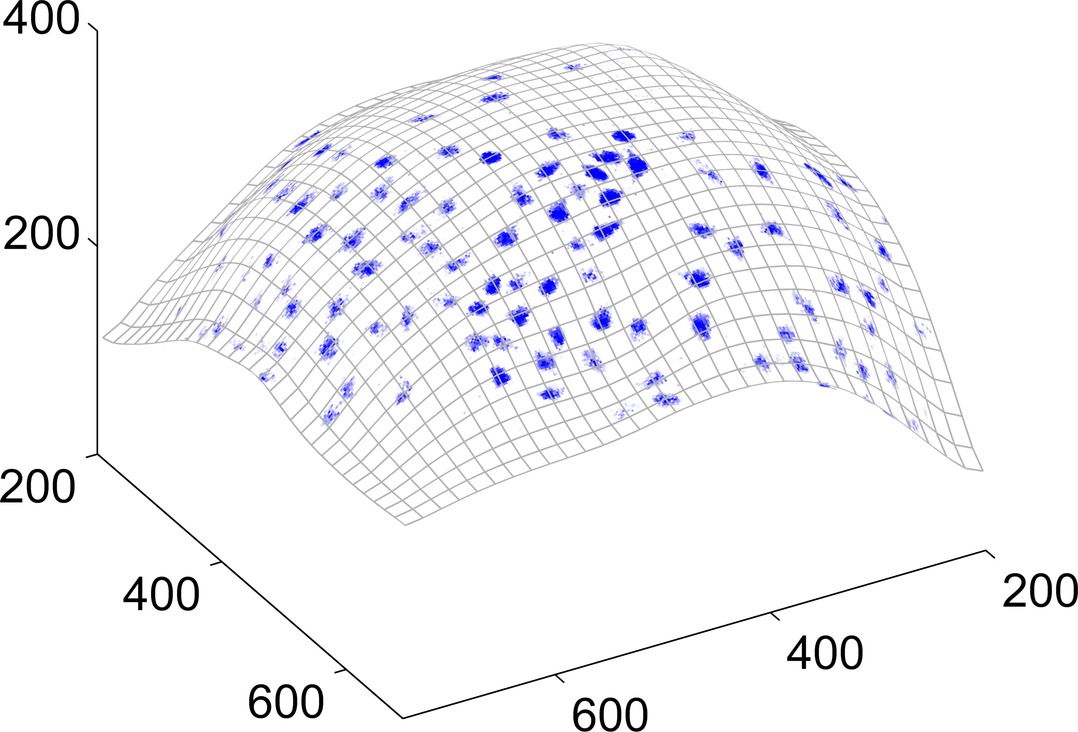}
	\hfill
\includegraphics[width=0.19\textwidth]{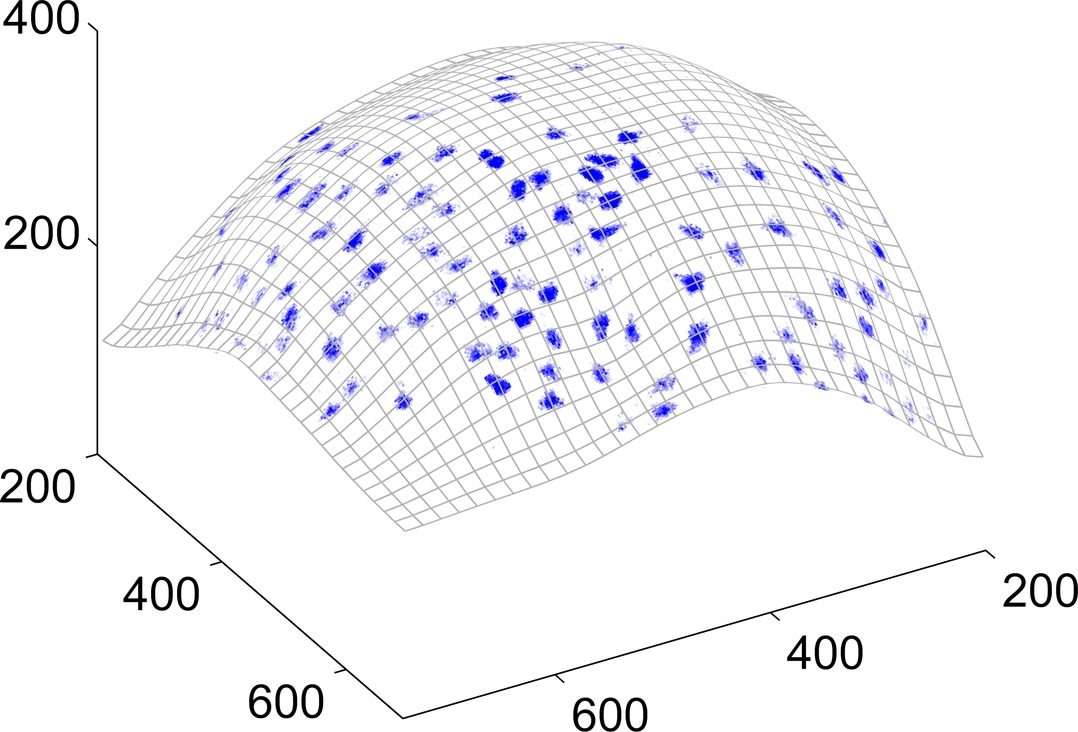}
	\hfill
\includegraphics[width=0.19\textwidth]{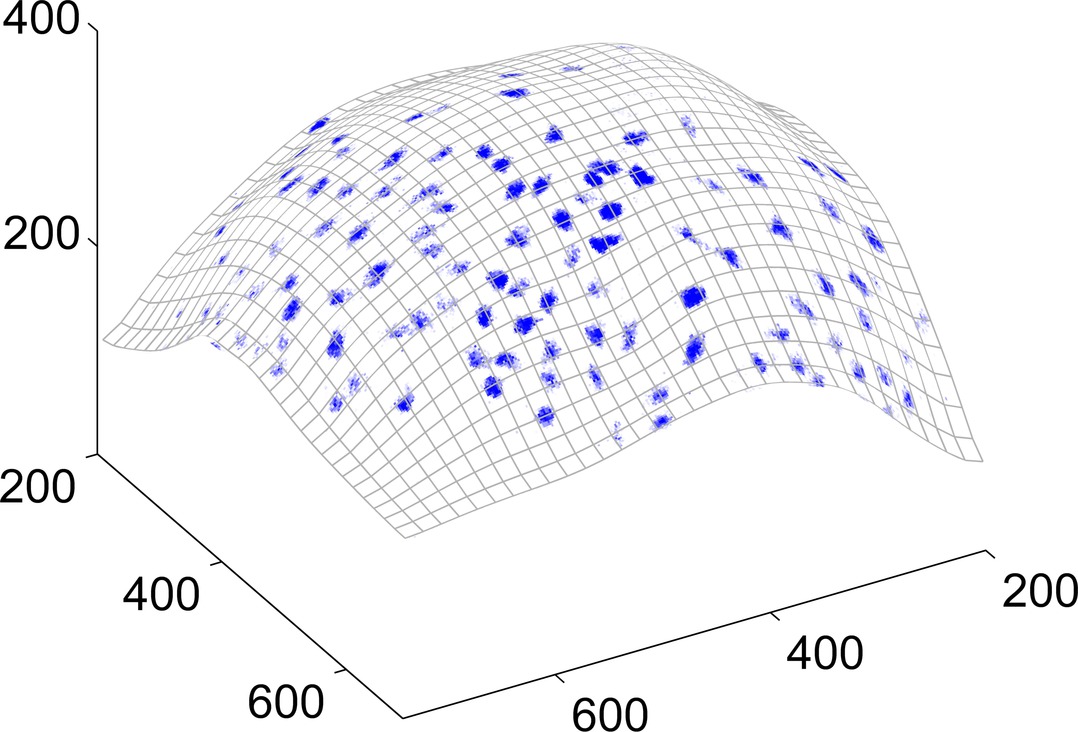}
	\hfill	
\includegraphics[width=0.19\textwidth]{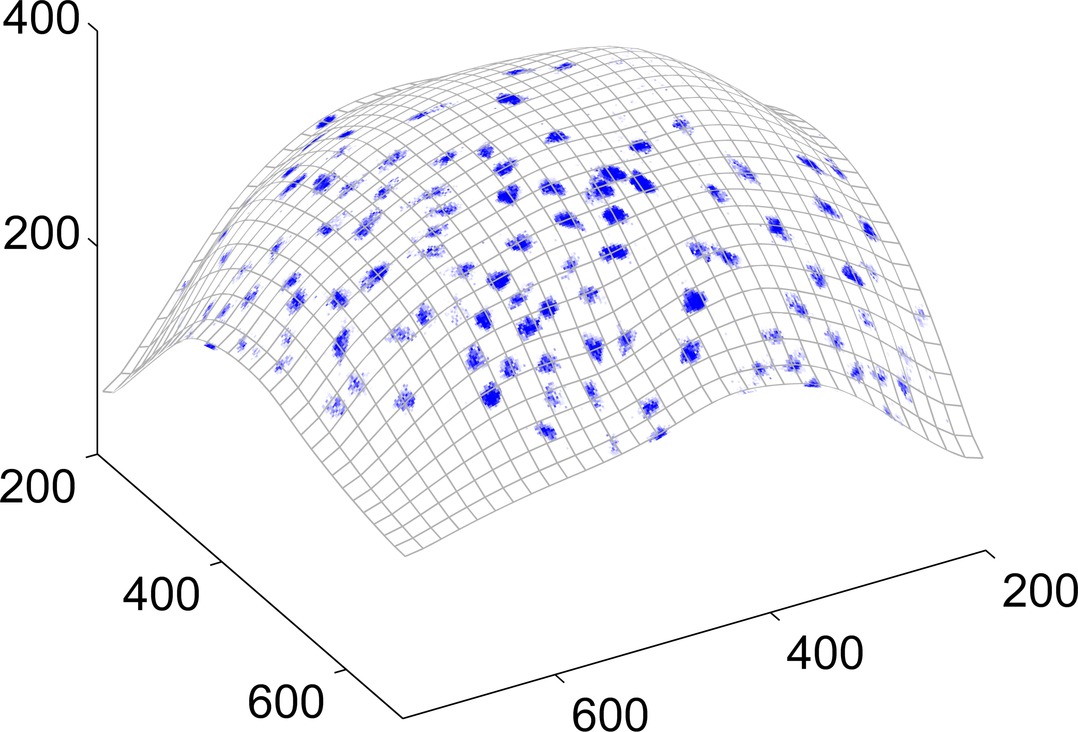}
	\hfill	
\includegraphics[width=0.19\textwidth]{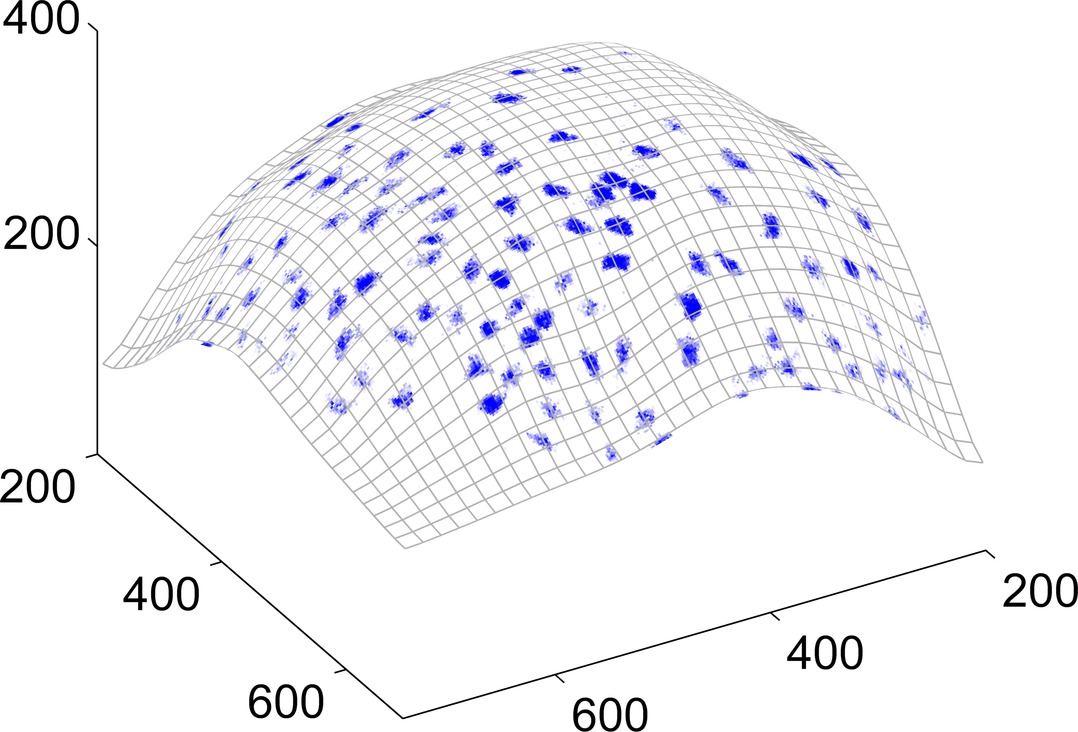} \\
\includegraphics[width=0.19\textwidth]{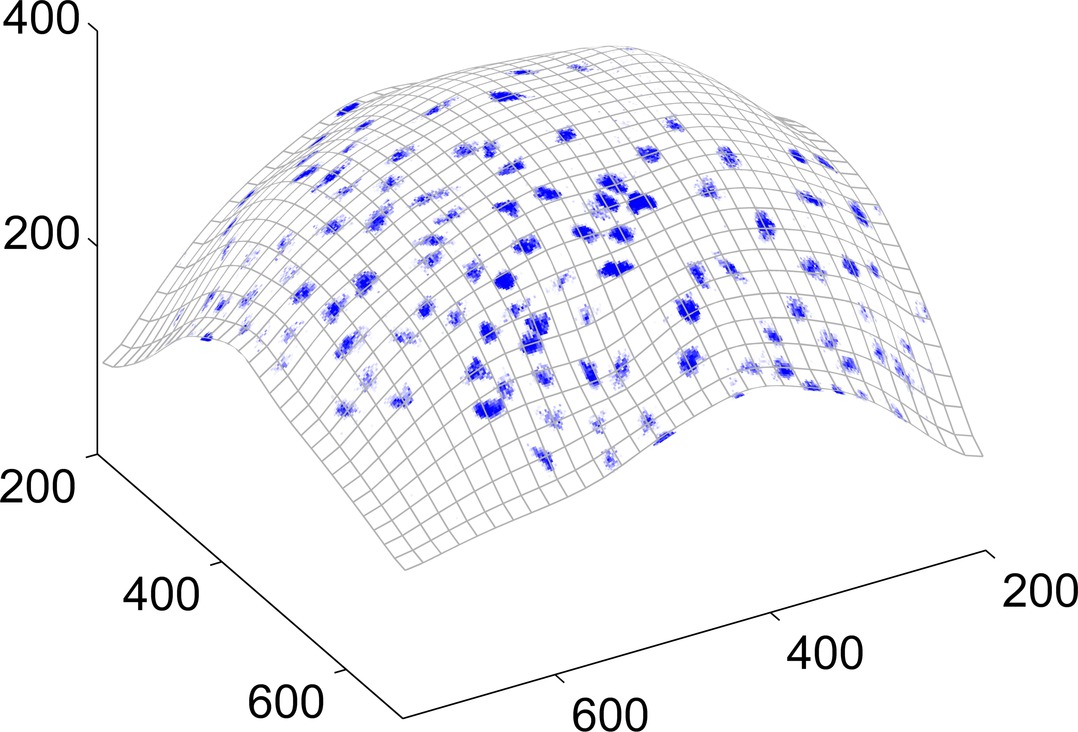}
	\hfill
\includegraphics[width=0.19\textwidth]{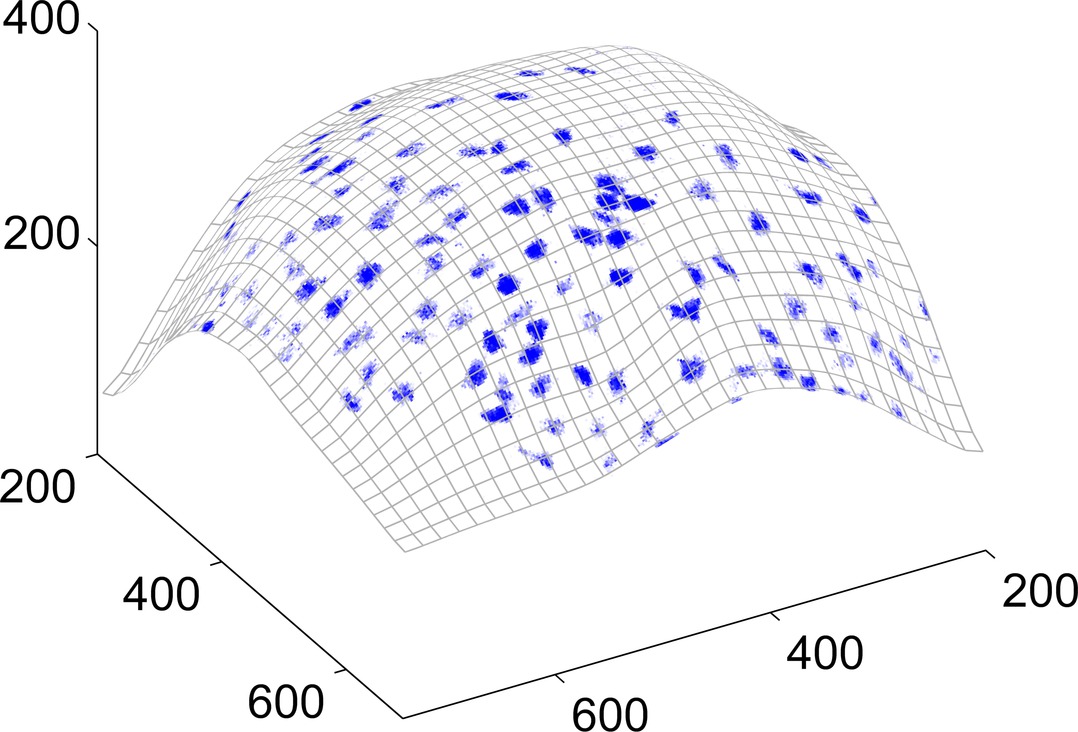}
	\hfill
\includegraphics[width=0.19\textwidth]{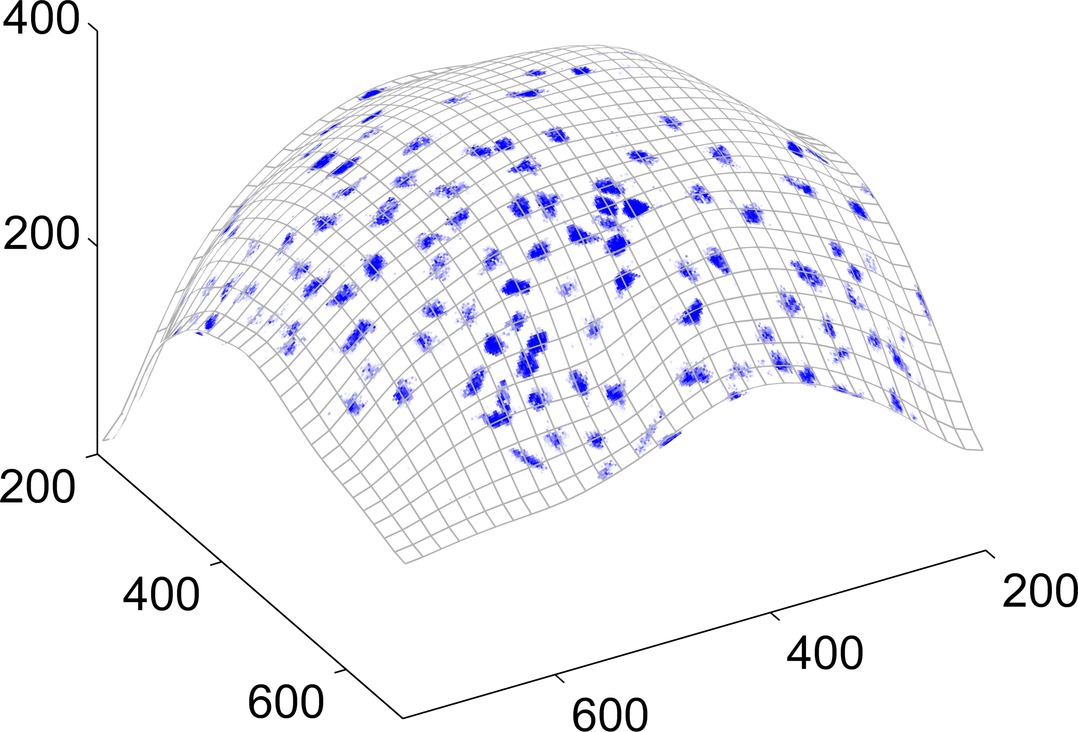}
	\hfill	
\includegraphics[width=0.19\textwidth]{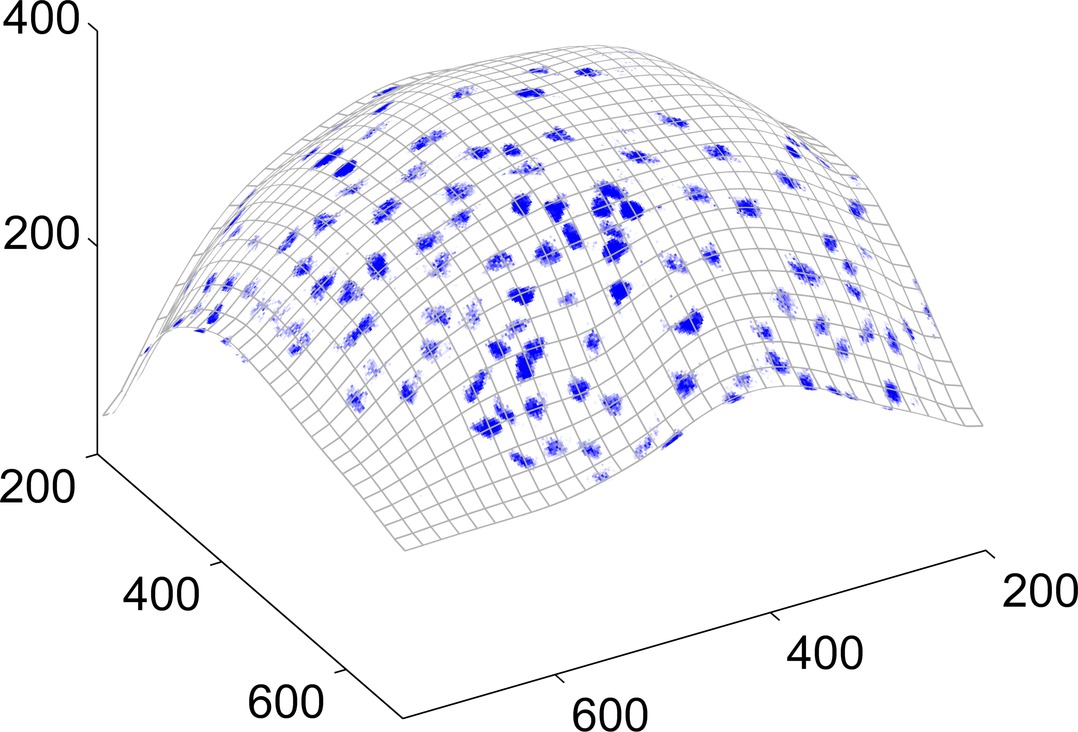}
	\hfill	
\includegraphics[width=0.19\textwidth]{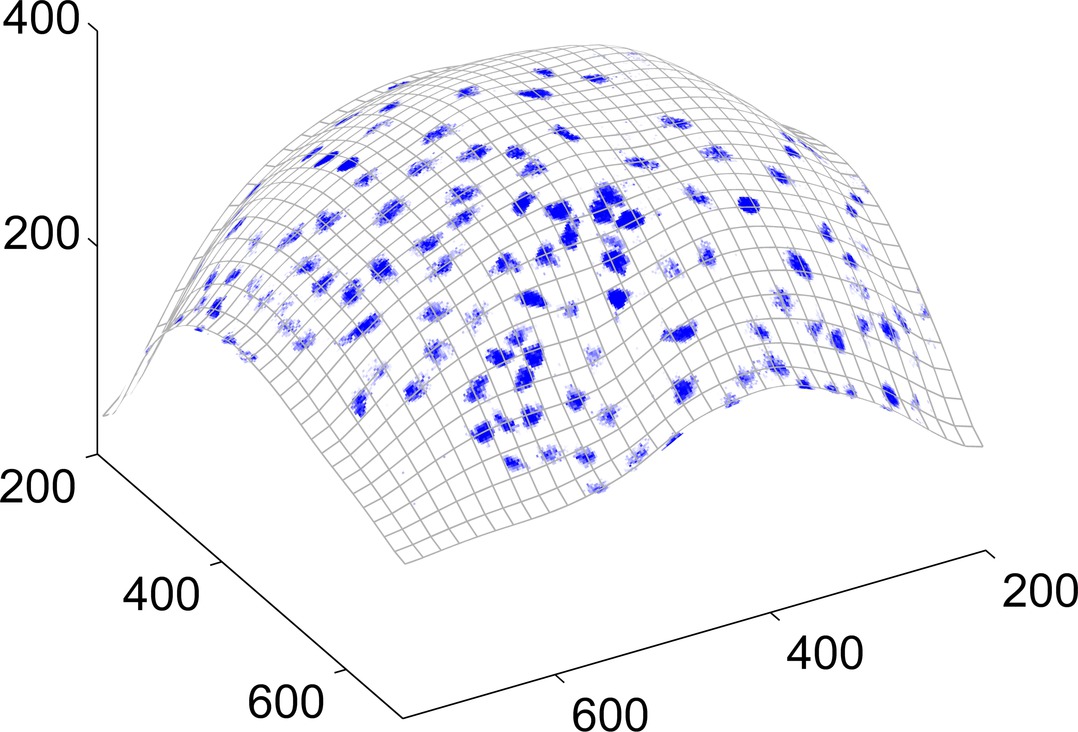} \\
\includegraphics[width=0.19\textwidth]{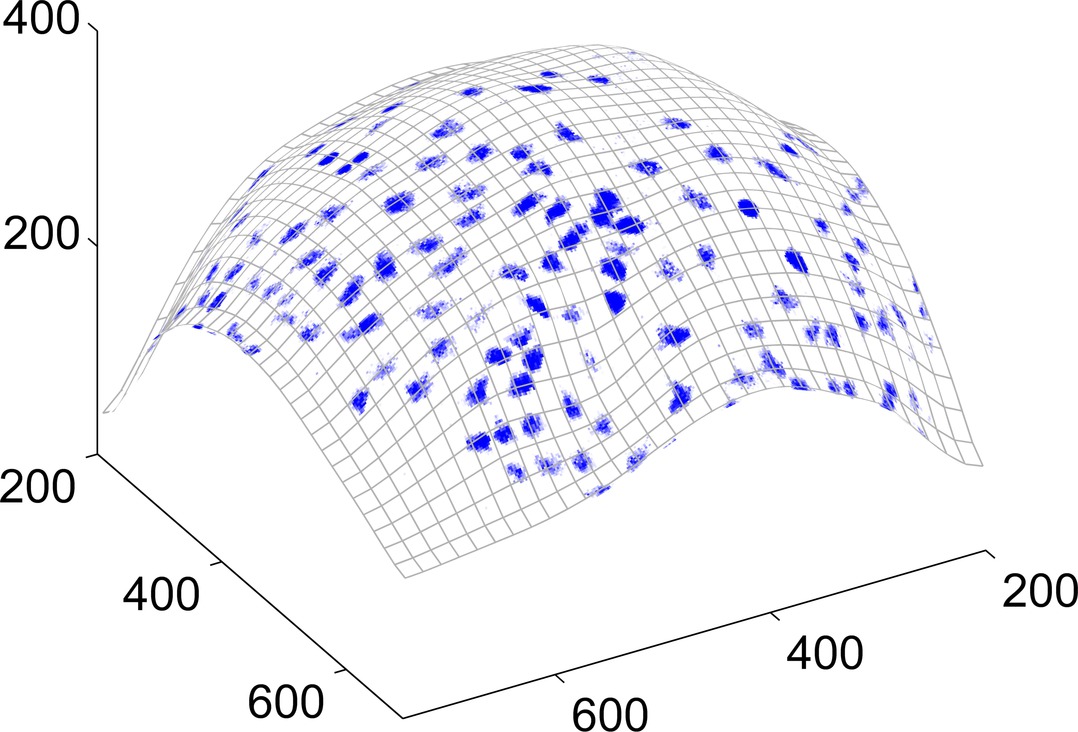}
	\hfill
\includegraphics[width=0.19\textwidth]{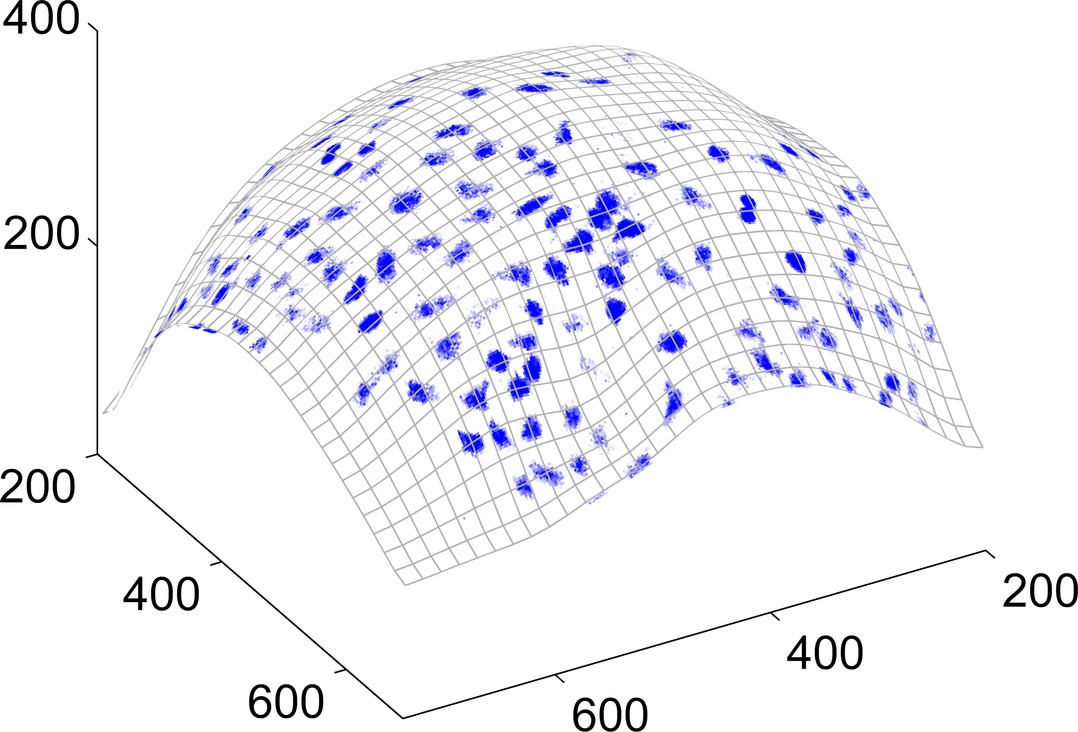}
	\hfill
\includegraphics[width=0.19\textwidth]{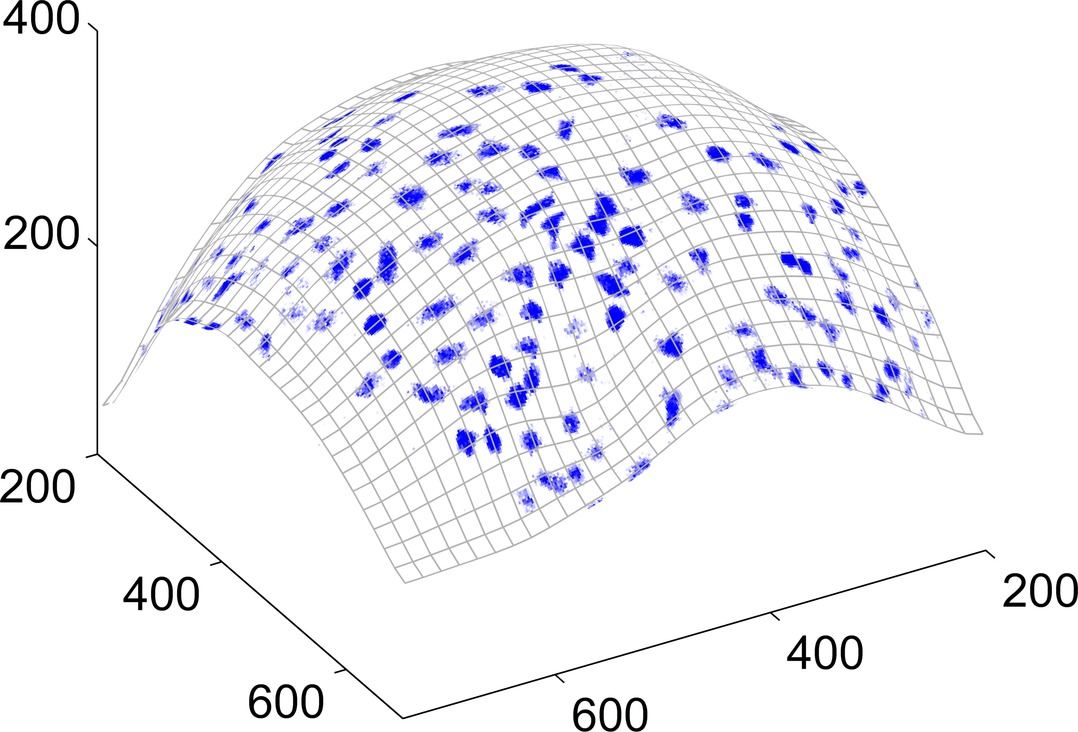}
	\hfill	
\includegraphics[width=0.19\textwidth]{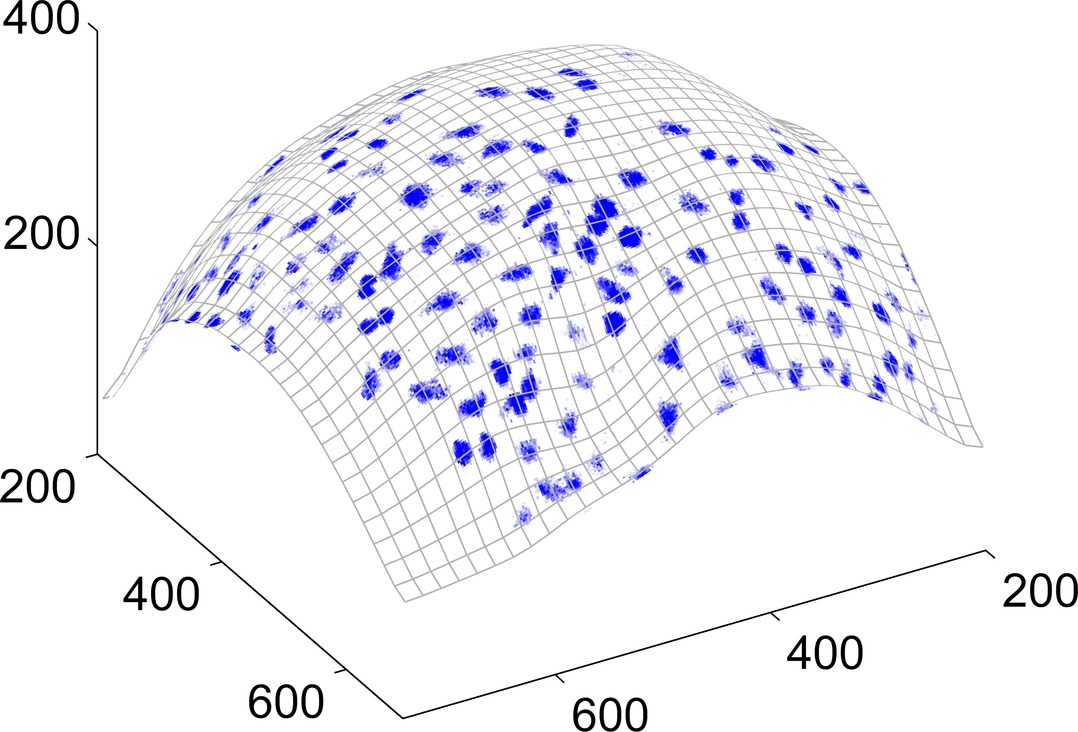}
	\hfill	
\includegraphics[width=0.19\textwidth]{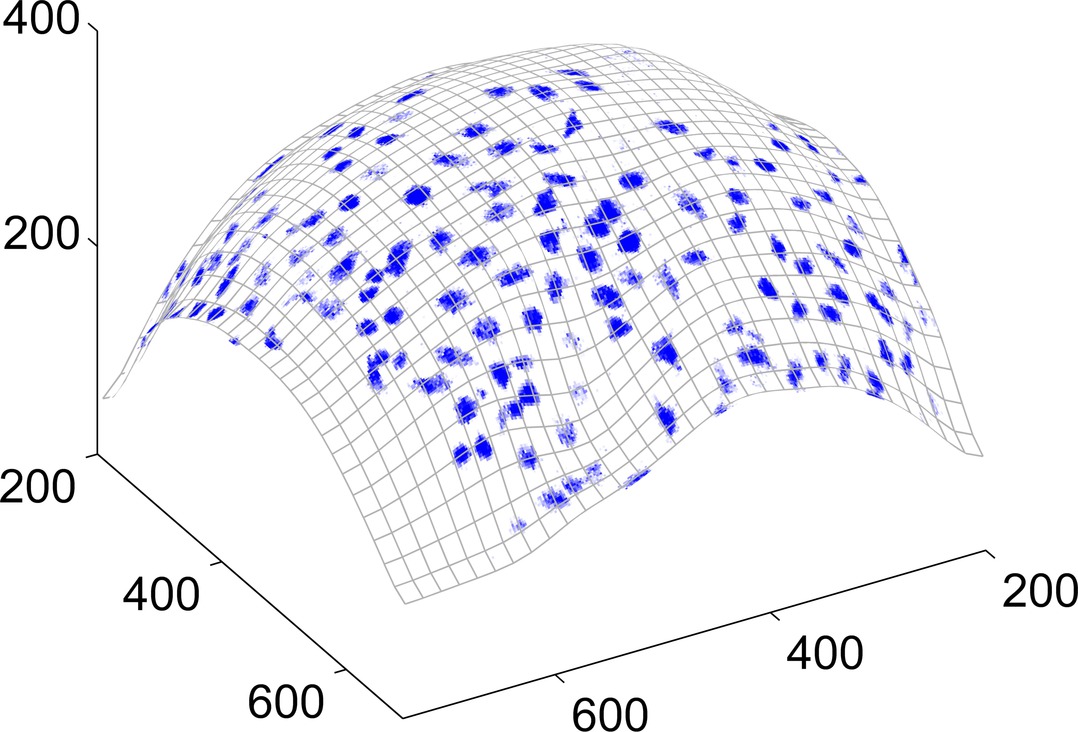}
	\caption{Sequence of embryonic zebrafish images. Depicted are frames no.\ 46 to 60 of the entire sequence (aligned left to right, top to bottom).}
	\label{fig:sequence}
\end{figure*}

\subsection{Preprocessing and Acquisition of Surface Data} \label{sec:parametrisation}

In this section, we relate the mathematical concepts introduced in Sec.~\ref{sec:preliminaries} to the 4D microscopic images. We give a concrete global parametrisation suitable for this type of data and discuss the necessary preprocessing steps leading to an approximation of the evolving surface $\bar{\mathcal{M}}$ together with an approximation of the scalar quantity $F$.

The first step is to extract approximate cell centres from the raw microscopy data. As the positions of cells are characterised by local maxima in intensity they can be reliably obtained as follows. For every frame, a Gaussian filter is applied to the volumetric data $\bar{F}^{\delta}$. Then, local maxima with respect to intensity are computed and treated as cell centres whenever they exceed a certain threshold.

Next we fit a surface to the cell positions. For every frame, this is done by least squares fitting of a piecewise linear function combined with first-order regularisation. From that we get a height field $z^\delta \in \mathbb{R} ^{77 \times 512 \times 512}$ which completely describes the discrete evolving surface. Finally, the numerical approximation $f^\delta$ of $f$ is calculated by linear interpolation of $\bar{F}^\delta$ and evaluation at surface points determined by $z^\delta$.

The combination of all processing steps described above turns the original 4D array $\bar{F}^\delta$ into two three-dimensional arrays
\begin{align*}
	f^{\delta} &\in [0,1]^{77 \times 512 \times 512}, \\
	z^{\delta} &\in \mathbb{R} ^{77 \times 512 \times 512}.
\end{align*}
Figure~\ref{fig:embryo}, right column, illustrates both surfaces and the obtained images for two particular frames. In Fig.~\ref{fig:sequence}, a segment of the sequence is shown.

Let us quickly relate $z^\delta$ to the quantities introduced in Sec.~\ref{sec:evolsurf}. The mapping
\begin{equation*}
(t, \xi^1, \xi^2) \mapsto (\xi^1, \xi^2, z^\delta(t, \xi^1, \xi^2)),
\end{equation*}
where $(t, \xi^1, \xi^2)$ ranges over a ${77 \times 512 \times 512}$ grid, is the discrete parametrisation. The corresponding $\phi$ is the function that identifies surface points with identical $(\xi^1, \xi^2)$ coordinates. Thus, the surface motion $\vec{V}$ occurs only in direction of $x^{3}$. However, we stress that this particular parametrisation was chosen due to the lack of knowledge about the true motion of material points on the surface.

\subsection{Solving for the Velocity Fields} \label{sec:steps}

After the preprocessing of the microscopy data as explained above, the following steps lead to the desired solution:
\begin{enumerate}
	\item From the parametrisation compute approximations of $\partial_i \vec{x}$, $g$, $\Gamma_{ij}^k$, $\alpha_i^j$, $\tilde \Gamma_{\mu j}^k$ as explained in Sec.~\ref{sec:preliminaries}. Like all other quantities the $\alpha_i^j$ are functions of space and time. They can be computed, for example, by Gram-Schmidt orthonormalisation of the coordinate basis $\{\partial_1 \vec{x}, \partial_2 \vec{x}\}$.
	\item Discretise optimality system \eqref{eq:optsys2} as described in Sec.~\ref{sec:discretisation}.
	\item Compute coefficients \eqref{eq:coefficients} of discretised optimality system from the quantities calculated in step 1.
	\item Solve resulting linear system for unknowns $w$, see Sec.~\ref{sec:numresults}.
	\item Compute relative tangential velocity $\vec{u}$ via~\eqref{eq:unknown} and recover total velocity $\vec{m} = \vec{u} + \vec{V}$.
	\item Finally, cell trajectories can be approximated by computing the integral curves of $\vec{m}$, see \eqref{eq:integralcurve} in Sec.~\ref{sec:numresults}.
\end{enumerate}

\subsection{Visualisation}

In order to illustrate the computed tangential velocity fields we apply the standard flow colour-coding from \cite{BakSchaLewRotBla11}.\footnote{Some figures may appear in colour only in the online version.} This coding turns $\mathbb{R}^2$ vector fields into colour images according to a particular 2D colour space.

However, we are interested in visualising tangential vector fields on an embedded manifold, which are functions with values in $\mathbb{R}^3$. To be able to apply the colour-coding mentioned above we turn the computed optical flow fields $\vec{u}$ into $\mathbb{R}^2$ vector fields in the following way
\begin{equation}
	\vec{u} \mapsto \frac{\abs{\vec{u}}}{\abs{\mathrm{P}_{x^{3}} \vec{u}}} \mathrm{P}_{x^{3}} \vec{u},
	\label{eq:scaledprojection}
\end{equation}
where $\mathrm{P}_{x^{3}}: (x^{1}, x^{2}, x^{3}) \mapsto (x^{1}, x^{2}, 0)$ is the orthogonal projection onto the $x^1$-$x^2$ plane. If the scaling factor $\frac{\abs{\vec{u}}}{\abs{\mathrm{P}_{x^{3}} \vec{u}}}$ were omitted, the new vector field would simply be the original one as viewed from above. The reason for including this scaling are vectors having a large $x^{3}$-component, which would otherwise seem unnaturally short. Finally, the image resulting from the colour-coding of vector field~\eqref{eq:scaledprojection} is painted back on the surface. Figure~\ref{fig:flow} illustrates the colour-coded tangential vector field $\vec{u}$ and the colour space. In all figures the surface is slightly smoothed for better visual effect.

\begin{figure*}
	\centering
	\includegraphics[width=0.45\textwidth]{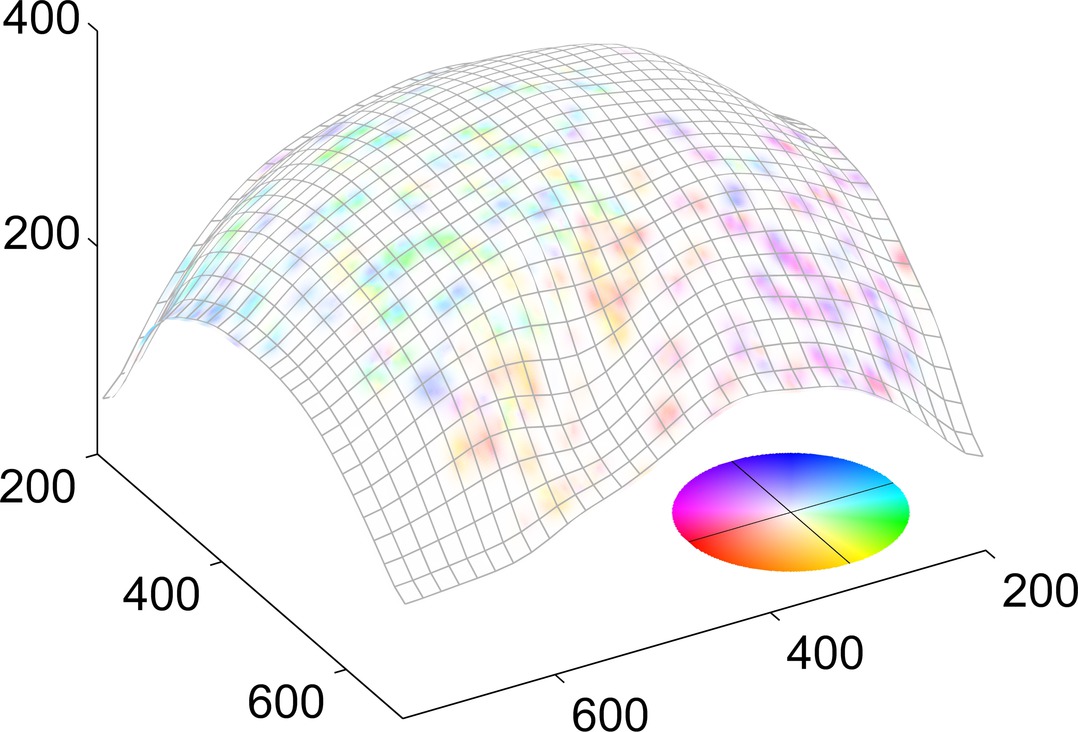}
	\hfill
	\includegraphics[width=0.40\textwidth]{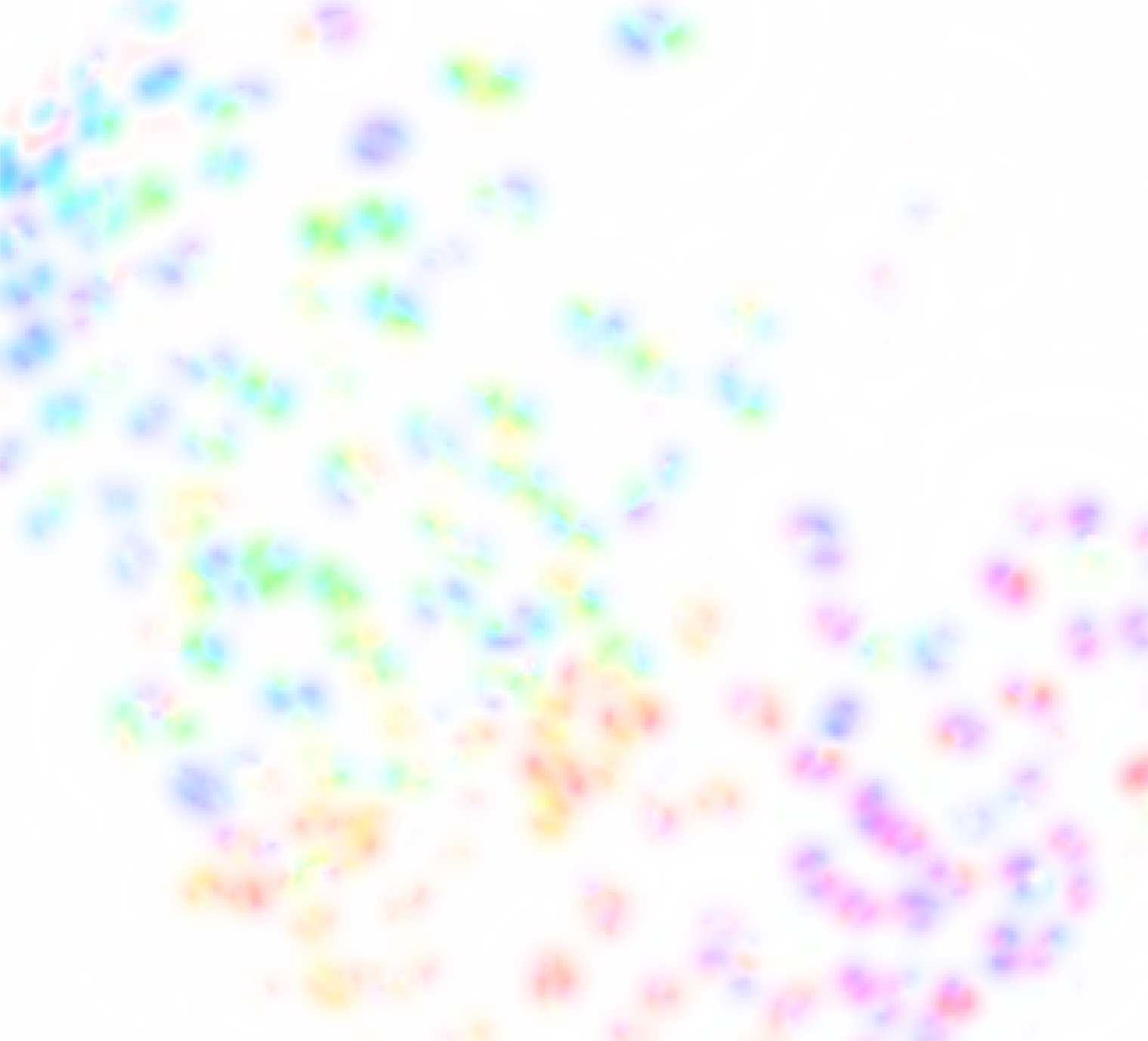}
	\caption{Optical flow field between frames 60 and 61 of the sequence. Colours indicate direction whereas darkness of a colour indicates the length of the vector. Note that the colour circle has been enlarged for better visibility. Parameters are $\lambda_{0} = c/100$ and $\lambda_{1} = \lambda_{2} = c$, where $c \coloneqq 0.5$.}
	\label{fig:flow}
\end{figure*}

\subsection{Numerical results} \label{sec:numresults}

We conducted four experiments with different parameter settings and minimised functional~\eqref{eq:functional1} as outlined in Sec.~\ref{sec:steps}. Due to a low cell density near the boundaries we only worked with a part of the whole dataset. The grid dimensions were $( N_0, N_1, N_2) = (30, 370, 370)$. Accordingly, grid spacing was set to $h_{\sigma} = 1/N_{\sigma}$. Our implementation was done in Matlab and all experiments were performed on an Intel Xeon E5-1620 3.6GHz workstation with 128GB RAM. We used the Generalized Minimal Residual Method (GMRES) to solve the resulting linear system. As a termination criterion we chose a relative residual of $0.02$ and a maximum number of 2000 iterations with a restart every 30 iterations. The resulting runtime was approximately two hours. In Table~\ref{tab:runtimes}, the parameters for all experiments are listed, and the resulting running times and relative residuals are given. Implementation and data are available on our website.\footnote{\url{http://www.csc.univie.ac.at}}

\begin{table}
	\centering
	\begin{tabular}{c|c|c|c|c|c}
		No. & $\lambda_{0}$ & $\lambda_{1} = \lambda_{2}$ & Runtime & Rel. residual \\
		\hline
		1 & $c$ & $c$ & $2.05\, \mathrm{h}$ & $0.075$ \\
		2 & $c/10$ & $c$ & $2.07\, \mathrm{h}$ & $0.086$ \\
		3 & $c/100$ & $c$ & $2.09\, \mathrm{h}$ & $0.103$ & \\
		4 & $c/100$ & $c/10$ & $2.14\, \mathrm{h}$ & $0.016$
	\end{tabular}
\caption{Runtimes and relative residuals of the experiments. For convenience, we define $c \coloneqq 0.5$.}
\label{tab:runtimes}
\end{table}

\paragraph{Regularisation.} In a first experiment, we compared different regularisation parameters. They were chosen such that individual movements of cells are well preserved and the velocity field is sufficiently homogeneous both in time and space. Figure~\ref{fig:reg} depicts these results. A visual inspection of the dataset shows that cells tend to move towards the embryo's body axis which roughly runs from the bottom left to the top right corner in Fig.~\ref{fig:flow}, right. This behaviour is clearly visible from the obtained velocity fields. In Fig.~\ref{fig:flowsequence}, we show the optical flow field for the sequence depicted in Fig.~\ref{fig:sequence}.

\begin{figure*}
	\centering
	\includegraphics[width=0.45\textwidth]{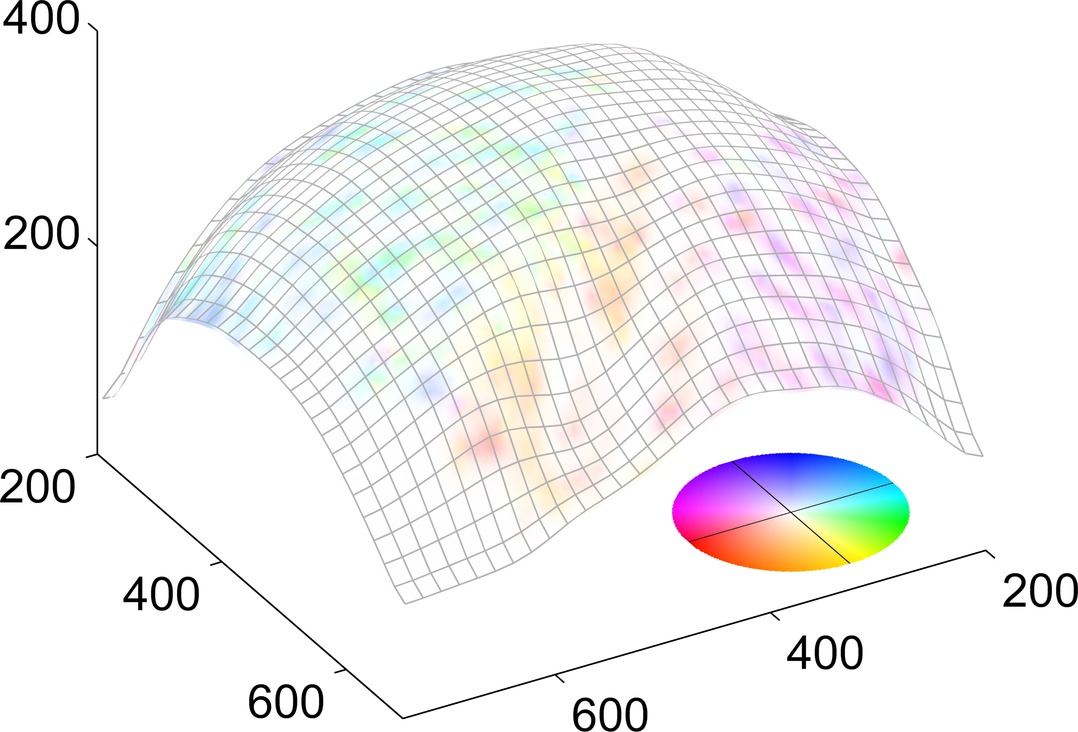}
	\hfill
	\includegraphics[width=0.45\textwidth]{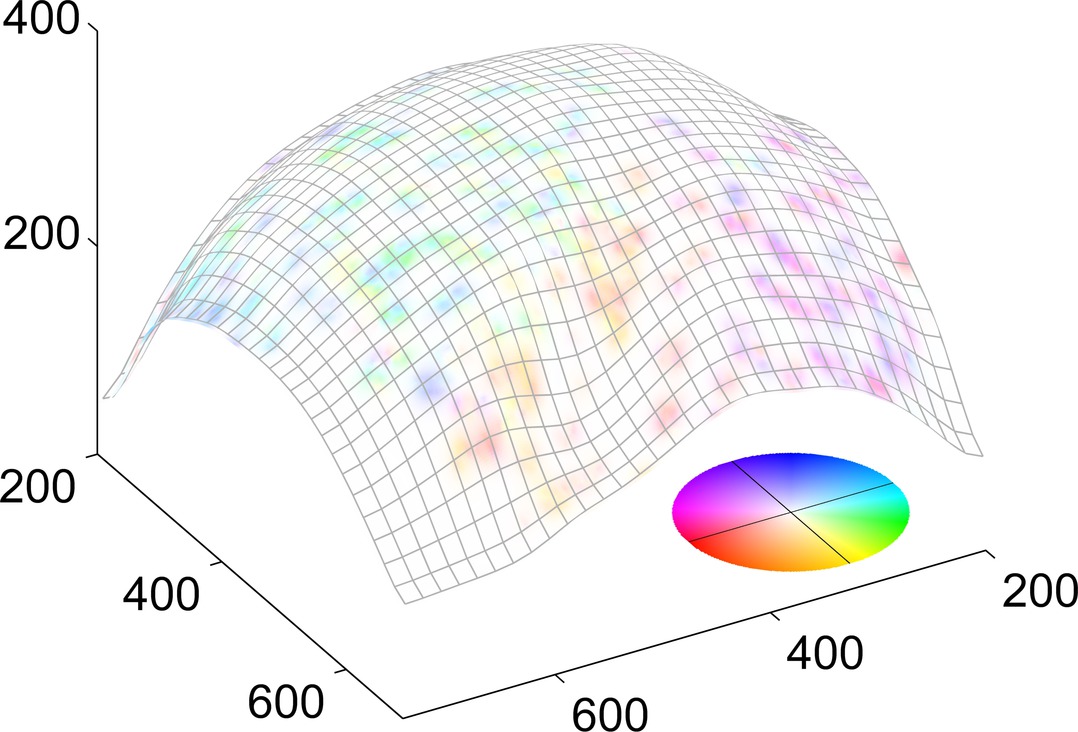}
	\\
	\includegraphics[width=0.45\textwidth]{figures/3-surface-framewise-cxcr4aMO2_290112-frames-120-122}
	\hfill
	\includegraphics[width=0.45\textwidth]{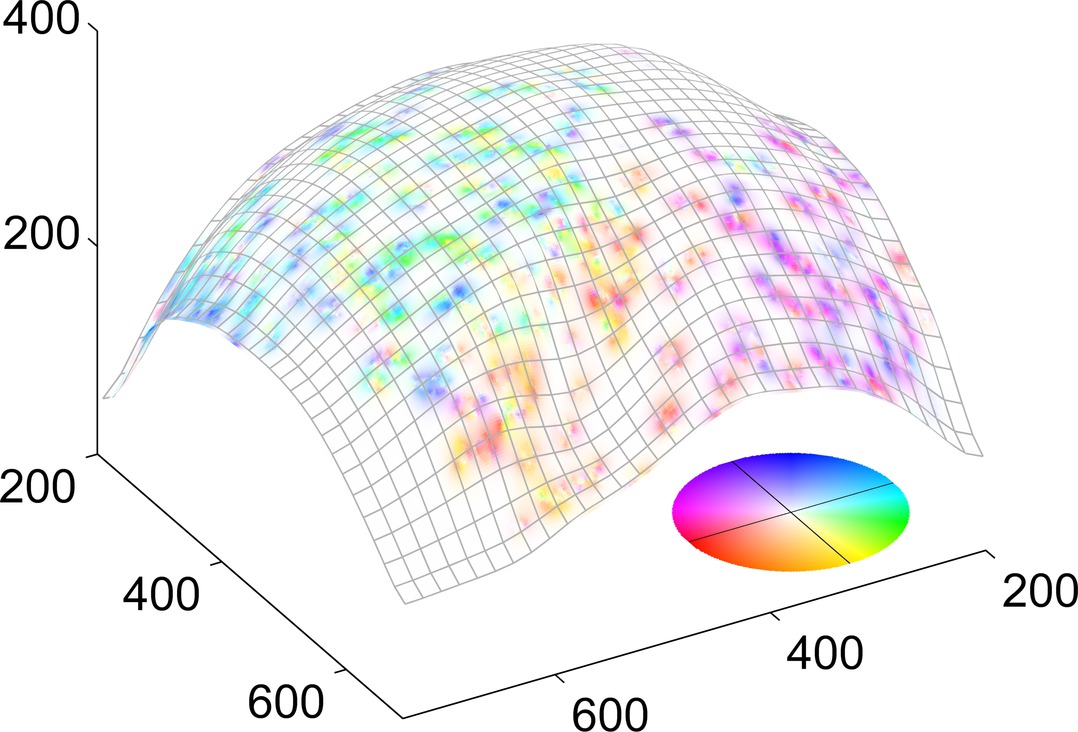}
	\caption{Resulting velocity field $\vec{u}$ between frames 60 and 61 obtained with different regularisation parameters. Denote $c \coloneqq 0.5$. Top left: $\lambda_{0} = \lambda_{1} = \lambda_{2} = c$. Top right: $\lambda_{0} = c/10$ and $\lambda_{1} = \lambda_{2} = c$. Bottom left: $\lambda_{0} = c/100$ and $\lambda_{1} = \lambda_{2} = c$. Bottom right: $\lambda_{0} = c/100$ and $\lambda_{1} = \lambda_{2} = c/10$.}
	\label{fig:reg}
\end{figure*}

\begin{figure*}
	\centering
\includegraphics[width=0.19\textwidth]{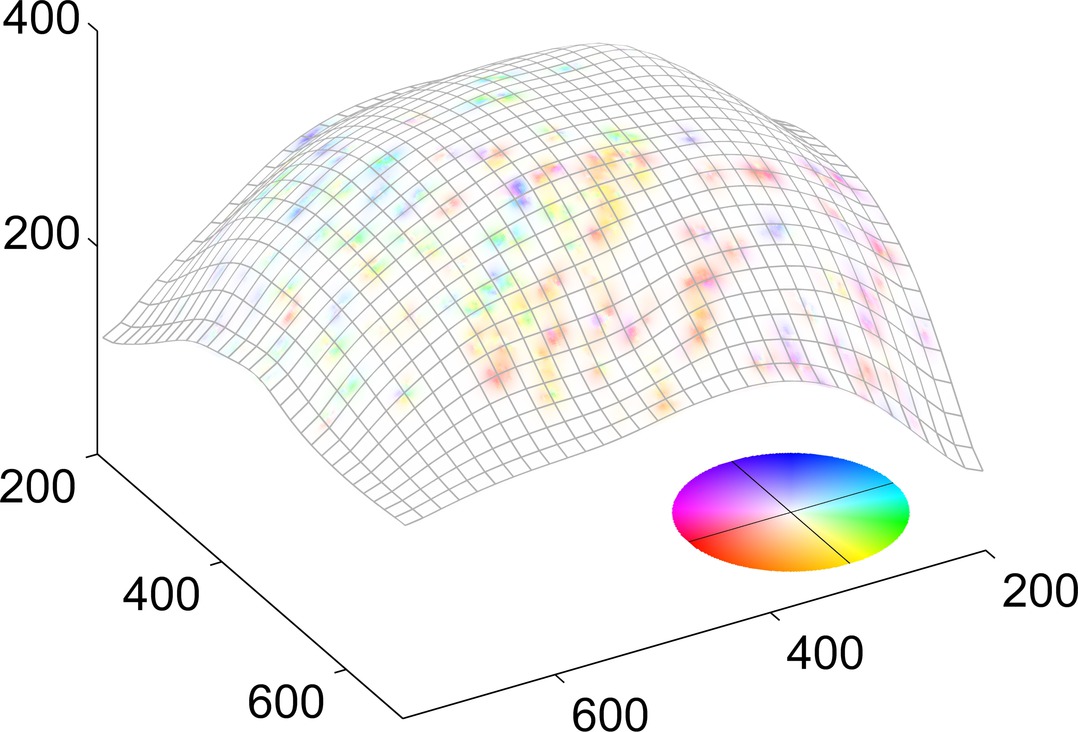}
	\hfill
\includegraphics[width=0.19\textwidth]{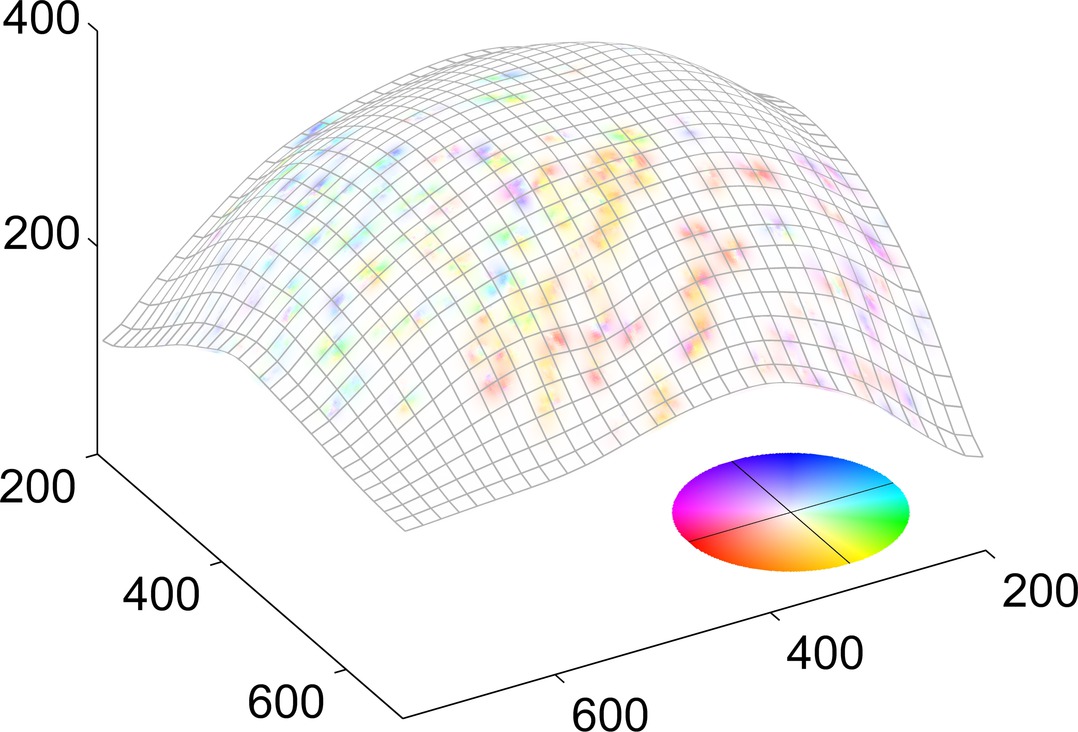}
	\hfill
\includegraphics[width=0.19\textwidth]{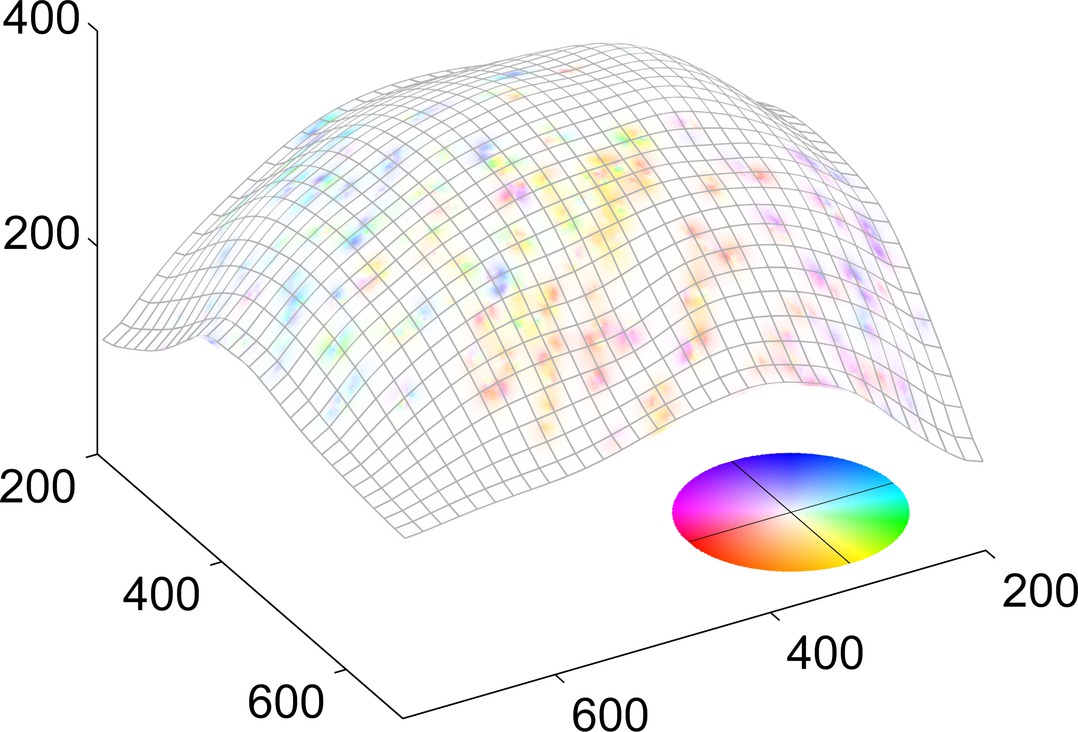}
	\hfill	
\includegraphics[width=0.19\textwidth]{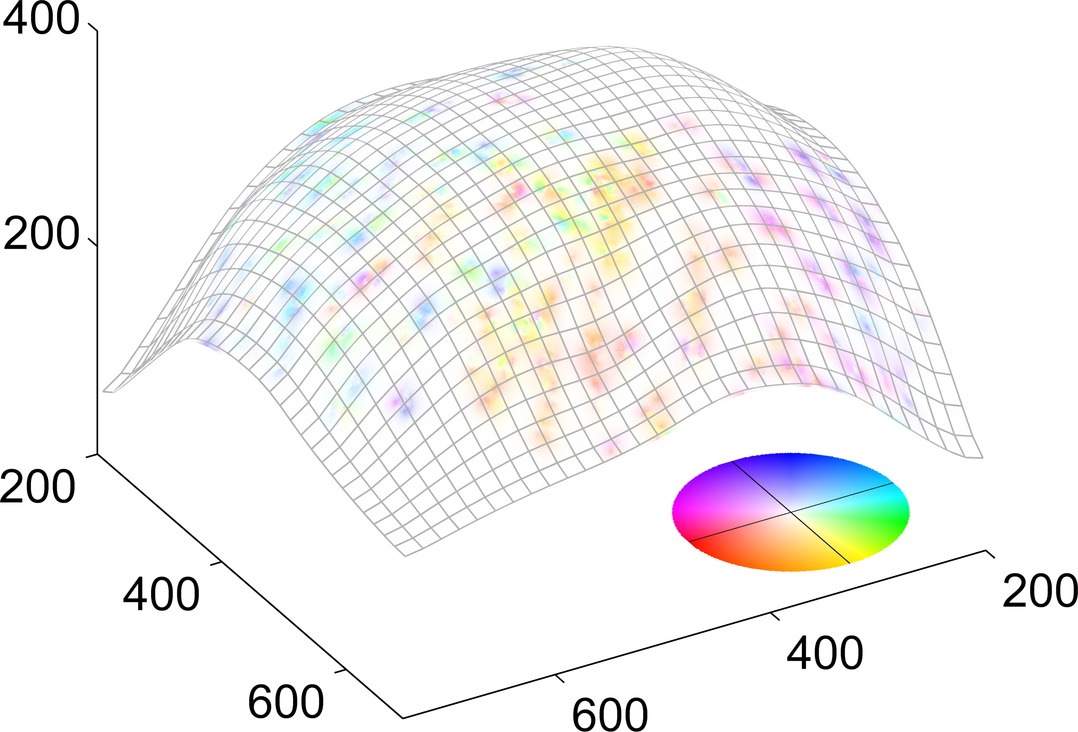}
	\hfill	
\includegraphics[width=0.19\textwidth]{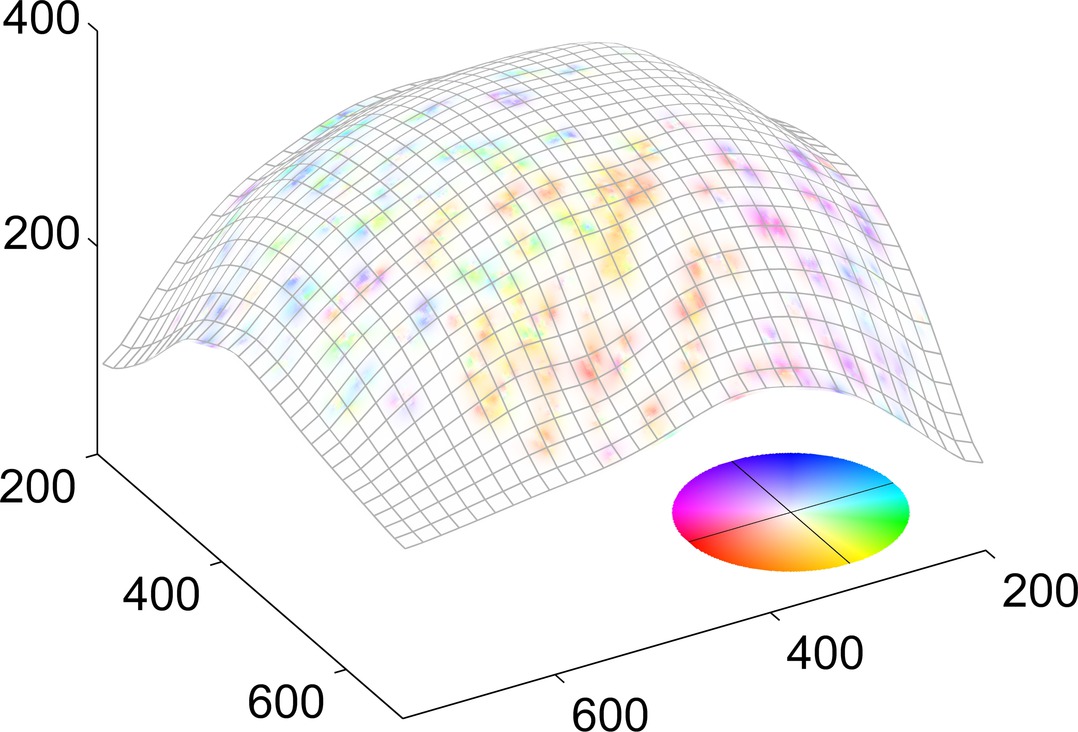} \\
\includegraphics[width=0.19\textwidth]{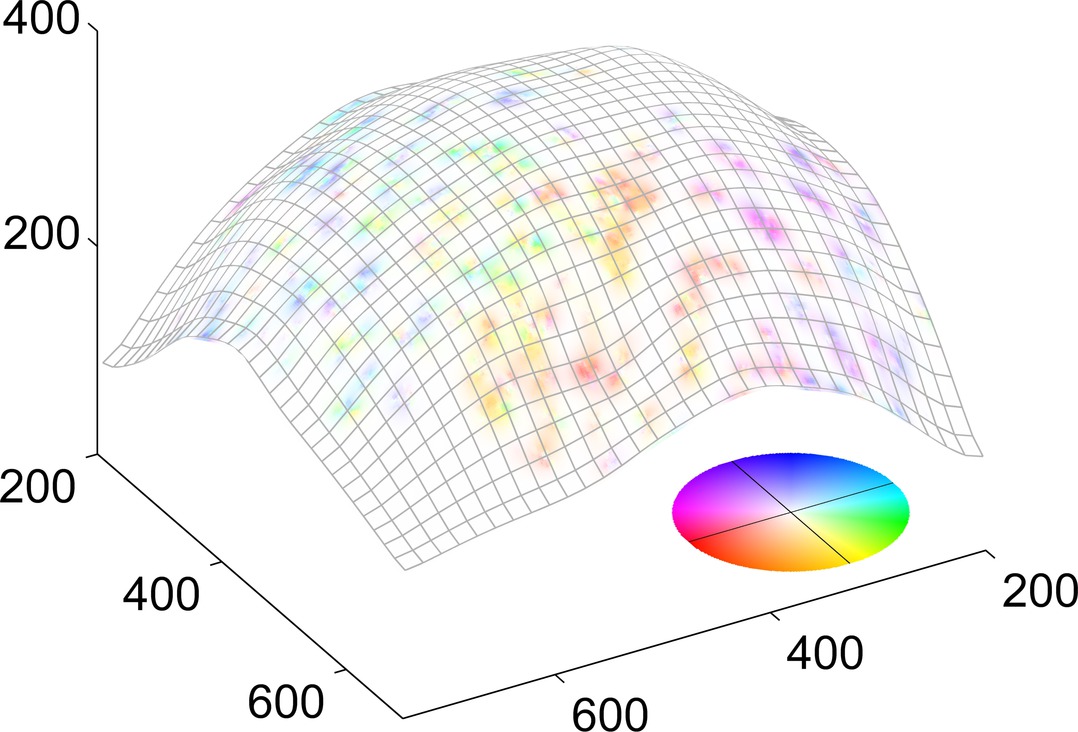}
	\hfill
\includegraphics[width=0.19\textwidth]{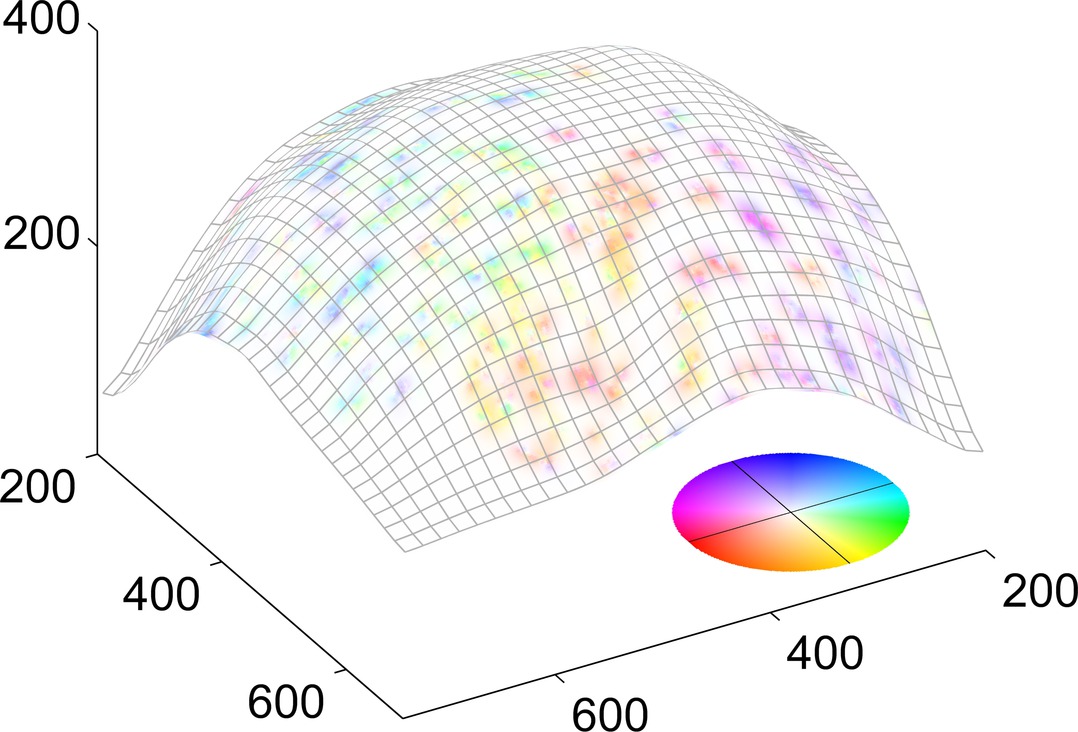}
	\hfill
\includegraphics[width=0.19\textwidth]{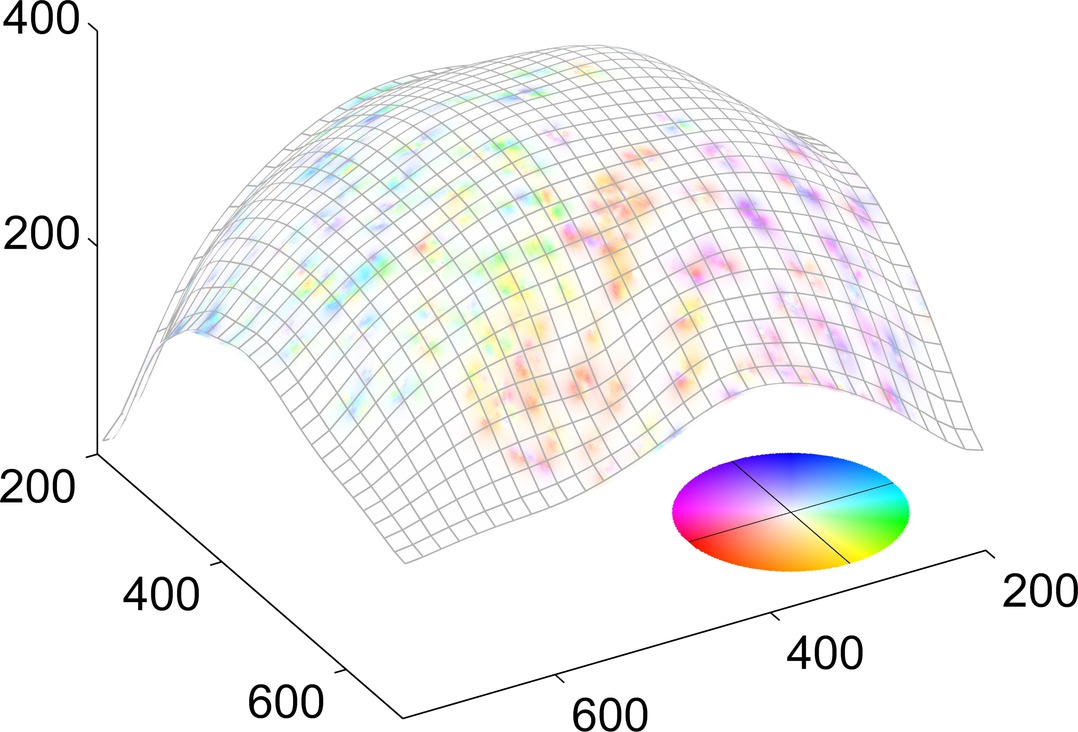}
	\hfill	
\includegraphics[width=0.19\textwidth]{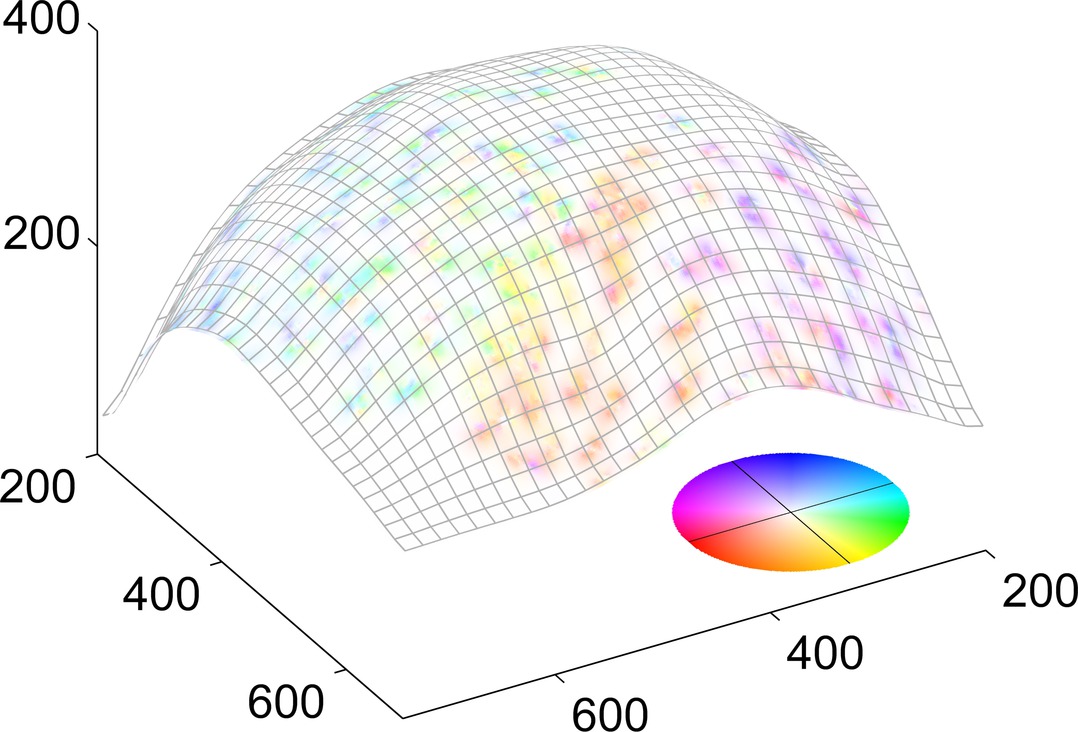}
	\hfill	
\includegraphics[width=0.19\textwidth]{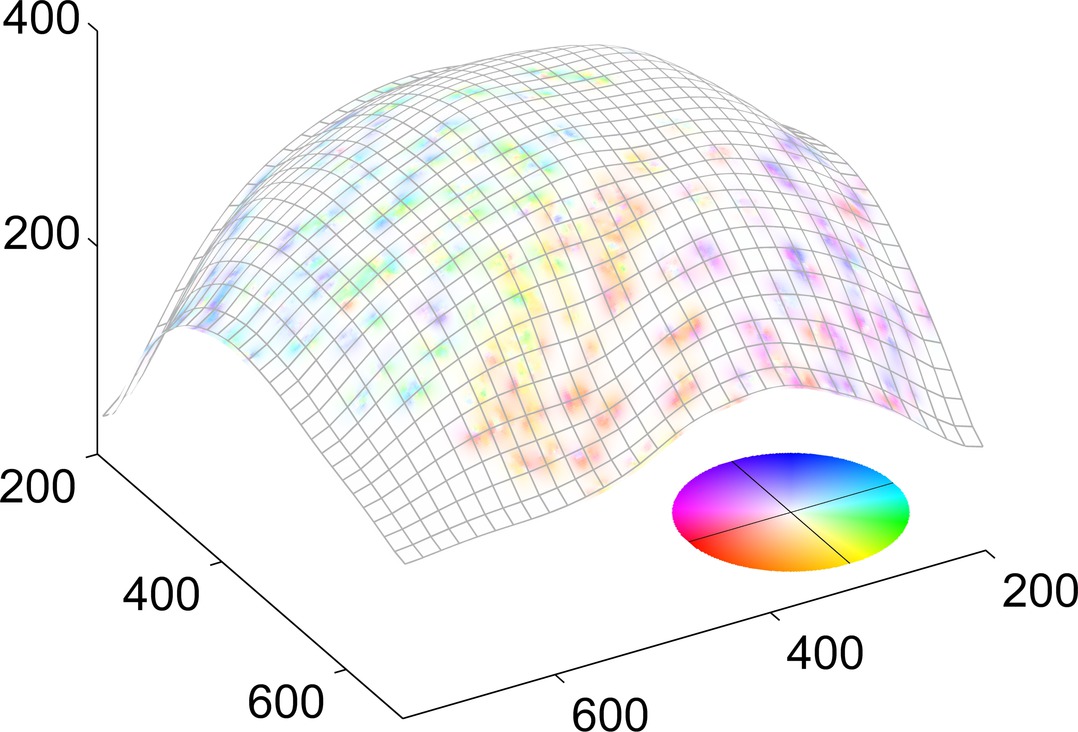} \\
\includegraphics[width=0.19\textwidth]{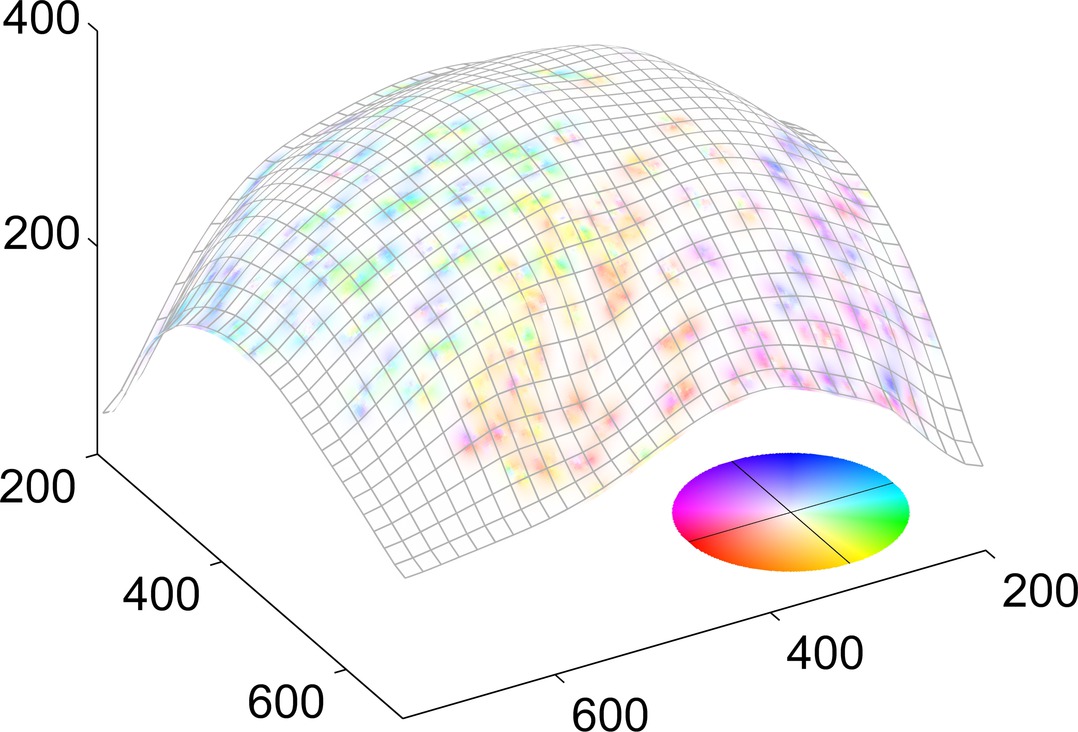}
	\hfill
\includegraphics[width=0.19\textwidth]{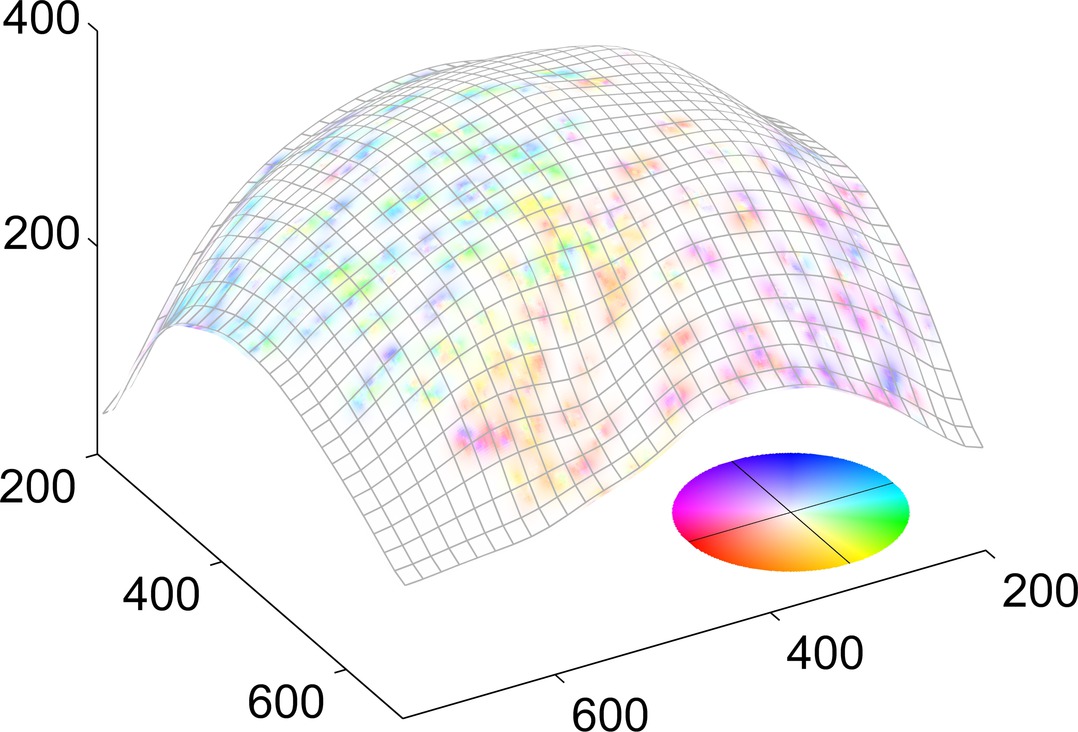}
	\hfill
\includegraphics[width=0.19\textwidth]{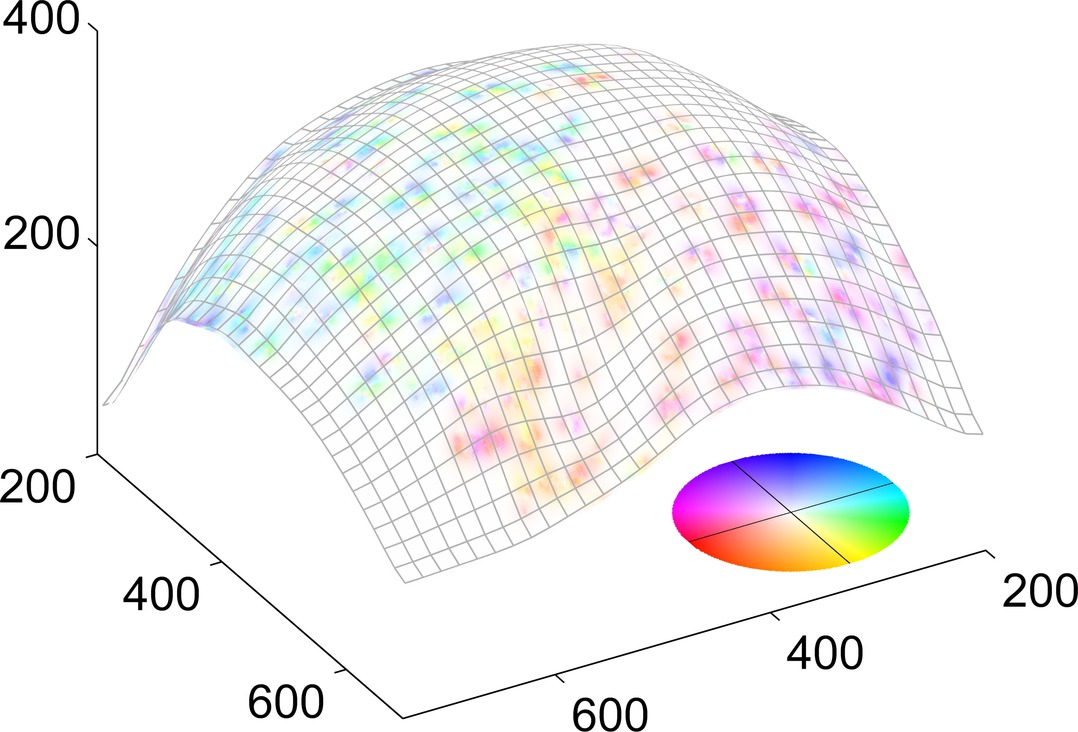}
	\hfill	
\includegraphics[width=0.19\textwidth]{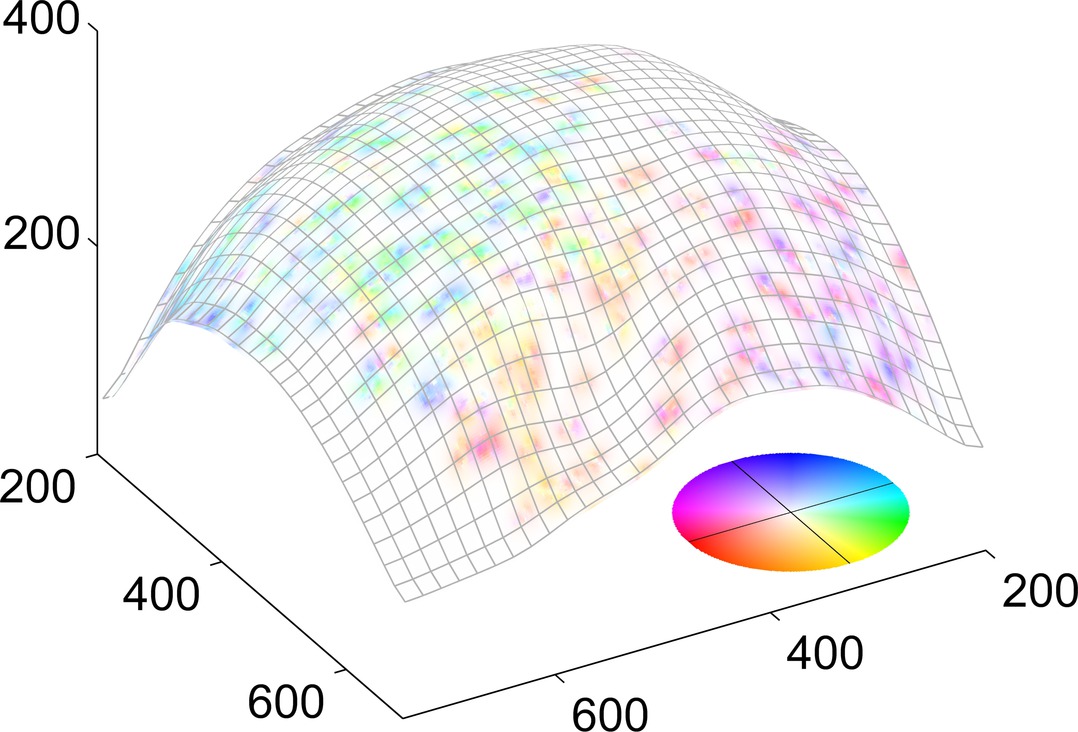}
	\hfill	
\includegraphics[width=0.19\textwidth]{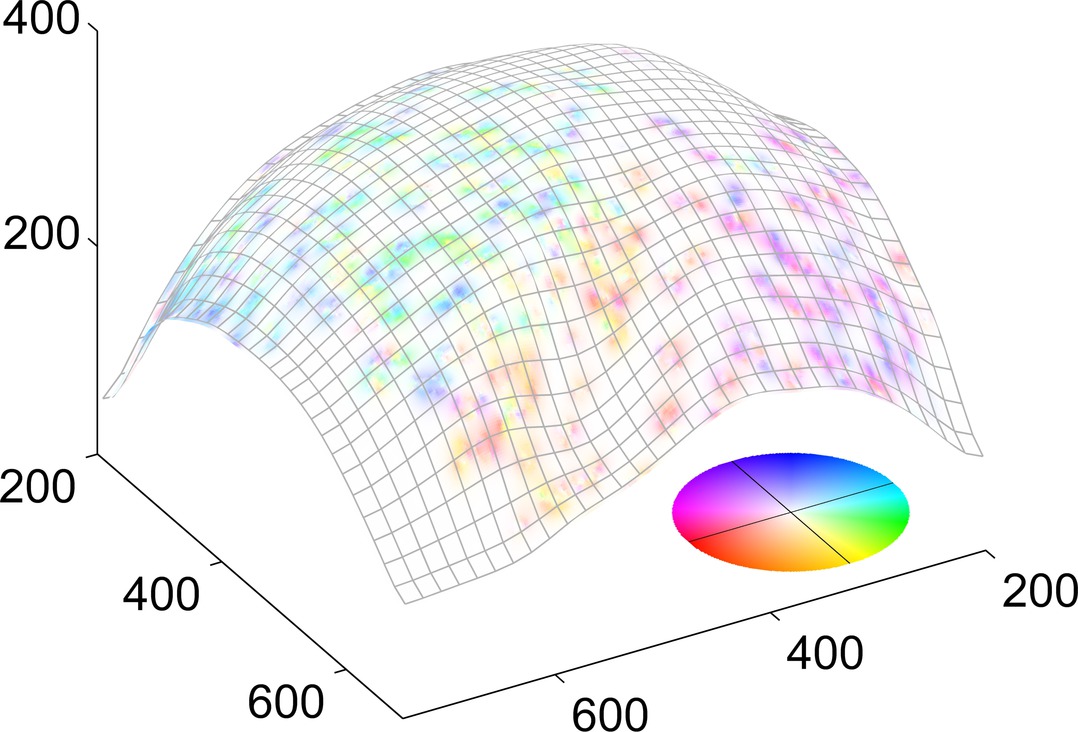}
	\caption{Sequence of colour-coded tangential velocity fields. Depicted are the same frames as in Fig.~\ref{fig:sequence}. Parameters are $\lambda_{0} = c/100$ and $\lambda_{1} = \lambda_{2} = c/10$.}
	\label{fig:flowsequence}
\end{figure*}

\paragraph{Cell Trajectories.} In order to reconstruct the paths travelled by individual cells, we computed the integral curves of $\vec{m}$. By~\eqref{eq:totalmotion}, for every starting point $x_{0} \in \mathcal{M}_{0}$ the trajectory $\gamma(\cdot, x_{0})$ is the solution of the following ordinary differential equation
\begin{equation}
\begin{aligned}
	\partial_{t}\gamma(t, x_{0}) & = \vec{m}(t,\gamma(t, x_{0})), \\
	\gamma(0, x_{0}) & = x_{0},
\end{aligned}
\label{eq:integralcurve}
\end{equation}
where $\vec{m}$ is the total velocity of a cell; cf.~Sec.~\ref{sec:preliminaries}. As discussed in Sec.~\ref{sec:parametrisation}, a local maximum of $F$ at $x_{0} \in \cali{M}_{0}$ indicates the approximate position of a cell. Hence, we chose local maxima as initial values and approximated~\eqref{eq:integralcurve} by solving the projection of
\begin{equation*}
\begin{aligned}
	\hat{\gamma}(t + 1, x_{0}) & = \hat{\gamma}(t, x_{0}) + s \vec{m}(t, \hat{\gamma}(t, x_{0})) \\
	\hat{\gamma}(0, x_{0}) & = x_{0},
\end{aligned}
\end{equation*}
to the $x^{1}$-$x^{2}$-plane, because it allows for a better illustration. The parameter $s$ is a step size and was chosen as $s := 10$. Figure~\ref{fig:integralcurves} shows the projection $\mathrm{P}_{x^3} \hat{\gamma}$ of the computed curves for several values of the regularisation parameters. The effect on the smoothness of the trajectories is clearly visible.

\begin{figure*}[!t]
	\centering
	\includegraphics[width=0.45\textwidth]{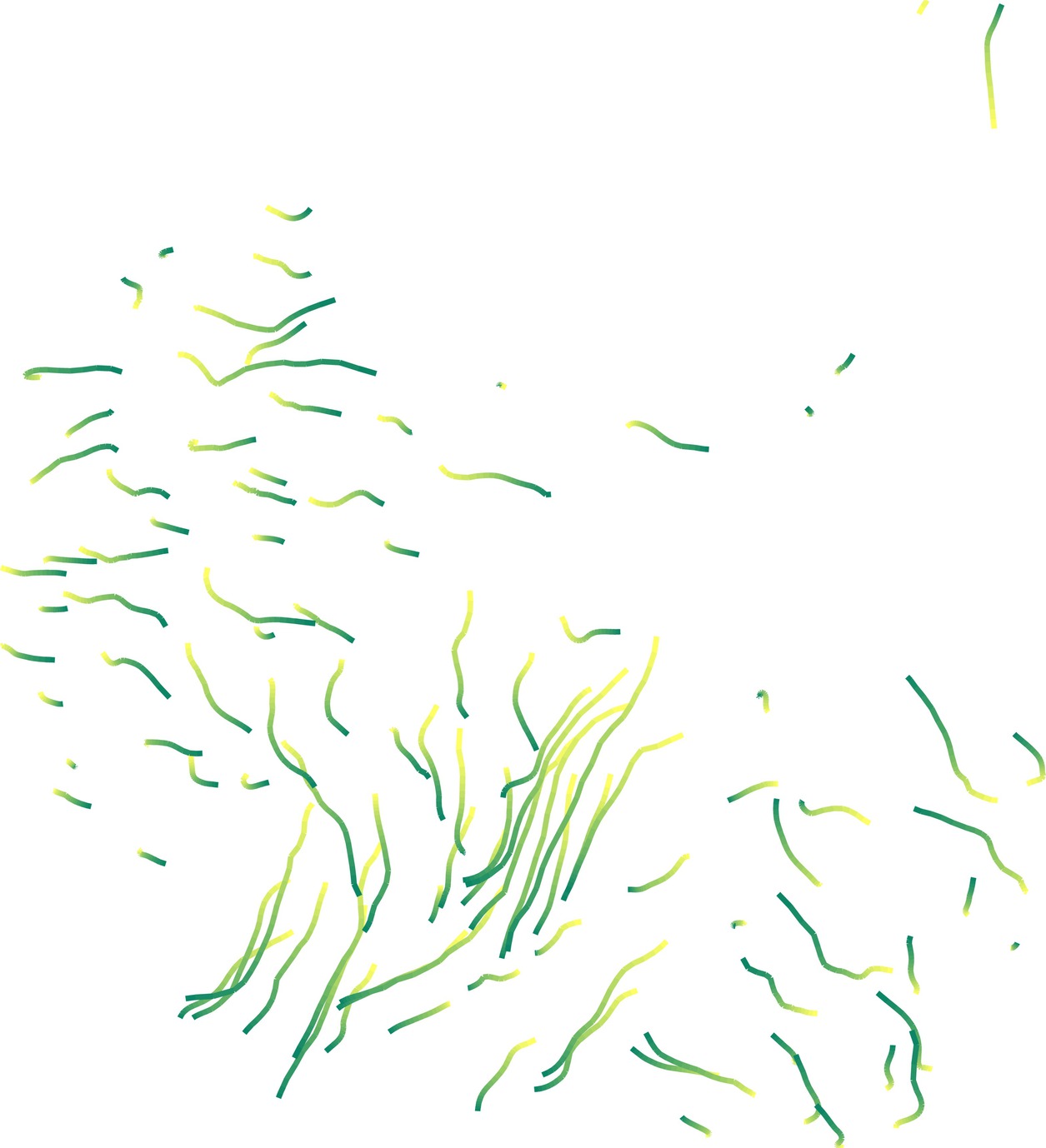}
	\hfill
	\includegraphics[width=0.45\textwidth]{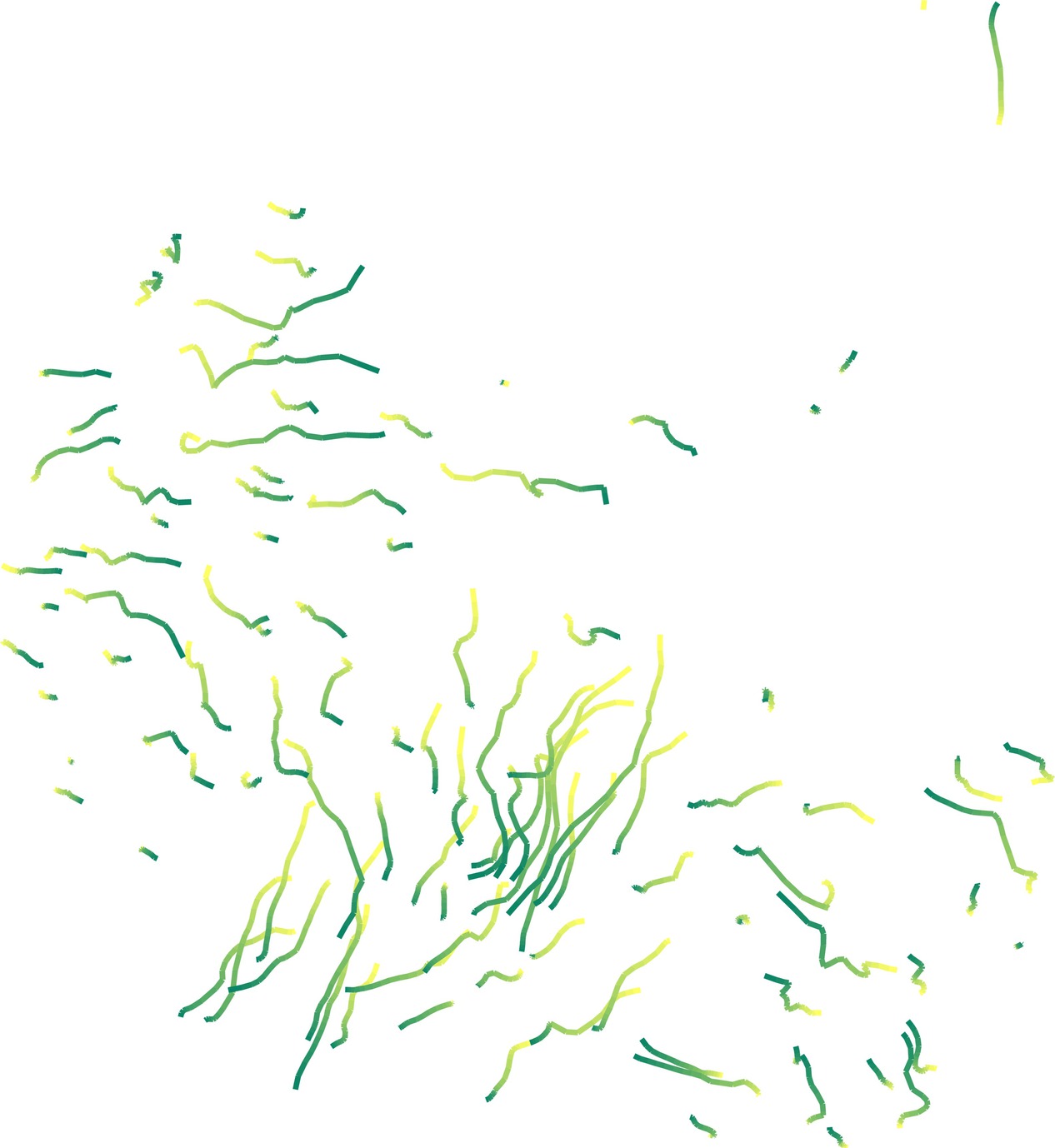}
	\\
	\includegraphics[width=0.45\textwidth]{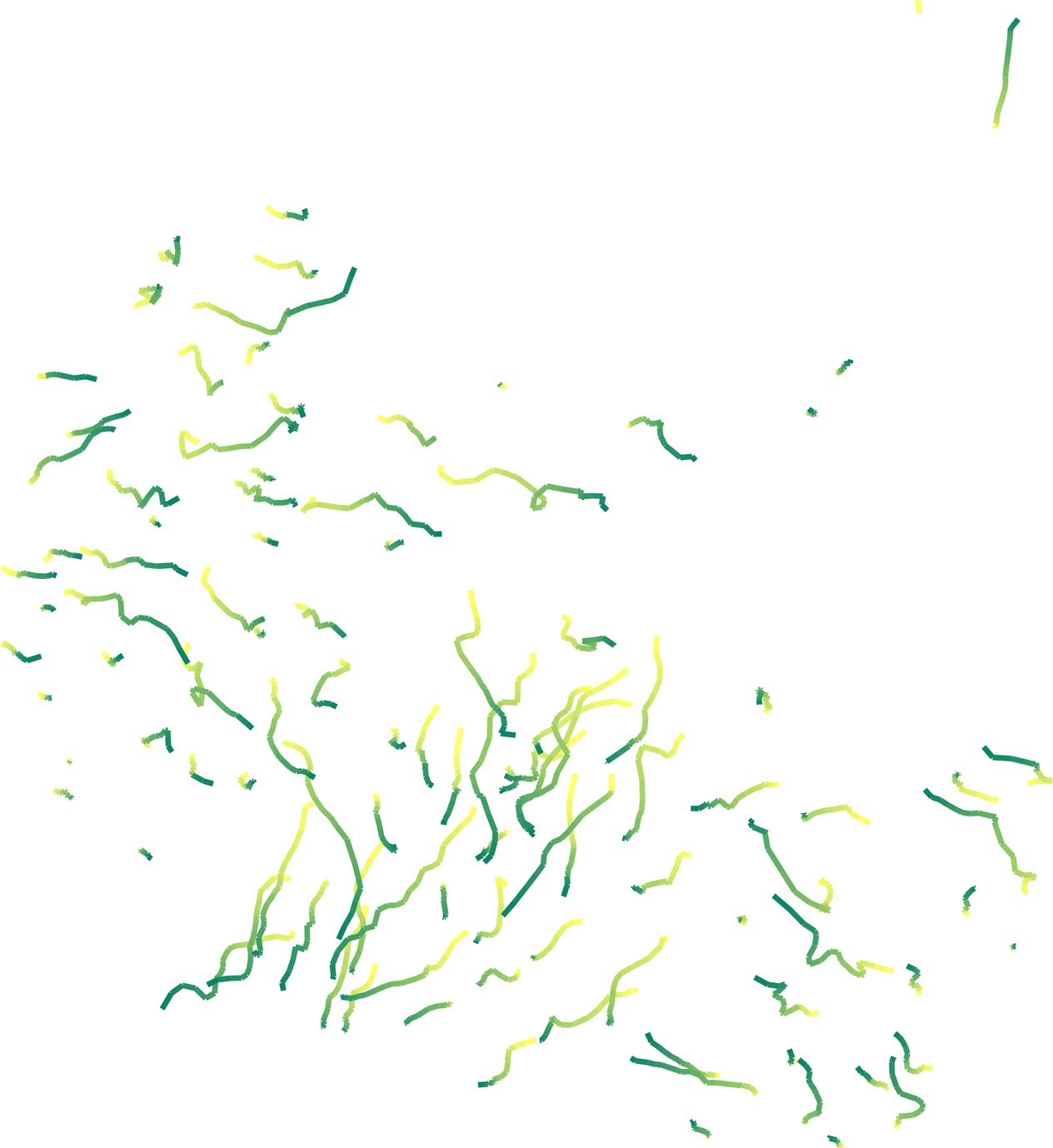}
	\hfill
	\includegraphics[width=0.45\textwidth]{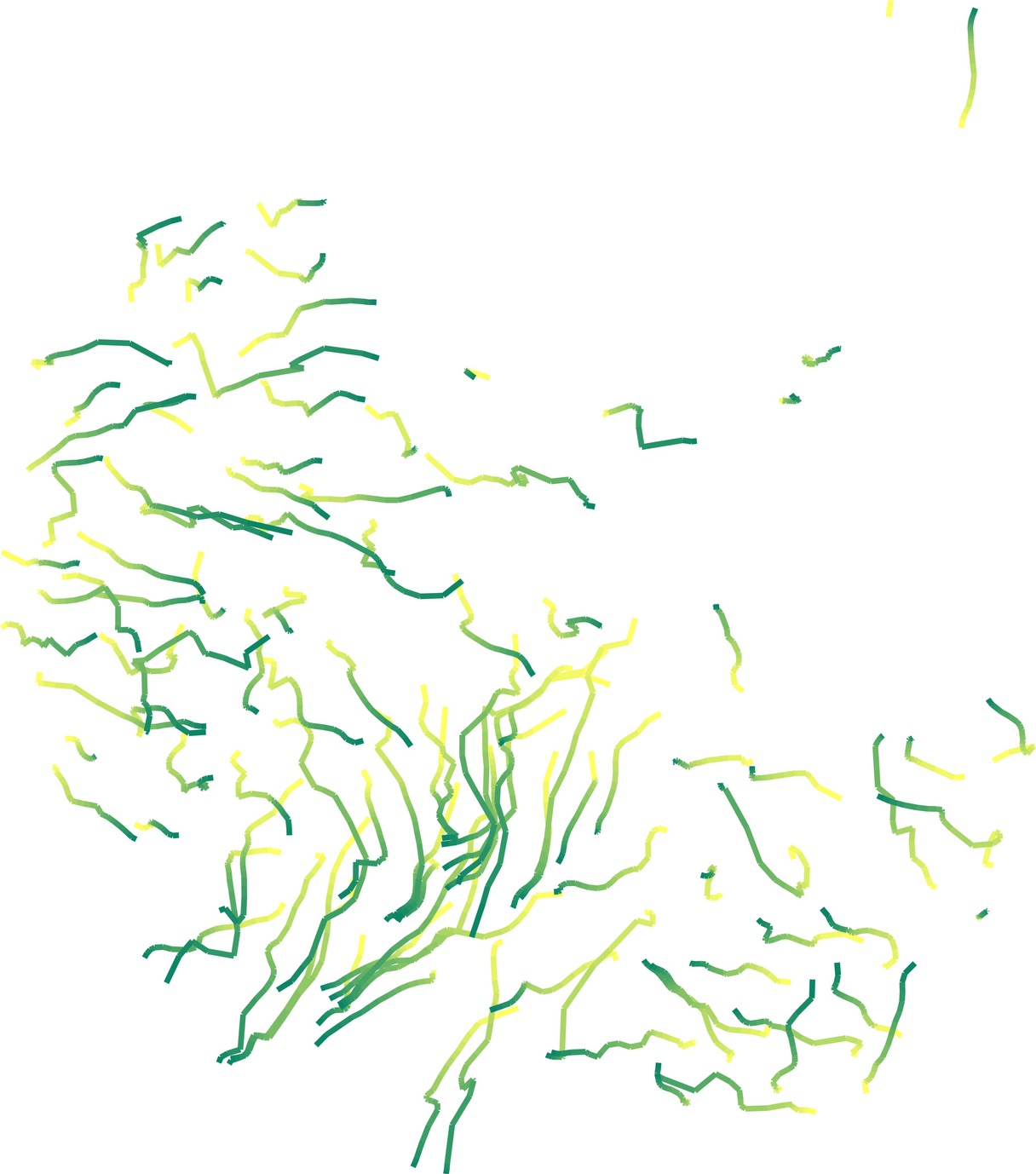}
	\caption{Integral curves for frames $\{40, \dots, 60\}$ for the identical regularisation parameters as in Fig.~\ref{fig:reg}. The colour gradient of a trajectory from yellow to green (bright to dark) indicates temporal progress. Local intensity maxima at the first frame serve as initial values. The embryo's body axis runs from bottom left to top right.}
	\label{fig:integralcurves}
\end{figure*}

\paragraph{Cell Divisions.} Figure~\ref{fig:division} shows two cell divisions in more detail. The displacement field clearly resembles the splitting of the mother cell and the diverging daughter cells. Our results suggest that cell divisions can be indicated reasonably well by the proposed model.

\begin{figure*}[!t]
	\centering
	\includegraphics[width=0.32\textwidth]{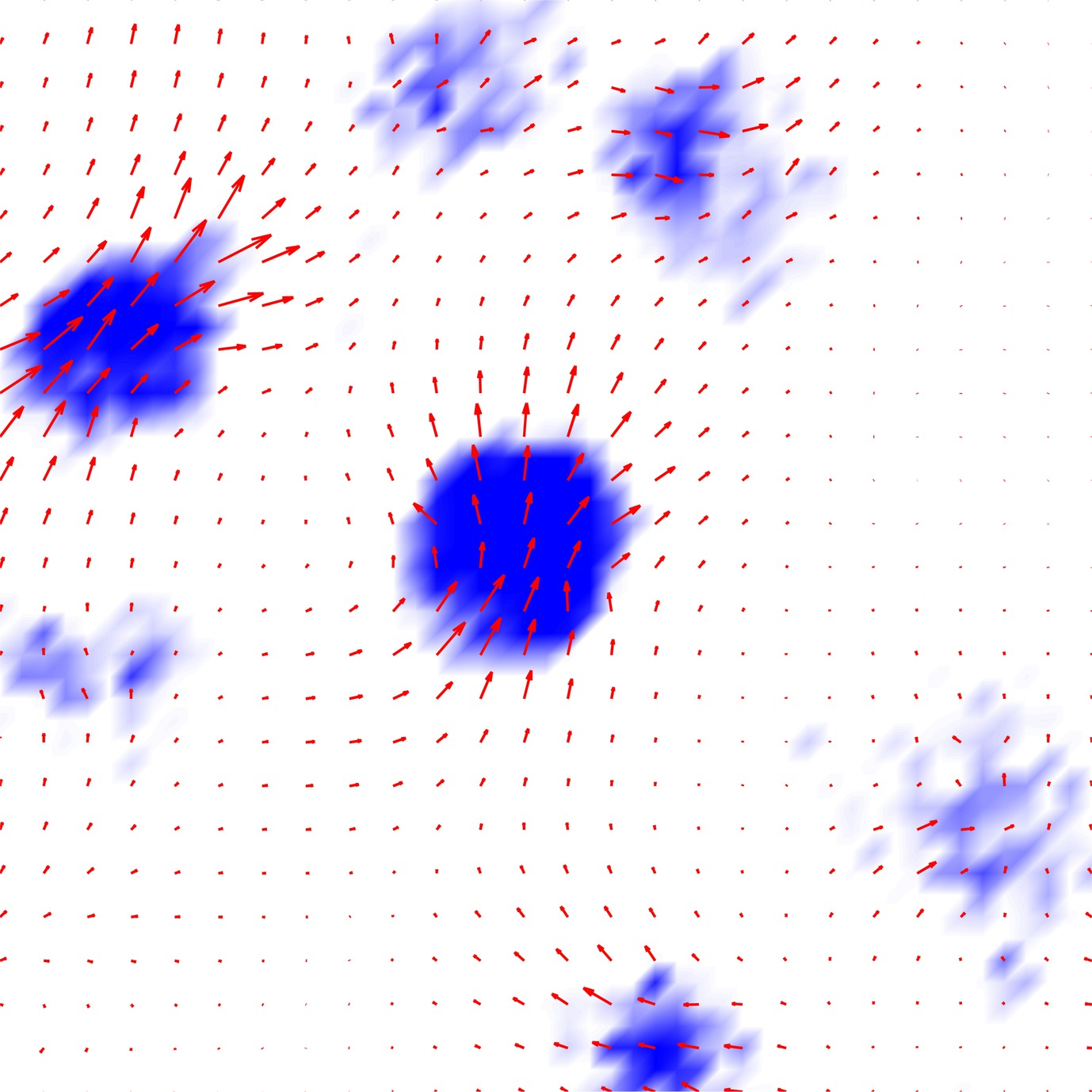}
	\hfill
	\includegraphics[width=0.32\textwidth]{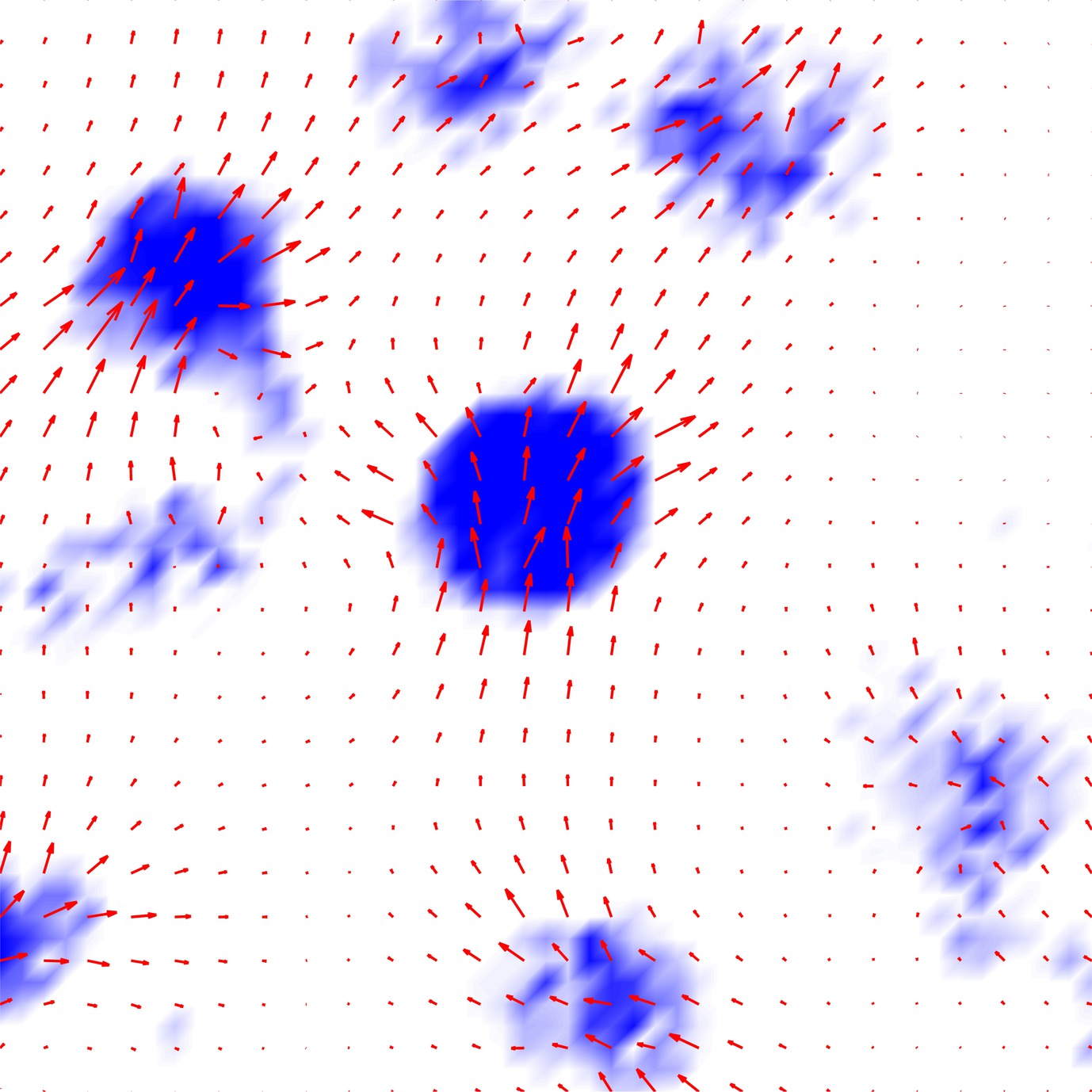}
	\hfill
	\includegraphics[width=0.32\textwidth]{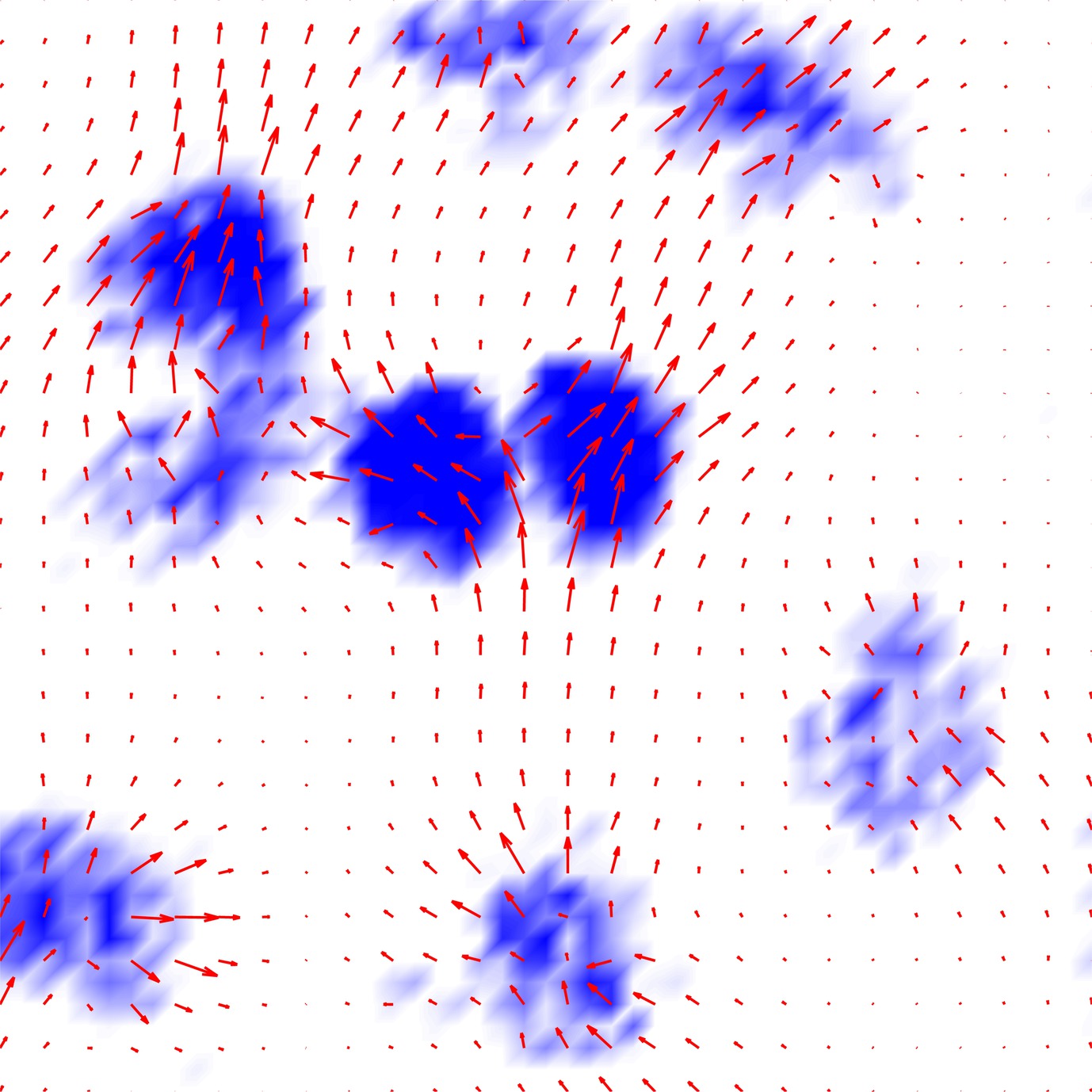}
	\\
	\medskip
	\includegraphics[width=0.32\textwidth]{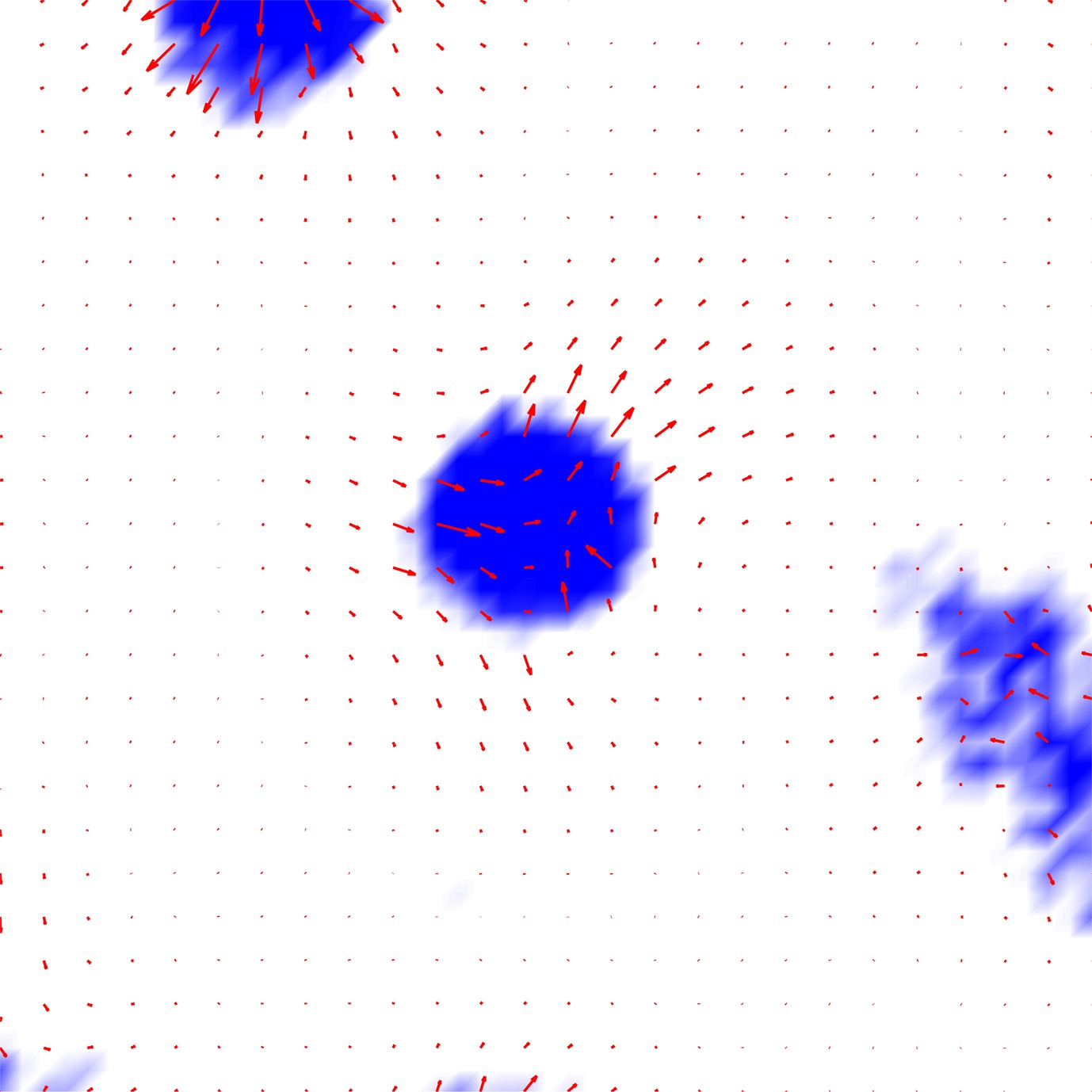}
	\hfill
	\includegraphics[width=0.32\textwidth]{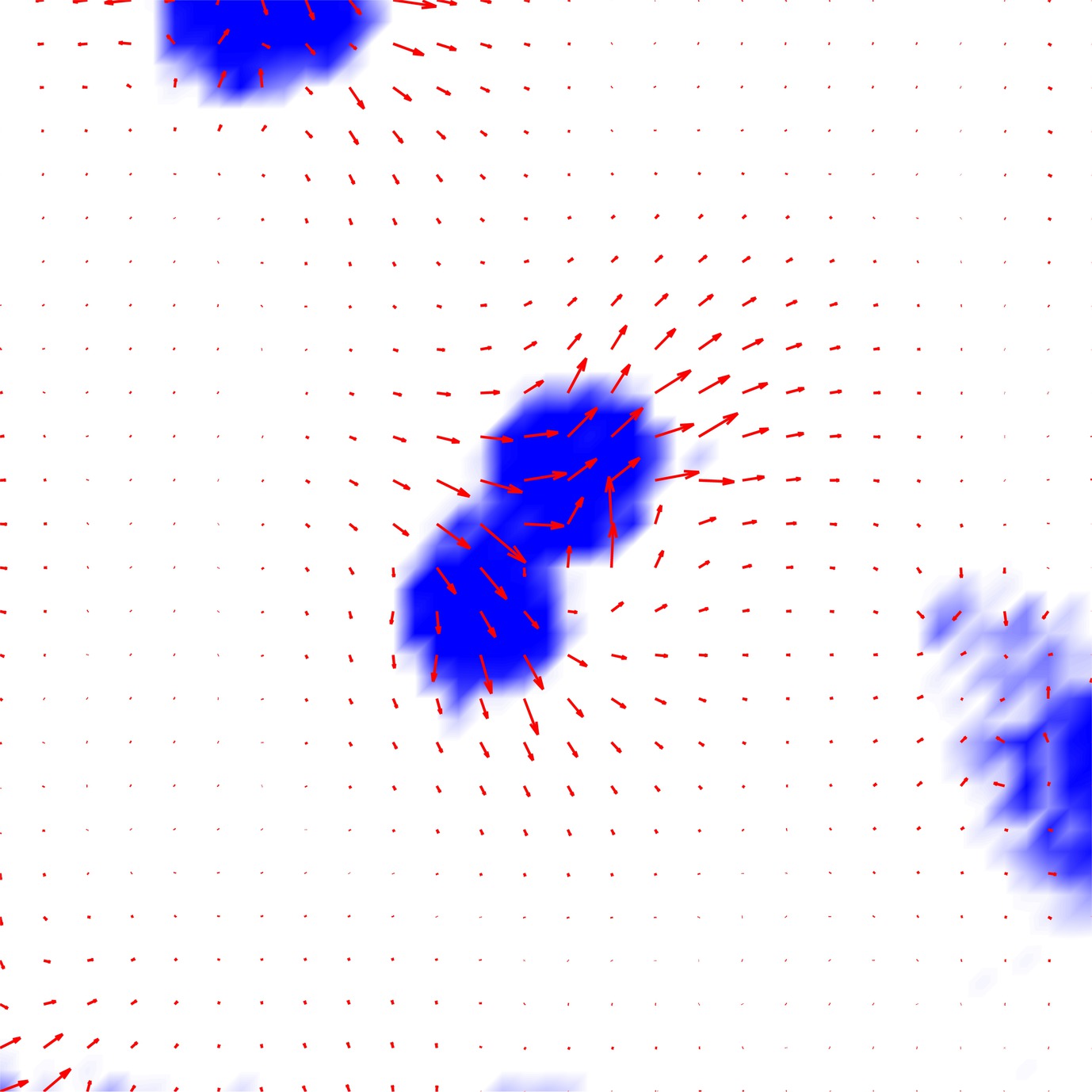}
	\hfill
	\includegraphics[width=0.32\textwidth]{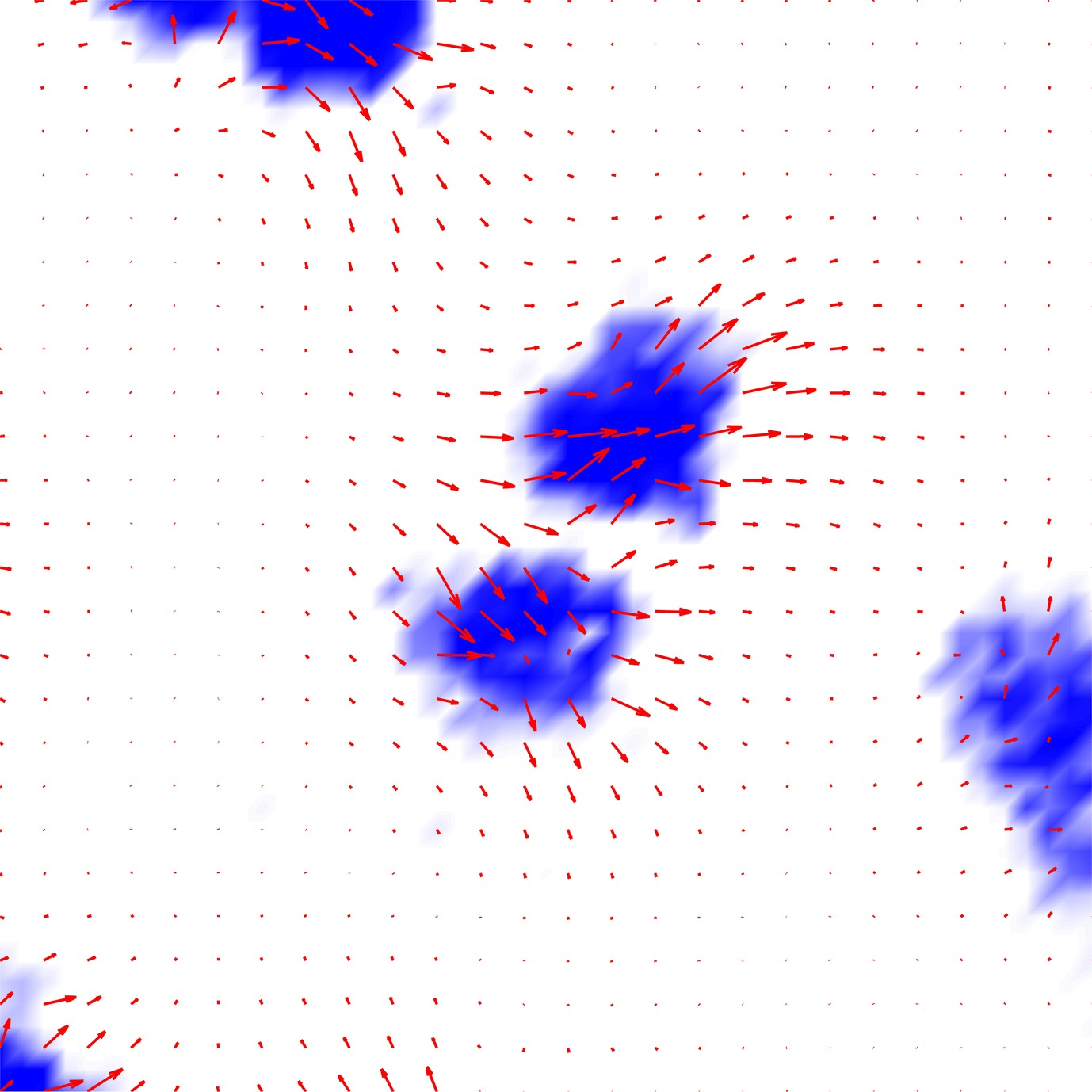}
	\caption{Detailed view of two cell divisions occurring between frames 41 and 44 (top, left to right) and frames 55 and 58 (bottom, left to right). Parameters are $\lambda_{0} = c/100$ and $\lambda_{1} = \lambda_{2} = c/10$. Vectors are scaled and only every second vector is shown. Data intensities are interpolated for smooth illustration.}
	\label{fig:division}
\end{figure*}
\section{Conclusion}
Aiming at an accurate and efficient motion analysis of 4D cellular microscopy data, we generalised both the Horn-Schunck and Weickert-Schn\"orr functionals to images defined on evolving surfaces. The resulting optical flow constraint was solved by means of quadratic regularisation and verified on the basis of real data. Our experimental results suggest that cell movements including divisions are well captured by our model.

\paragraph{Acknowledgements.}
We thank Pia Aanstad from the University of Innsbruck for sharing her biological insight and for kindly providing the microscopy data. This work has been supported by the Vienna Graduate School in Computational Science (IK I059-N) funded by the University of Vienna. In addition, we acknowledge the support by the Austrian Science Fund (FWF) within the national research networks ``Photoacoustic Imaging in Biology and Medicine" (project S10505-N20, Reconstruction Algorithms for PAI) and ``Geometry + Simulation" (project S11704, Variational Methods for Imaging on Manifolds).

\appendix
\section*{Appendix} \label{sec:append}
We first sketch a proof about the statement from Sec.~\ref{sec:evolsurf} that the normal velocity of an evolving surface is independent of $\phi$.
\begin{proposition} \label{thm:normvelo}
	Let $\phi$ be a Langrangian specification of $\mathcal{M}$ and $\vec{V}$ the corresponding velocity as defined in \eqref{eq:velocity}. Then $\vec{V} \cdot \vec{N}$ is independent of the chosen specification.
\end{proposition}
\begin{proof}
We can represent $\bar{\mathcal{M}}$ locally as the level set of a real-valued function $G(t,x)$, whose gradient does not vanish, see e.g.~\cite[Prop.~5.16]{Lee13}. We now express $\vec{V} \cdot \vec{N}$ solely in terms of $G$ and thereby prove the assertion. Observing that the composition of $G$ with $\phi$ is constant, we calculate
\begin{equation*}
	0	= \frac{\mathrm{d}}{\mathrm{d} t} G(t,\phi(t,x_0)) = \partial_t G + \nabla_{\mathbb{R}^3} G \cdot \vec{V} = \partial_t G + \abs{\nabla_{\mathbb{R}^3} G} \vec{V} \cdot \vec{N}.
\end{equation*}
The second equality holds, because $\nabla_{\mathbb{R}^3} G$ is normal to the surface. We conclude that
\begin{equation*}
	\vec{V} \cdot \vec{N} = - \frac{\partial_t G}{\abs{\nabla_{\mathbb{R}^3} G}}.
\end{equation*}
\end{proof}
In other words, different specifications of a surface can only differ in their respective tangential velocities.

Next we prove the transformation law \eqref{eq:transymb}, \eqref{eq:transymb0} for the connection coefficients $\tilde \Gamma^j_{\mu j}$.
\begin{lemma} \label{thm:transymb}
	The symbols defined by \eqref{eq:symbols} are given by \eqref{eq:transymb}.
\end{lemma}
\begin{proof}
	Take inner products on both sides of \eqref{eq:transymb} with $\vec{e}_j$ to get
	\begin{equation*}
		\vec{e}_j \cdot \nabla_{\vec{e}_i} \vec{e}_k = \tilde \Gamma^j_{ik}.
	\end{equation*}
	Next express both terms on the left hand side in the coordinate basis by using $\vec{e}_j = \alpha_j^m \partial_m \vec{x}$ and formula \eqref{eq:covariant}. The assertion follows now immediately.
\end{proof}
An analogous calculation yields formula \eqref{eq:transymb0}.

For our implementation the Euler-Lagrange equations \eqref{eq:optsys} are needed in the following form
\begin{equation} \label{eq:optsys2}
	\begin{gathered}
		d^{\nu \sigma} \partial_{\nu\sigma}w^m + c^{\sigma m}_{i} \partial_\sigma w^i + b^m_{i} w^i = a^m , \qquad \text{in } D, \\
		q^{\nu \sigma} \partial_\sigma w^m + p_{i}^{\nu m} w^i = 0, \qquad \text{on } \{\xi^\nu = 0\} \cup \{\xi^\nu = 1\},
	\end{gathered}
\end{equation}
where we assumed $D=(0,1)^3$. As usual the system is to be understood for $m=1,2$ and $\nu=0,1,2$. Below we give the exact coefficients.
\begin{equation} \label{eq:coefficients}
\begin{aligned}
	a^m					&=	-\alpha_m^i \partial_i f \partial_t f\\
	b^m_{i}				&=	\alpha_m^j \alpha_i^k \partial_j f \partial_k f + \textstyle\sum_{\mu} \lambda_\mu \left( \textstyle\sum_{j}\tilde{\Gamma}^j_{\mu m} \tilde{\Gamma}^j_{\mu i}
							- G_\nu \alpha^\nu_\mu \tilde{\Gamma}^m_{\mu i} + \partial_\nu \left( \alpha^\nu_\mu \tilde{\Gamma}^m_{\mu i} \right) \right)\\
	c^{\sigma m}_{i}	&=	\textstyle\sum_{\mu} \lambda_\mu \Big( \alpha^\sigma_\mu \tilde{\Gamma}^i_{\mu m} - \alpha^\sigma_\mu \tilde{\Gamma}^m_{\mu i} -\delta_{im} \left( G_\nu \alpha^\nu_\mu \alpha^\sigma_\mu + \partial_\nu ( \alpha^\nu_\mu \alpha^\sigma_\mu ) \right) \Big)\\
	d^{\nu \sigma}		&=	-\textstyle\sum_{\mu} \lambda_\mu \alpha^\nu_\mu \alpha^\sigma_\mu\\
	p_{i}^{\nu m}		&=	\textstyle\sum_{\mu} \lambda_\mu \alpha^\nu_\mu \tilde{\Gamma}^m_{\mu i} \\
	q^{\nu\sigma}		&=	\textstyle\sum_{\mu} \lambda_\mu \alpha^\nu_\mu \alpha^\sigma_\mu
\end{aligned}
\end{equation}
Here we used the shorthand
\begin{equation*}
	G_\nu = \frac{\partial_\nu \sqrt{ \det g}}{2 \sqrt{ \det g}}.
\end{equation*}
Recall that the functional without time regularisation \eqref{eq:functional2} leads to a sequence of decoupled systems for every instant $t$. Each of those has the form
\begin{equation*}
	\begin{gathered}
		d^{jk} \partial_{jk}w^m + c^{k m}_{i} \partial_k w^i + b^m_{i} w^i = a^m , \qquad \text{in } D,\\
		q^{jk} \partial_k w^m + p_{i}^{ mj} w^i = 0, \qquad \text{on } \{\xi^j = 0\} \cup \{\xi^j = 1\}.
	\end{gathered}
\end{equation*}
Note that, in comparison to system \eqref{eq:optsys2}, we only replaced Greek indices by Latin ones. The coefficients $a,b,c,d,p,q$ of this simpler system can be obtained from the list above by setting $\lambda_0 = 0$.

\bibliographystyle{plain}

\def\cprime{$'$}
  \providecommand{\noopsort}[1]{}\def\ocirc#1{\ifmmode\setbox0=\hbox{$#1$}\dimen0=\ht0
  \advance\dimen0 by1pt\rlap{\hbox to\wd0{\hss\raise\dimen0
  \hbox{\hskip.2em$\scriptscriptstyle\circ$}\hss}}#1\else {\accent"17 #1}\fi}
  \def\cprime{$'$}

\end{document}